\documentclass[10pt]{article}
\usepackage{amsmath}
\usepackage{amsfonts}
\usepackage{amssymb}
\usepackage[english]{babel}
\usepackage{graphicx}
\usepackage{amsthm}
\usepackage{hyperref}
\usepackage{accents}
\usepackage[autostyle, english = american]{csquotes}
\MakeOuterQuote{"}
\numberwithin{equation}{section}
\theoremstyle{plain}
\newtheorem*{theorem*}{Theorem}
\newtheorem*{lemma*}{Lemma}
\newtheorem{theorem}{Theorem}
\newtheorem{lemma}{Lemma}[section]
\newtheorem{corollary}[lemma]{Corollary}

\newtheorem{proposition}[lemma]{Proposition}
\newtheorem{innercustomthm}{Theorem}
\newenvironment{customthm}[1]
{\renewcommand\theinnercustomthm{#1}\innercustomthm}
{\endinnercustomthm}

\theoremstyle{definition}

\newtheorem{remark}[lemma]{Remark}
\newtheorem{example}[lemma]{Example}

\def\o{\omega}

\def\pr{\partial}
\def\n{\nabla}  
\def\V{\Vert}

\usepackage{geometry}
\begin{document}
	\title{Symmetries and Critical Phenomena in Fluids}
	\author{Tarek M. Elgindi\textsuperscript{1} \and In-Jee Jeong\textsuperscript{2}}
	\footnotetext[1]{Department of Mathematics, UC San Diego. E-mail: telgindi@ucsd.edu.}
	\footnotetext[2]{Department of Mathematics, Korea Institute for Advanced Study. E-mail: ijeong@kias.re.kr.}
	\date{\today}
	\maketitle
	\begin{abstract}
		We study $2D$ active scalar systems arising in fluid dynamics in critical spaces in the whole plane. We prove an optimal well-posedness result which allows for the data and solutions to be scale-invariant. These scale-invariant solutions are new and their study seems to have far reaching consequences.  
		
		More specifically,  we first show that the class of bounded vorticities satisfying a discrete rotational symmetry is a global existence and uniqueness class for the $2D$ Euler squation. That is, in the well-known $L^1 \cap L^\infty$ theory of Yudovich, the $L^1$ assumption can be dropped upon having an appropriate symmetry condition. We also show via explicit examples the necessity of discrete symmetry for the uniqueness. This already answers problems raised by Lions \cite{L} and Bendetto, Marchioro, and Pulvirenti \cite{BMP}. Next, we note that merely bounded vorticity allows for one to look at solutions which are invariant under scaling--the class of vorticities which are $0-$homogeneous in space. Such vorticity is shown to satisfy a new $1D$ evolution equation on $\mathbb{S}^1.$ 
		Solutions are also shown to exhibit a number of interesting properties. In particular, using this framework, we construct time quasi-periodic solutions to the $2D$ Euler equation exhibiting pendulum-like behavior. Finally, using the analysis of the $1D$ equation, we exhibit strong solutions to the $2D$ Euler equation with compact support for which angular derivatives grow (almost) quadratically in time (in particular, superlinear) or exponential in time (the latter being in the presence of a boundary).   
		
		A similar study can be done for the surface quasi-geostrophic (SQG) equation. Using the same symmetry condition, we prove local existence and uniqueness of solutions which are merely Lipschitz continuous near the origin---though, without the symmetry, Lipschitz initial data is expected to lose its Lipschitz continuity immediately. Once more, a special class of radially homogeneous solutions is considered and we extract a $1D$ model which bears great resemblance to the so-called De Gregorio model. We then show that finite-time singularity formation for the $1D$ model \emph{implies} finite-time singularity formation in the class of Lipschitz solutions to the SQG equation which are compact support.     
		
		While the study of special infinite energy (i.e. non-decaying) solutions to fluid models is classical, this appears to be the first case where these special solutions can be embedded into a natural existence/uniqueness class for the equation. Moreover, these special solutions approximate finite-energy solutions for long time and have direct bearing on the global regularity problem for finite-energy solutions. 
	\end{abstract}
	\tableofcontents
	\section{Introduction}
	The dynamics of incompressible fluids, while an old and active subject, is still very elusive due to the possibility of rapid small-scale creation as well as the inherent non-locality of incompressible fluid motion. Since the work of Euler, who wrote down the partial differential equations which model the flow of an ideal fluid, a number of works have been devoted to proving existence, uniqueness, and stability results but only recently did authors start to work on the notion of "criticality" in the study of ideal fluids. Critical spaces are those function spaces which are small enough to allow for an existence and uniqueness theory but large enough to allow for the presence of certain fine scale structures (jump discontinuities or corners, for example). While it may seem that solvability in critical spaces is simply a "philosophical" question of how irregular one can take the initial data to be and still solve the given PDE uniquely, working in critical spaces allows one to get to the heart of what's going on in the PDE, as we shall soon see in the examples of the $2D$ Euler equation and the surface quasi-geostrophic equation. In critical spaces, one can actually study the evolution of special structures which may satisfy a much simpler dynamic than the full problem. In particular, since we are working in a "critical space" and since criticality is related to a notion of scaling, the dynamic of these special structures may actually be governed by "scale-invariant" solutions. Furthermore, the dynamic of these scale-invariant solutions, such as whether they collapse in finite time, will turn out to have a direct bearing on regularity questions. This is the basic philosophy of this work and we will illustrate it through two examples: the $2D$ Euler equation and the surface quasi-geostophic equation. In fact, in a companion work, we will also discuss the 3D Euler equation, but we elected to give it its own place due to the importance of the questions associated with the 3D Euler equation. We begin by introducing the $2D$ Euler equation and some of the works which have been done in critical spaces. 
	\subsection{The $2D$ Euler equation}
	Consider the Cauchy problem for the $2D$ Euler equation on the whole plane. That is, given a divergence-free initial data $u_0:\mathbb{R}^2\rightarrow\mathbb{R}^2$, we look for solutions of 
	\begin{equation}\label{eq:Euler}
	\begin{split}
	\left\{ \begin{array}{rl}
	\pr_t u + u\cdot\n u + \n p &=0~,\\
	\n\cdot u & = 0~,\\
	u(0) &=u_0~.
	\end{array} \right.
	\end{split}
	\end{equation}
	Classically, it is well known that if the initial data $u_0$ is regular enough, then one can solve \eqref{eq:Euler} uniquely. This can be done, locally in time, by energy methods. More concretely, if $u_0$ belongs to either $H^s$ (with $s > d/2 + 1$, $d$ is the dimension of the domain) or $C^{k,\alpha}$ (with $k \ge 1$ and $0<\alpha<1$), then, there is a unique solution in the class $C^0([0,T];H^s)$ or $C^0([0,T];C^{k,\alpha})$, for some $T = T(u_0) > 0$. 
	A key step in showing such a well-posedness result is to obtain an appropriate \textit{a priori} estimate, which in this case takes the form \begin{equation*}
	\begin{split}
	\frac{d}{dt} \V  u(t) \V_{H^s} &\le C_s \V \nabla u(t) \V_{L^\infty} \V u(t) \V_{H^s}~, \qquad
	\frac{d}{dt} \V  u(t) \V_{C^{k,\alpha}} \le C_{k,\alpha} \V \nabla u(t) \V_{L^\infty} \V u(t) \V_{C^{k,\alpha}}~, 
	\end{split}
	\end{equation*} applied with the embeddings $\V \n u\V_{L^\infty} \le  C_{k,\alpha}  \V u(t) \V_{C^{k,\alpha}}$ and $\V \n u \V_{L^\infty} \le C_s \V u(t) \V_{H^s}$. One sees from this that the Lipschitz norm of the velocity actually propagates higher regularity in time. On the other hand, the Lipschitz bound allows one to solve uniquely the following ODE system \begin{equation}\label{eq:flow_ODE}
	\begin{split}
	\frac{d}{dt} \Phi(t,x) &= u(t,\Phi(t,x))~,\\
	\Phi(0,x) &= x~,
	\end{split}
	\end{equation} for each $x$ in the domain. Then for each fixed $t$, $\Phi(t,\cdot)$ is a diffeomorphism of the domain, called the flow map. It is useful to define the vorticity of the fluid, $\omega=\nabla\times u$.
	The following logarithmic estimates are available for smooth solutions: \begin{equation*}
	\begin{split}
	\V \n u \V_{L^\infty} \le C_s \V \o \V_{L^\infty} \log(1 +\V  u(t) \V_{H^s} )~, \qquad \V \n u \V_{L^\infty} \le C_{k,\alpha} \V \o \V_{L^\infty} \log(1 +\V  u(t) \V_{C^{k,\alpha}} )~,
	\end{split}
	\end{equation*} and combined with the standard estimates above, one obtains the criteria of Beale, Kato, and Majda \cite{BKM} that as long as we keep $\V \o \V_{L^\infty}$ in control, the solution of \eqref{eq:Euler} remains smooth. This is actually automatic in the case of two dimensions; the vorticity equals a scalar $\pr_{x_1} u_2 - \pr_{x_2} u_1$, and taking the curl of \eqref{eq:Euler}, one obtains the vorticity equation; \begin{equation}\label{eq:Euler_vort}
	\begin{split}
	\pr_t \o + u\cdot\n \o = 0~. 
	\end{split}
	\end{equation} Note that the vorticity is simply being transported by the flow: we have  \begin{equation}\label{eq:flow}
	\begin{split}
	\omega(t,x) = \omega_0( \Phi_t^{-1}(x) )~,
	\end{split}
	\end{equation} and hence the maximum of the vorticity is a conserved quantity. In the case of the whole plane, when $\omega$ has some decay at infinity ($\o \in L^1\cap L^\infty(\mathbb{R}^2)$ would suffice), it is well-known that the following Biot-Savart law \begin{equation}\label{eq:BS}
	\begin{split}
	u(x) = \frac{1}{2\pi} \int_{\mathbb{R}^2}  \frac{(x-y)^\perp}{|x-y|^2} \o(y) dy
	\end{split}
	\end{equation} uniquely recovers the velocity which decays at infinity, and the set of equations \eqref{eq:flow_ODE}, \eqref{eq:flow}, and \eqref{eq:BS} gives a formulation of the $2D$ Euler equation, which does not involve differentiation. Now, a bound on $\|\o\|_{L^\infty}$ only gives that $u$ is quasi-Lipschitz. More precisely, if $\omega\in L^1\cap L^\infty,$ Yudovich \cite{Y1} showed that for all $|x-x'|<\frac{1}{2},$ $$|u(x)-u(x')|\leq C\|\omega\|_{L^1\cap L^\infty}|x-x'| \log\frac{1}{|x-x'|}~,$$ for some universal constant $C>0.$ This allows us to solve the ODE \eqref{eq:flow_ODE} uniquely. These considerations lead to an existence and uniqueness theory for $L^1\cap L^\infty$ vorticities which was first established by Yudovich in 1963 \cite{Y1}. This extension of the well-posedness theory over the classical results is a significant one; for instance, the class $L^1\cap L^\infty$ contains patch-type vorticity (e.g. the characteristic function of a measurable set), which models certain physical situations. However, the fact that the $L^\infty$ bound on $\omega$ does not necessarily lead to Lipschitz control on the velocity field can be problematic since a non-Lipschitz velocity field can lead to a flow map $\Phi(t,\cdot)$ which loses its regularity in time. Since $\omega=\omega_0\circ\Phi^{-1},$  a bound on the Lipschitz norm of $u$ is crucial to propagate fine-scale structures which may be present in the initial data. Unfortunately, as has been established in the works of the first author and Masmoudi \cite{EM1} as well as in the work of Bourgain and Li \cite{BL2}, propagating a bound on the Lipschitz norm of $u$ is, in general, not possible if $u_0$ is only assumed to be Lipschitz continuous. In order to actually propagate a Lipschitz bound on $u$, the initial data must be taken to be smoother than just Lipschitz. As noted above, $C^{1,\alpha}$ or $H^{s}$ with $s>2$ would do, but one can also propagate Lipschitz bounds on $u$ using anisotropic regularity such as is the case with smooth vortex patches (see \cite{C}). Propagating the Lipschitz bound on $u$ for solutions which have corner like discontinuities was left open in previous works and will be returned to later on in this work (see the subsection on the propagation of angular regularity). 
	% Serfati solutions
	\subsection{Solutions with non-decaying vorticity and the symmetry condition}
	
	\subsubsection*{Serfati Solutions}
	
	The usual assumption that $\omega \in L^1 \cap L^\infty$ may be argued physically unsatisfactory, since the vorticity and the velocity have to decay at infinity. Indeed, the well-posedness of the Euler equations with merely bounded velocity and vorticity has been achieved in the works of Serfati \cite{S1,S2}, and they have been further expanded and generalized in  \cite{AKLN,K1,Ta,TTY}. The main result states that (see \cite{K1} for details) given a pair of bounded functions $(u_0,\o_0)$ satisfying $\n \times u_0 = \o_0$ and $\nabla \cdot u_0 = 0$, and an arbitrary continuous function $U_\infty(t): \mathbb{R}^2 \rightarrow \mathbb{R}^2 $ satisfying $U_\infty(0) = 0$, there is a unique global-in-time solution to the $2D$ Euler equation where for each time, $u_t, \o_t$ are bounded functions and satisfies the ``renormalized Biot-Savart law'' \begin{equation}\label{eq:rBS}
	\begin{split}
	u_t(x) - u_0(x) = U_\infty(t) + \lim_{R \rightarrow \infty} (a_R K) * (\omega_t - \omega_0) (x) ~,
	\end{split}
	\end{equation} where $K(\cdot)$ denotes the Biot-Savart kernel in \eqref{eq:BS} and $a_R$ is some cutoff whose support increases to infinity with $R$. Here, the ``behavior at infinity'' $U_\infty(t)$ can be removed with the following change of variables \begin{equation*}
	\begin{split}
	\bar{u} (t,x ) := u(t, x + \int_0^t U_\infty(s)ds ) - U_\infty(t)~,\quad \bar{p} (t,x):= p(t, x + \int_0^t U_\infty(s)ds ) + U'_\infty(t) \cdot x~. 
	\end{split}
	\end{equation*} % Such a non-uniqueness issue arises here, only because we do not have decay of the velocity at infinity. 
	With $\o $ just in $L^\infty$, the Biot-Savart law clearly does not converge. The key observation of Serfati was that, using the Euler equations and integration by parts, for smooth solutions one has the following Serfati identity: \begin{equation}\label{eq:Serfati_identity}
	\begin{split}
	u_t - u_0 &= U_\infty(t) + (aK)*(\o_t - \o_0) + ( (1-a)K ) * (\o_t - \o_0) \\  &= U_\infty(t) +  (aK)*(\o_t - \o_0)  - \int_0^t \left( \n \n^\perp [(1-a)K] \right) * (u\otimes u) (s) ds~,
	\end{split}
	\end{equation} with $a$ being some compactly supported cut-off function, and surprisingly the last expression makes sense with $u \in L^\infty$, since $\n \n^\perp [(1-a)K]$ has decay $|x|^{-3}$, which is integrable. The identities \eqref{eq:rBS} and \eqref{eq:Serfati_identity} are indeed equivalent, see \cite{K1}. 
	
	\subsubsection*{The role of symmetry}
	
	In this paper, we keep the assumption $\omega \in  L^\infty $ but replace the assumption $ u \in L^\infty$ with $\omega$ being $m$-fold rotationally symmetric around the origin, for some integer $m \ge 3$. Our first main result shows that one can uniquely solve the $2D$ Euler equation in this symmetry class: \begin{customthm}{A}[Existence and Uniqueness for bounded and $m$-fold symmetric vorticity]\label{MainThm1}
		Assume that $\omega_0$ is a bounded and $m$-fold rotationally symmetric function on the plane, for some $m \ge 3$. Then, there is a unique global solution to the $2D$ Euler equation with $\omega \in L^\infty([0,\infty); L^\infty(\mathbb{R}^2) )$ and $m$-fold rotationally symmetric. 
	\end{customthm}
	\begin{remark}
		Part of the statement is that there is a well-defined velocity field associated with such non-decaying vorticity. This is made precise in the statement of the Theorem \ref{thm:main} below. 
	\end{remark}

	Somewhat analogously to the Serfati case, the key fact we utilize is that under the symmetry assumption, the Biot-Savart kernel actually gains integrable decay at infinity (this fact was obtained in a very recent work of the first author \cite{E1}), which in particular enables us to recover the velocity from the vorticity. Indeed, using the symmetry of $\omega$, one may rewrite \begin{equation*}
	\begin{split}
	u(x) = K* \o (x ) = \frac{1}{m} \sum_{i=1}^m  \int_{\mathbb{R}^2}  K(x- O^i_{2\pi/m} y) \omega(y)   dy~,
	\end{split}
	\end{equation*} with $O_{2\pi/m}$ being the counterclockwise rotation matrix by the angle $2\pi/m$, and the point is that \begin{equation*}
	\begin{split}
	\sum_{i=1}^m K(x- O^i_{2\pi/m} y) \approx c \frac{|x|^{m-1}}{|y|^m}
	\end{split}
	\end{equation*} in the regime $|y| \ge C|x|$. This is integrable for $m \ge 3$, and barely fails to be so for $m = 2$. 
	In this situation, the origin is a fixed point for all time, and the velocity and stream function have bounds \begin{equation}\label{eq:Lipschitz_bound}
	\begin{split}
	|u(x)| \le C \V \omega\V_{L^\infty} |x|~, \qquad |\Psi(x)| \le C\V \omega\V_{L^\infty} |x|^2~,
	\end{split}
	\end{equation}
	which are natural in view of physical dimensions. This in particular removes the non-uniqueness issue arising from $U_\infty(t)$ in the Serfati case, and more importantly, it says that the logarithmic correction $\ln |x|$ (which we usually expect) is absent at the origin, and this plays a crucial role in the proof of Theorem \ref{MainThm1}. 
%	In fact, as we will show, the symmetry condition actually allows us to propagate global Lipschitz bounds on $u$ even though $\omega$ may have a "corner" jump discontinuity at the origin which, without the symmetry, was shown to lead to unbounded Lipschitz norm in \cite{EM1}. Moreover, the symmetry condition gives that for $\omega_0$ which is Lipschitz, its gradient can grow at most exponentially at the origin. This type of result was recently established by Itoh, Miura, and Yoneda \cite{IMY1,IMY2}, which we discuss after giving the proof of our main result. 
	
\subsubsection*{An Example Illustrating the Importance of Symmetry for Uniqueness}

\begin{proposition}
Let $\omega_0= \chi_{y\geq 0}.$ Then, there exist two sequences of compactly supported initial data $\{\omega_0^{i,n}\}_{n \ge 1}$ with  $\omega^{i,n}_0\rightarrow \omega_0$  in $L^2_{loc}$ for $i=1,2$, so that the corresponding unique Yudovich solutions converge in $L^2_{loc}$ to \emph{different} functions at $t=1$. In fact, there are infinitely many solutions to the $2D$ Euler equation which are initially $\omega_0$ which can be realized as a limit of compactly supported smooth solutions. 
\end{proposition}
\begin{remark}
This cannot happen in the presence of $m-$fold symmetric with $m\geq 3$ and when the approximating sequences respect this symmetry (for example, under radial cut-off).
\end{remark}

\begin{proof}
First, let us take $\omega^{1,n}_0=\chi_{A_n}$ with $$A_n=\left\{\frac{x^2}{n^2}+\frac{(y-n)^2}{n^2}<1 \right\},$$ which is a disk of radius $n$ centered at $(0,n)$. It is easy to see that $\chi_{A_n}$ converges point-wise everywhere to $\chi_{y>0}$ (and hence locally in $L^2$). Moreover, $\omega^{1,n}(t)=\chi_{A_n}$ for all $t>0$ since this is actually a stationary Yudovich solution to the $2D$ Euler equation. Thus, in this sense, $\chi_{y > 0}$ can be seen to be stationary (in the weak sense). Now let us consider $\omega^{2,n}_0=\chi_{B_n}$ with $$B_n=\left\{\frac{x^2}{4n^2}+\frac{(y-n)^2}{n^2}<1\right\},$$ which is just an ellipse with horizontal axis $2n$ and vertical axis $n$ centered at $(0,n)$. As above, it is clear that $\chi_{B_n}$ converges point-wise everywhere to $\chi_{y>0}$. However, this time, $\omega^{2,n}(t)$ is not stationary but a rotating Kirchoff vortex which is rotating at the angular speed $\frac{2n^2}{(2n+n)^2}=\frac{2}{9}$ which is independent of $n$.  This means that at the time $t=\frac{9\pi}{4}$, we have  $\omega^{2,n}(t)=\chi_{C_n}$ with $$C_n=\left\{\frac{(y-n)^2}{4n^2}+\frac{x^2}{n^2}<1\right\},$$ and we see that $\omega^{2,n}(t=\frac{9\pi}{4})\rightarrow \chi_{ \mathbb{R}^2}$ everywhere as $n\rightarrow 0$ and we are done.  
\end{proof}
	
	\subsubsection*{Previous results}
	
	Yudovich's $L^1\cap L^\infty$ result was improved mainly in two directions, one which weakens the $L^1$-assumption\footnote{It is straightforward to see that the Yudovich theorem holds for vorticity in $L^p \cap L^\infty$ for any $p < 2$.}, and the other which allows the vorticity to be (slightly) unbounded. Regarding the latter, we just refer the interested reader to works \cite{BH, BK, Y2, E2, Vi1, MR}. In the other direction, Benedetto, Marchioro, and Pulvirenti have shown in \cite{BMP} that if the initial data $(u_0,\omega_0)$ satisfy $\omega_0 \in L^p \cap L^\infty(\mathbb{R}^2)$ and $|u_0(x)| \le C(1 + |x|^\alpha)$ for some $\alpha < 1$ with $\alpha p < 2$, there is a unique solution to the $2D$ Euler equation. Note that by imposing some restriction on the growth of the velocity at infinity, one can relax $L^1$ up to any $L^p$ with $p < \infty$, as $\alpha \rightarrow 0^+$. The authors also ask what happens for just $L^\infty$ vorticity. Comparing this to Theorem \ref{thm:main} below, we can treat velocities growing linearly in space, at the cost of imposing $m$-fold symmetry for $m \ge 3$. Indeed, by imposing just $2$-fold symmetry, one can obtain existence and uniqueness for $L^{p} \cap L^\infty$-vorticity (for any $p < \infty$), without restricting the growth of velocity. This can be done using the methods of this paper and will be discussed in detail somewhere else (but see recent \cite{CK}). 
	
	\subsection{Local well-posedness in a critical space and bounds on the velocity gradient}
	A natural question is whether one can actually prove bounds on the Lipschitz norm of $u$ even if $\omega$ has a non-smooth jump discontinuity at the origin. In fact, one of the basic themes of this paper is to study the question of whether $L^\infty$ estimates can be established for singular integrals without extra regularity assumptions in the sense of scaling using only symmetry conditions. To this end, we define a scale of spaces $\mathring{C}^{0,\alpha}$ by: 
	$$\|f\|_{\mathring{C}^{0,\alpha}}:=\|f\|_{L^\infty} + \| |\cdot|^\alpha f\|_{{C}^\alpha_*}~.$$
	These spaces have the same scaling as $L^\infty$ but, under a symmetry condition, we prove boundedness of the singular integrals arising from the operator which sends $\omega$ to $\nabla u$ via the Biot-Savart law and more general problems: 
	\begin{customthm}{B} \label{MainThm2}
		Let $f\in \mathring{C}^{0,\alpha}(\mathbb{R}^2)$ be $m-$fold symmetric for some $m\geq 3$. Then, $\nabla^2(-\Delta)^{-1}f\in\mathring{C}^{0,\alpha}(\mathbb{R}^2).$
	\end{customthm} 
	\begin{remark} Regarding the statement of the above theorem, we remark that: 
		\begin{itemize}
			\item We are not assuming that $f$ has compact support or has any decay at infinity. Therefore, it is already non-trivial to define the operator $(-\Delta)^{-1}$ for such functions. Essentially, it is not possible to do so for general functions which are not symmetric at all, or just 2-fold symmetric around the origin.
			\item There exists $f\in \mathring{C}^{0,\alpha}(\mathbb{R}^2)$, compactly supported, and 2-fold symmetric for which $\nabla^2(-\Delta)^{-1}f\not\in L^\infty(B_1(0)).$
		\end{itemize} 
	\end{remark}
	These bounds are crucial to prove existence and uniqueness for the SQG equation in a class of merely Lipschitz continuous initial data (see Theorem \ref{MainThm5} below) and allow us to propagate Lipschitz bounds on the velocity field in the $2D$ Euler equation even when the initial vorticity has a jump discontinuity. In fact, we  prove the following theorem: \begin{customthm}{C} \label{MainThm3}
		Let $\omega_0\in \mathring{C}^{0,\alpha}(\mathbb{R}^2)$ be $m-$fold symmetric for some $m\geq 3.$ Then the solution to the $2D$ Euler equation with initial data $\omega_0$ satisfies \begin{equation*}
		\begin{split}
		\V \omega(t)\V_{\mathring{C}^{0,\alpha}} \le C\exp(C\exp(Ct)),
		\end{split}
		\end{equation*} as well as \begin{equation*}
		\begin{split}
		\V \nabla u(t)\V_{L^\infty} \le C\exp(Ct),
		\end{split}
		\end{equation*} where $C > 0$ depends only on $\V \omega_0\V_{\mathring{C}^{0,\alpha}}$.
	\end{customthm}

	For future use, let us record an equivalent definition of the $\mathring{C}^{0,\alpha}$-norms: \begin{lemma}\label{lem:equiv}
		For any $0 < \alpha \le 1$ and $f \in \mathring{C}^{0,\alpha}(\mathbb{R}^2)$, we have \begin{equation}\label{eq:equiv}
		\begin{split}
		\frac{1}{2}\V f\V_{\mathring{C}^{0,\alpha}} \le \V  f\V_{L^\infty} + \sup_{x \ne x'}\left[\min\{ |x|^\alpha, |x'|^\alpha \} \frac{|f(x) - f(x')|}{|x-x'|^\alpha}\right] \le 2 \V f \V_{\mathring{C}^{0,\alpha}} 
		\end{split}
		\end{equation} as well as \begin{equation}\label{eq:equiv2}
		\begin{split}
		\frac{1}{2}\V f\V_{\mathring{C}^{0,\alpha}} \le \V  f\V_{L^\infty} + \sup_{x \ne x'}\left[\max\{ |x|^\alpha, |x'|^\alpha \} \frac{|f(x) - f(x')|}{|x-x'|^\alpha} \right]\le 2 \V f \V_{\mathring{C}^{0,\alpha}} .
		\end{split}
		\end{equation}
	\end{lemma}
	
	\begin{proof}
		Take two points $x \ne x'$ in $\mathbb{R}^2$ and without loss of generality assume that $|x| \le |x'|$. Note the following equality: \begin{equation*}
		\begin{split}
		\frac{|x|^\alpha f(x) - |x'|^\alpha f(x')}{|x-x'|^\alpha} = |x|^\alpha \frac{f(x) - f(x')}{|x-x'|^\alpha} + \frac{|x|^\alpha - |x'|^\alpha}{|x-x'|^\alpha} f(x'). 
		\end{split}
		\end{equation*} Noting that \begin{equation*}
		\begin{split}
		\frac{\left||x|^\alpha - |x'|^\alpha\right|}{|x-x'|^\alpha} \le 1 
		\end{split}
		\end{equation*} holds, \eqref{eq:equiv} readily follows. Then \eqref{eq:equiv2} follows simply by switching the roles of $x$ and $x'$ in the above equality. 
	\end{proof}

	In the following we shall use Lemma \ref{lem:equiv} several times, sometimes implicitly. 
	
	\subsection{Scale-invariant solutions for $2D$ Euler}
	This extension to the Yudovich theory contains some interesting classes of solutions. A distinguished class is the case of bounded and radially homogeneous vorticity, i.e. $\omega$ satisfying $\omega(x) = h(x/|x|)$. Homogeneity is propagated by the Euler dynamics, and, by uniqueness, the system reduces to a simple but nontrivial $1D$ equation on the unit circle. In particular, $h$ satisfies the following $1D$ active scalar equation:
	
	$$\partial_t h(\theta,t)+ 2 H(\theta,t)\partial_\theta h(\theta,t)=0~,$$
	$$4H+H''=h~.$$
	This $1D$ system is even more regular than the $2D$ Euler equation, in the sense that the advecting velocity field is two derivatives more regular than the advected quantity. As a result, $\partial_\theta h$ can grow at most exponentially in time, in contrast with the double exponential rate for the case of $2D$ Euler. We show that this exponential rate is sharp when we have a boundary available, and rule it out in the absence of boundaries, under some extra assumptions. These are shown via establishing that there is a trend to equilibrium as time goes to infinity. We note that these solutions, while being infinite energy, can be placed into a natural uniqueness class, which is the class of bounded vorticities satisfying the symmetry assumption. As a consequence, they can also be used to show that solutions with finite energy experience growth of angular derivatives as $t\rightarrow\infty$. In fact, using our analysis of the $1D$ model, we can prove that there exist solutions to the $2D$ Euler equation with Lipschitz velocity field and which are smooth in the angular variable for all times for which the angular derivative of $\omega,$ $\partial_\theta\omega$ experience (almost) quadratic-in-time $L^\infty$ growth. When a boundary is present, this growth can actually be (at least) exponential in time. This follows simply from analyzing the behavior of the $1D$ system. Hence, we emphasize, while the solutions of the $1D$ system are infinite energy, any growth of $\partial_\theta h$ implies the existence of compact support $\omega$ with $\partial_\theta\omega$ growing at least as fast. Formally, we state it as follows: \begin{customthm}{D} \label{MainThm4}
		Fix some $m \ge 3$, and let $\omega^{1D}(t,x) \in L^\infty_{loc}\mathring{C}^{0,\alpha}(\mathbb{R}^2) $ be a $m$-fold symmetric and 0-homogeneous solution to the $2D$ Euler equation. Take any $m$-fold symmetric initial data $\omega_0 \in \mathring{C}^{0,\alpha}(\mathbb{R}^2)$ such that $\omega_0 - \omega^{1D}(0,\cdot)$ belongs to $C^\alpha(\mathbb{R}^2)$ and vanishes at the origin. Then, for all $t \in [0,\infty)$, we have \begin{equation*}
		\begin{split}
		\V \omega(t)\V_{\mathring{C}^{0,\alpha}} \ge \V \omega^{1D}(t)\V_{\mathring{C}^{0,\alpha}},
		\end{split}
		\end{equation*} where $\omega(t)$ is the unique solution associated with $\omega_0$. 
	\end{customthm} The above theorem gives compactly supported initial vorticity whose angular derivatives grow (almost) quadratically in time, and exponentially in time in the presence of a boundary.  

	\subsubsection{Previous examples of infinite energy solutions}
	The consideration of radially homogenous vorticity is comparable to the well-known \textit{stagnation-point similitude} ansatz, which in the case of 2D takes the form \begin{equation}\label{eq:stagnation-point_similitude}
	\begin{split}
	u(t,x,y) = (f(t,x),-y \pr_x f(t,x) )~
	\end{split}
	\end{equation} (note that $u$ satisfies the divergence free condition), and when inserted in the $2D$ Euler equation, one obtains the so-called Proudman-Johnson equation \cite{PJ} \begin{equation}\label{eq:PJ}
	\begin{split}
	\pr_t \pr_x^2 f = \pr_x f \cdot \pr_x^2 f - f \pr_x^3 f~. 
	\end{split}
	\end{equation} To the best of our knowledge, in the context of the Euler equations, \eqref{eq:PJ} was first studied by Stuart \cite{St} who showed formation of singularities in finite-time. This is one of the unfortunate sides of the ansatz \eqref{eq:stagnation-point_similitude} because clearly the fact that these solutions form singularities in finite time has no bearing on singularity formation in the $2D$ Euler equation (since singularities \emph{cannot} form in finite time for the $2D$ Euler equation). In this sense, the singularity is "coming from infinity" and is a consequence of the solution being of infinite energy. 
	
	The ansatz \eqref{eq:stagnation-point_similitude} can be inserted in many other nonlinear evolution equations to define a $1D$ system. For an example, inserting it to the SQG (surface quasi-geostrophic) equation \cite{CaCo}, one obtains a particular case of the De Gregorio model (discussed in Section \ref{sec:SQG}). See also a recent work of Sarria and Wu \cite{SaWu} where they study the $2D$ Boussinesq model with \eqref{eq:stagnation-point_similitude}. 
	
	In higher dimensions, one can similarly put \begin{equation*}
	\begin{split}
	u(t,x',x_n) = ( f(t,x'), - x_n \n \cdot f(t,x') )~, \qquad x' = (x_1,\cdots,x_{n-1})~. 
	\end{split}
	\end{equation*} This was suggested in Gibbon, Fokas, and Doering \cite{GFD} in the context of the 3D Euler equations, and was shown to blow-up in finite time in the papers of Childress, Ierley, Spiegel, and Young \cite{CISY} and Constantin \cite{Con}. Notice that in all these cases, the vorticity is never a bounded function; indeed, from \eqref{eq:stagnation-point_similitude}, one sees that $\o(x,y) = -y f''(x)$ grows linearly in space. 
	Our radial homogeneity ansatz is special in this regard, and it seems that having bounded vorticity is the key to having \textit{uniqueness}.

	\subsection{The SQG equation}
	Above, we discussed how, while working in the critical spaces $\mathring{C}^{0,\alpha}$ and under a symmetry condition, one can derive information about finite-energy solutions from information about the dynamic of scale-invariant solutions. We now explain how this can also be done for the SQG equation.  
	Recall the $2D$ surface quasi-geostrophic equation:
	\begin{equation*}
	\begin{split}
	&\partial_t\Theta+u\cdot\nabla\Theta=0~, \\
	&u=\nabla^\perp (-\Delta)^{-\frac{1}{2}}\Theta~.
	\end{split}
	\end{equation*}
	This system bears great resemblance to the $2D$ Euler equation except that the Biot-Savart law (the relation between $\Theta$ and $u$) is more singular than in the Euler case. While in the $2D$ Euler equation, the vorticity is one derivative better than $\omega,$ here $\Theta$ and $u$ are at the same level of regularity. For this reason, the global regularity problem for the SQG equation is still open. Indeed, to prove regularity for the SQG equation, one would basically need an $L^\infty$ bound on $\nabla u$, just as in the Euler equation, but this time bounds on the $L^\infty$ norm of $\nabla u$ are very far from the obvious conservation laws in the equation--$L^p$ conservation of $\Theta$ which only leads to $L^p$ bounds on $u.$ As a way to approach this problem, one may first try to prove local well-posedness for the SQG equation in the critical space of Lipschitz $\Theta$. Unfortunately, this seems to be impossible to do directly due to the presence of singular integrals in the map from $\nabla \Theta$ to $\nabla u$. One can actually show that the SQG at least\footnote{Though, one can imagine that a little work using the framework of \cite{EM1} will yield strong ill-posedness of the SQG in the Lipschitz class.} "mildly ill-posed" in the Lipschitz class in the sense given in \cite{EM1}. However, using our intuition from the Euler case, one could hope to establish local well-posedness in a critical space which includes "scale-invariant" solutions, then derive the equation for scale-invariant solutions and study its properties, and finally deduce some growth mechanism, like finite-time singularity formation for the SQG system with finite energy. What we will show here is that local well-posedness can indeed be established and a $1D$ model is derived we then prove a conditional result: if singularities form in finite time for the $1D$ model, then there exists data with compact support in the local well-posedness class (a subset of the class of Lipschitz functions) whose unique solution breaks down in finite time. An interesting open problem is now to establish whether the $1D$ model breaks down in finite time. The following theorem sums up the discussion above: 
	\begin{customthm}{E}\label{MainThm5}
		For $0<\alpha<1$, let $X^\alpha$ be the class of functions $f:\mathbb{R}^2\rightarrow \mathbb{R}$ for which:
		\begin{enumerate}
			\item $\nabla f\in \mathring{C}^{0,\alpha}(\mathbb{R}^2)$,
			\item $f$ is $m-$fold rotationally symmetric and odd with respect to an axis, for some $m \ge 3$.
		\end{enumerate}
		Then, for every $\Theta_0\in X^\alpha,$ there exists $T= T(\| \Theta_0\|_{\mathring{C}^{0,\alpha}}) > 0$ and a unique solution to the inviscid SQG equation  $\Theta\in C([0,T^*); X^\alpha)$ and $\Theta(0,\cdot)=\Theta_0.$ 
		
		Moreover, there is a one-dimensional evolution equation defined on the unit circle describing radially homogeneous solutions to the $2D$ SQG equation. Singularity formation for smooth enough solutions ($C^{2,\alpha}(S^1)$ is sufficient) to this $1D$ equation implies the existence of $\Theta_0\in X^{\alpha}$ with compact support whose solution of the SQG equation $\Theta(t)$ blows up in finite time --  $\limsup_{t \rightarrow T^*} \V \nabla\Theta(t)\V_{L^\infty} = + \infty$ for some finite $T^* > 0$. 
	\end{customthm}  
	
	\subsection{Outline of the paper}
	The structure of the rest of this paper is as follows. Section \ref{sec:EU} is devoted to the analysis of the $2D$ Euler equation: we begin with providing some simple and explicit solutions covered by our analysis. Then we proceed to prove our first main result, which is the global existence and uniqueness of the solution with bounded and $m$-fold rotationally symmetric vorticity. After that we establish global existence in the critical spaces $\mathring{C}^{\alpha}$. We conclude the section with some general bounds in these spaces, which will be useful in the following sections.
	
	Then, in Section \ref{sec:homog}, we study the special case of radially homogeneous solutions to the $2D$ Euler equation. This consideration gives a $1D$ active scalar equation, and we establish some basic properties of the model. In particular, we show that upon a few simple assumptions on the initial data, its solution converges to an equilibrium state. Moreover, by considering measure valued vorticities, we obtain existence of time periodic and quasi-periodic solutions to $2D$ Euler. 
	
	In Section \ref{sec:SQG} we study the SQG equation. We establish a local well-posedness in the critical spaces, which contains the class of radially homogeneous solutions described by a $1D$ model. We then show a conditional finite time blow-up result.

\section{Existence and uniqueness for the $2D$ Euler}\label{sec:EU}

Let us state the main result of this section. Assume that the initial vorticity $\omega_0$ is just bounded and $m$-fold rotationally symmetric with some $m \ge 3$. Then, it shows that under this setting, the velocity can be uniquely recovered from the vorticity (so that the $2D$ Euler system can be formulated in terms of the vorticity alone), and that this vorticity formulation is globally well-posed. 

\begin{theorem}\label{thm:main}
	Assume that $\omega_0$ is a bounded and $m$-fold rotationally symmetric function on the plane, for some $m \ge 3$. Then, there is a unique global solution to the $2D$ Euler equation \begin{equation*}
	\begin{split}
	\pr_t \omega + u \cdot \n \o = 0~, 
	\end{split}
	\end{equation*} with $\omega \in L^\infty([0,\infty); L^\infty(\mathbb{R}^2) )$ and $m$-fold rotationally symmetric. Here, $u = u(\o)$ is the unique solution to the system \begin{equation*}
	\begin{split}
	\nabla \times u &= \o~,\qquad \nabla \cdot u = 0~,
	\end{split}
	\end{equation*} under the assumptions $|u(x)| \le C|x|$ and $m$-fold rotationally symmetric.\footnote{We say that a vector-valued function $v: \mathbb{R}^2 \rightarrow \mathbb{R}^2$ is $m$-fold symmetric when $v(O_{2\pi/m}x) = O_{2\pi/m}(v(x))$ for all $x \in \mathbb{R}^2$.} It satisfies the principal value version of the Biot-Savart law: \begin{equation*}
	\begin{split}
	u(t,x) = \lim_{R \rightarrow \infty} \frac{1}{2\pi} \int_{|y| \le R}  \frac{(x-y)^\perp}{|x-y|^2} \o(t,y) dy~. 
	\end{split}
	\end{equation*}
\end{theorem}

%Given the $L^\infty$ theorem, it is straightforward to prove the global well-posedness in H\"older spaces, without the usual $L^p$ assumption on the vorticity. One may proceed as in the proof of Theorem \ref{thm:smooth_away_from_0} using the flow maps.

%\begin{corollary}\label{thm:main_Holder} 
%	Assume that $\omega_0 \in C^{\alpha}(\mathbb{R}^2)$ with $0 \le \alpha \le 1$, and it is $m$-fold rotationally symmetric for some $m \ge 3$. Then, there is a unique global solution to the $2D$ Euler equation with $\omega \in L^\infty_{loc}([0,\infty); C^{\alpha}(\mathbb{R}^2) )$ and $m$-fold rotationally symmetric.
%\end{corollary}

\subsection{Explicit solutions}

Before going into the proof of Theorem \ref{thm:main}, we give a few classes of very simple solutions, whose \textit{uniqueness} is covered by our extension of the Yudovich theory. Some further examples will be provided in Section \ref{sec:homog}.

\begin{example}[Radial eddies]
	Radial and stationary solutions in $2D$ are widely known, see for instance the book of Bertozzi and Majda \cite[Chap. 2]{MB}. Take some vorticity function which depends only on the radius; i.e. $\o(x) = f(|x|)$ for some $f$ smooth and compactly supported. Then, it defines a stationary solution with the velocity obtained via the following radial Biot-Savart law: \begin{equation}\label{eq:radial_BS}
	\begin{split}
	u(x) = \frac{x^\perp}{|x|^2} \int_0^{|x|} s f(s) ds~. 
	\end{split}
	\end{equation} 
	
	In the simplest case when $f$ is the characteristic function on $[0,R]$, one sees that $u(x) = x^\perp/2$ up to $|x| \le R$ and then decays as $|x|^{-1}$. One may consider the limiting case when $R \rightarrow \infty$: our result implies that, $u(x):= x^\perp/2$ on $\mathbb{R}^2$ is the \textit{unique} solution to the Euler equation, as long as we require $m$-fold symmetry for some $m \ge 3$ and the growth condition $|u(x)| \le C|x|$. More generally, the expression \eqref{eq:radial_BS} defines the unique solution, with $f$ just in $L^\infty(\mathbb{R})$. 
\end{example}

We now turn to some other stationary solutions, which does not seem to be well known. 

\begin{example}[Stationary solutions with odd symmetry]\label{ex:stationary}
	Take the disk with some radius $R > 0$. For each integer $m \ge 1$, consider the following vorticity configuration (in polar coordinates): 
	\begin{equation}\label{eq:stationary}
	\begin{split}
	\omega^{(m)} (r,\theta) = \left\{ \begin{array}{rl}
	+1 &\mbox{ if  $\theta \in \cup_{j=0}^{m-1}[2j\pi/m,(2j+1)\pi/m)$}~, \\
	-1 &\mbox{ otherwise .}
	\end{array} \right.
	\end{split}
	\end{equation} These solutions are stationary for any $m \ge 1$; indeed, the Euler equation preserve odd symmetry of the vorticity along a line, and for each line of discontinuity of $\omega^{(m)}$, it is odd so that the fluid particles cannot cross such a line. Now we take the limit $R \rightarrow \infty$: our main result implies that, for $m \ge 3$, the stationary one is the \textit{unique} solution to the Euler equation. 	
\end{example}

Taking any linear combination of the above two examples, we get another class of interesting solutions, which simply rotates around the origin with constant angular speed. 

\begin{example}[Rotating solutions]
	For $ m \ge 3$, take any two constants $c_1$ and $c_2$, and define using polar coordinates \begin{equation*}
	\begin{split}
	\omega (r,\theta) = \left\{ \begin{array}{rl}
	c_1 &\mbox{ if  $\theta \in \cup_{j=0}^{m-1}[2j\pi/m,(2j+1)\pi/m)$}~, \\
	c_2 &\mbox{ otherwise .}
	\end{array} \right.
	\end{split}
	\end{equation*} on $\mathbb{R}^2$. Then, the resulting unique solution rotates with angular velocity $(c_1 + c_2)/2$. (In particular, when $c_1 = c_2 = 1$, the vorticity equals 1 everywhere, which should rotate with angular speed 1.) This is based on the previous example and the following simple fact:  if $(u,\omega)$ be a solution of \eqref{eq:Euler}, then for any constant $c \in \mathbb{R}$, \begin{equation}\label{eq:rotating_solution}
	\begin{split}
	\tilde{\omega}(t,x)&:= \omega(t,O_{ct}^{-1}x) + 2c \\
	\tilde{u}(t,x)&:= O_{ct} u(t,O_{ct}^{-1}x) + cx^\perp
	\end{split}
	\end{equation}
	is also a solution, where $O_\theta$ denotes the matrix of counterclockwise rotation by angle $\theta$.
\end{example}

\begin{remark}
	In all of the above examples, the solutions are well-defined in terms of the Yudovich theory on the disk $B(0,R)$ for any $R> 0$. In this setting, it is an easy matter to show that the particle trajectories of these Yudovich solutions converge to those of ours in the limit $R\rightarrow \infty$ uniformly on compact sets. 
\end{remark}

	\begin{remark}
	In an interesting recent work \cite{LS1,LS2}, the authors have successfully classified all stationary solutions to the $2D$ Euler equation, which has the form \begin{equation*}
	\begin{split}
	u = \nabla^\perp(r^\lambda \Psi(\theta) )
	\end{split}
	\end{equation*} for some $\lambda \in \mathbb{R}$, and with some regularity assumption on $\Psi$. The above stationary solutions correspond to the case $\lambda = 2$ but they have escaped the classification in \cite{LS1,LS2}; our understanding is that the authors work under the situation where $u \in C^1$ on the unit circle (which is natural to do in their framework), while our solutions only satisfy $ u \in C^{0,1}$. 
\end{remark}

\begin{example}[Counterexamples for $m < 3$]
	By means of a few explicit examples, we show that the symmetry assumption for some $m \ge 3$ is essential. First, we consider the stationary solutions \eqref{eq:stationary} from Example \ref{ex:stationary} for $m = 1$ and $2$. In the case $m = 1$ (i.e. odd vorticity), the velocity does not vanish at the origin. Actually, it is easy to check that $|u(0,0)| \approx CR$ as $R \rightarrow \infty$, and therefore there cannot exist a limiting solution. Indeed, one has the following two-parameter family of velocity fields associated: \begin{equation*}
	\begin{split}
	u = \begin{pmatrix}
	|x_2| \\ 0
	\end{pmatrix} + c_1 \begin{pmatrix}
	-x_1 \\ x_2
	\end{pmatrix} + c_2 \begin{pmatrix}
	x_2 \\ x_1
	\end{pmatrix},
	\end{split}
	\end{equation*} where $c_1, c_2$ are some constants. Observe that these velocity fields are Lipschitz continuous and grow linearly in space. This is simply the well-known nonuniqueness issue coming from harmonic polynomials $x_1x_2$ and $x_1^2 - x_2^2$, which makes it impossible to determine uniquely the stream function.  Next, when $m = 2$, these solutions for finite $R$ are often called the Bahouri-Chemin solutions after the work \cite{BC}, and the limit $R \rightarrow \infty$ is not covered by our analysis. In this case, the associated velocity vector field near the origin satisfies $|u(x)| \approx C|x|(\ln R - \ln |x|)$ (see also \eqref{eq:counterex_Poisson}), which is in contrast with the estimate $|u(x)| \le C|x|$ available when $m \ge 3$. Actually one may compute that $C(R) \approx C\ln R$ as $R \rightarrow +\infty$, and hence it is not possible to make sense of the ``infinite'' Bahouri-Chemin solution. Note that the harmonic polynomials $x_1x_2$ and $x_1^2 - x_2^2$ are 2-fold symmetric, so that the above nonuniqueness issue is still present in this case. 
\end{example}

\begin{remark}
	In \cite{EJill}, we utilized a suitably smoothed-out version of the Bahouri-Chemin solution to give a proof of the ill-posedness of the $2D$ Euler equation in Sobolev spaces scale as $\V \o \V_{L^\infty}$. In view of this, Theorem \ref{thm:main} confirms that, in a sense, the Bahouri-Chemin scenario is the ``only option'' for showing such an ill-posedness result, at least near the origin. 
\end{remark}

\subsection{Proof of Theorem \ref{thm:main}}

The proof proceeds as follows. First, we deal with the issue of uniquely solving for the stream function $\Psi$ in $\Delta \Psi = \omega$  in our setup, which is equivalent with uniquely recovering the velocity from the vorticity. The existence is then easily shown thanks to the Biot-Savart formula. 

Once this is done, we have a well-defined vorticity formulation of the $2D$ Euler equation, and we first show uniqueness, using the Osgood type uniqueness for ODEs. Then the existence can be actually shown along the similar lines. These arguments are indeed an adaptation of an elegant proof of the Yudovich theorem given in the book \cite[Chap. 2]{MP}. 

Let us actually prove a slightly more general version of the uniqueness result than we need, which hopefully helps to clarify the situation. 

\begin{lemma}[Uniqueness of the Poisson problem]\label{lem:unique_Poisson}
	Assume that a function $\Psi:\mathbb{R}^2 \rightarrow \mathbb{R}$ in $\dot{W}^{2,p}(\mathbb{R}^2)$ for some $1 <p \le \infty$ satisfies \begin{itemize}
		\item $\Delta \Psi = 0$, 
		\item for all $ x\in \mathbb{R}^2$, $|\Psi (x)| \le C|x|^2(1+|x|^{1-\epsilon})$ for some $\epsilon$, and 
		\item for all sufficiently large $R > 0$, \begin{equation*}
		\begin{split}
		\int_{-\pi}^{\pi} \Psi(R\theta) \exp({\pm i\theta}) d\theta = 0 = \int_{-\pi}^{\pi} \Psi(Re^{i\theta}) \exp({\pm 2i\theta}) d\theta~.
		\end{split}
		\end{equation*}
	\end{itemize} Then $\Psi$ is identically zero. 
\end{lemma}

In particular, this implies that the Poisson problem \begin{equation}\label{eq:Poisson}
\begin{split}
\Delta \Psi = \omega
\end{split}
\end{equation} is uniquely solvable in the class of functions that grows at most quadratically and has $m$-fold rotational symmetry, up to a constant. 

\begin{proof}
	For each $R > 0 $ (which is assumed to be sufficiently large), we have the representation formula \begin{equation*}
	\begin{split}
	\Psi(x) = \frac{R^2 - |x|^2}{2\pi R} \int_{\partial B(0,R)} \frac{1}{|y-x|^2}  \Psi(y) dS(y)~.
	\end{split}
	\end{equation*} We will keep subtracting zeroes to gain more and more decay in $y$ in the above formula. 
	
	First, from $\Psi(0) = 0 $, we may subtract an appropriate multiple of $\Psi (0)$ to obtain \begin{equation*}
	\begin{split}
	\Psi (x)=  \frac{R^2 - |x|^2}{2\pi R} \int_{\partial B(0,R)} \left[ \frac{|y|^2 - |y-x|^2}{|y-x|^2 |y|^2}  \right] \Psi(y) dS(y)~.
	\end{split}
	\end{equation*} Note that the kernel now decays as $R^{-3}$ for a fixed $x$. 
	
	Next, from the vanishing condition \begin{equation*}
	\begin{split}
	\int_{\pr B(0,R)} \frac{y}{|y|^4} \Psi(y) dS(y) = 0~,
	\end{split}
	\end{equation*} we further rewrite \begin{equation*}
	\begin{split}
	\Psi (x) &= - \frac{R^2 - |x|^2}{2\pi R} \int_{\partial B(0,R)}  \left[     \frac{|x|^2 - 2x \cdot y}{|y-x|^2 |y|^2}  -  \frac{2x\cdot y}{|y|^4}  \right]\Psi(y) dS(y)~, \\
	&= - \frac{R^2 - |x|^2}{2\pi R} \int_{\partial B(0,R)}  \left[     \frac{|x|^2 |y|^2 +2|x|^2  x\cdot y - 4(x\cdot y)^2 }{|y-x|^2|y|^4} \right]\Psi(y) dS(y)~.
	\end{split}
	\end{equation*} 
	
	Finally, from the vanishing of second Fourier modes, \begin{equation*}
	\begin{split}
	\int_{\partial B(0,R)}  \frac{ 2y_1y_2 }{|y|^6} \Psi(y) dS(y) = \int_{\partial B(0,R)}  \frac{ y_1^2 - y_2^2 }{|y|^6} \Psi(y) dS(y) = 0~.
	\end{split}
	\end{equation*} Hence, we can subtract appropriate multiples of the above to modify the kernel \begin{equation*}
	\begin{split}
	\frac{|x|^2 |y|^2 +2|x|^2  x\cdot y - 4(x\cdot y)^2 }{|y-x|^2|y|^4} +\frac{2|x|^2}{|y|^4} + \frac{(y_1^2 - y_2^2)(x_1^2 - x_2^2)}{|y|^6} - \frac{2x_1x_2 y_1 y_2}{|y|^6}~.
	\end{split}
	\end{equation*} A direct computation shows that this expression decays as $R^{-5}$ in $|y|$ for a fixed $x$. Therefore, we bound \begin{equation*}
	\begin{split}
	|\Psi(x)| \le 10 C R \int_{\pr B(0,R)} \frac{|x|^3}{R^5} R^{3-\epsilon} dS(y) \le C'|x|^3 \frac{1}{R^\epsilon}~,
	\end{split}
	\end{equation*} since $|\Psi(y)| \le CR^{3-\epsilon} $ on $\pr B(0,R)$. Fixing $x$ and taking $R \rightarrow +\infty$ finishes the proof. 
\end{proof}

\begin{lemma}\label{lem:vel_linearbound}
	Assume that $\omega \in L^\infty(\mathbb{R}^2)$ and satisfies \begin{equation*}
	\begin{split}
	\int_{-\pi}^{\pi} \omega(Re^{i\theta}) \exp({\pm i\theta}) d\theta = 0 = \int_{-\pi}^{\pi} \omega(R\theta) \exp({\pm 2i\theta}) d\theta~.
	\end{split}
	\end{equation*} Then, the principal value \begin{equation}
	\begin{split}
	u(x) = \lim_{R \rightarrow \infty} \frac{1}{2\pi} \int_{|y| \le R} \frac{(x-y)^\perp}{|x-y|^2} \omega(y)dy
	\end{split}
	\end{equation} is pointwise well-defined, with a bound \begin{equation}
	\begin{split}
	|u(x)| \le C \V \omega \V_{L^\infty} |x| ~.
	\end{split}
	\end{equation}
\end{lemma}
\begin{proof}
	We consider two regions: when $y$ satisfies $|y| \le 2|x|$ and $|y| > 2|x|$. In the first region, a brute force bound gives \begin{equation}
	\begin{split}
	|u(x)| &\le C\V \omega \V_{L^\infty} \int_{|y| \le 2|x|} \frac{1}{|x-y|} dy \le  C |x| \V \omega\V_{L^\infty}~.
	\end{split}
	\end{equation}
	In the latter region, we proceed exactly as in the previous lemma; from the vanishing of first Fourier modes, we rewrite it as \begin{equation}
	\begin{split}
	p.v.~ \frac{1}{2\pi}\int_{|y| > 2|x|} \left[ \frac{(x-y)^\perp}{|x-y|^2} +\frac{y^\perp}{|y|^2}  \right]\omega(y)dy~,
	\end{split}
	\end{equation} and note that the first component equals \begin{equation*}
	\begin{split}
	 p.v.~\frac{1}{2\pi}  \int_{|y|>2|x|} \left[ \frac{y_2(x_1^2+x_2^2) - x_2(y_1^2 - y_2^2) + 2x_1 y_1y_2 }{|x-y|^2|y|^2} \right] \omega(y)dy~,
	\end{split}
	\end{equation*} so that the principal value makes sense once we have vanishing of the second Fourier modes. A similar conclusion holds for the second component of the velocity. Subtracting appropriate quantities and directly integrating in $y$ gives the desired bound $|u(x)| \le C \V \o \V_{L^\infty} |x|$. 
\end{proof}

\begin{lemma}
	Under the assumptions of Lemma \ref{lem:vel_linearbound}, we have a log-Lipschitz estimate of the form \begin{equation}
	\begin{split}
	|u(x) - u(x')| \le C \V\omega\V_{L^\infty} |x-x'| \ln\left( \frac{c\max(|x|,|x'|)}{|x-x'|} \right)~.
	\end{split}
	\end{equation}
\end{lemma}

\begin{proof}
	We start with the expression \begin{equation}
	\begin{split}
	u(x) - u(x') = p.v. \, \frac{1}{2\pi} \int_{\mathbb{R}^2} \left[ \frac{(x-y)^\perp}{|x-y|^2} - \frac{(x'-y)^\perp}{|x'-y|^2} \right] \omega(y)dy~,
	\end{split}
	\end{equation} and assume $|x| \ge |x'|$. We split the integration into 3 domains, $A= \{ y: |x-y| \le 2|x-x'| \}$, $C = \{ y: |y| > 3|x| \}$, and $B = \mathbb{R}^2 \backslash (A \cup C)$.
	
	In the region $A$, a brute force bound on the kernel gives a contribution \begin{equation}
	\begin{split}
	C \V\omega\V_{L^\infty} |x-x'|~,
	\end{split}
	\end{equation} and in the annulus-shape region $B$, due to the singular nature of the kernel $\nabla K$ we obtain a log-Lipschitz contribution \begin{equation}
	\begin{split}
	C \V \omega\V_{L^\infty} |x-x'| \ln\left( \frac{c|x|}{|x-x'|} \right)~.
	\end{split}
	\end{equation} Finally, in the region $C$, we re-write each of $K(x,y)$ and $K(x',y)$ as done in the previous lemma to gain decay in $y$, and combine the resulting expressions to obtain a bound of the form \begin{equation}
	\begin{split}
	C |x-x'| \frac{|x|}{|y|^3}
	\end{split}
	\end{equation} in the region $C$. Integrating in $y$ completes the proof.
\end{proof}

In particular, when $\omega$ is bounded and  $m$-fold rotationally symmetric for some $m \ge 3$, there exists a velocity vector field which grows at most linearly and log-Lipschitz continuous. From the convergence of the Biot-Savart law, it follows that $\n \cdot u = 0$ and $\n \times u = \o $ in the sense of distributions. 

Integrating this velocity vector field in space, one obtains the existence to the Poisson problem \eqref{eq:Poisson}, which actually satisfies the growth condition $|\Psi(x)| \le C|x|^2$. 

\begin{remark}
	The situation is different when we do not have vanishing of the second Fourier modes (equivalently, in the case of $2$-fold rotational symmetry). Explicitly, for \begin{equation*}
	\begin{split}
	\omega(y_1,y_2) = \alpha_1 \frac{y_1y_2}{|y|^2} + \alpha_2 \frac{y_1^2 - y_2^2}{|y|^2}~,
	\end{split}
	\end{equation*} (which is bounded) we have the following solutions to the Poisson problem \begin{equation}\label{eq:counterex_Poisson}
	\begin{split}
	\Psi(y_1,y_2) = \alpha_1 y_1y_2 \ln |y| + \beta_1 y_1y_2 + \alpha_2 (y_1^2 - y_2^2) \ln |y| + \beta_2 (y_1^2 - y_2^2)~. 
	\end{split}
	\end{equation} These examples were presented in \cite{E1}. 
\end{remark}

\begin{remark}
	We have seen that in the frozen-time case, the necessary condition is vanishing of the first two Fourier modes. Unfortunately, this information does not propagate under the Euler dynamics, in general. Therefore we really have to stick to the $m$-fold rotationally symmetric assumption, with some $m \ge 3$. For an explicit example, take the following vorticity configuration: \begin{equation*}
	\begin{split}
	\omega_0(r,\theta) = \left\{ \begin{array}{rl}
	1 &\mbox{ if  $\theta \in [-a,a] \cup [-a+\pi/2,a+\pi/2] ~,$ 
	}\\ 0 &\mbox{ otherwise .}
	\end{array} \right.
	\end{split}
	\end{equation*} This satisfies the vanishing assumptions, while $u_0 \cdot \n \o_0 $ does not. Note that this example was also presented in \cite{E1}.
\end{remark}

We have seen that the vanishing of the first two Fourier modes gives the decay $|x|^{-3}$ of the Biot-Savart kernel. Indeed, upon assuming further vanishing conditions, one continues to gain extra decay by subtracting appropriate quantities. Of course, it is more efficient to prove it by explicitly symmetrizing the kernel when we assume $m$-fold symmetry. More formally, one has:

\begin{lemma}[see \cite{E1}]\label{lem:decay}
	Let $K$ denote the Biot-Savart kernel \begin{equation*}
	\begin{split}
	K(x,y) := \frac{1}{2\pi} \frac{(x-y)^\perp}{|x-y|^2}~. 
	\end{split}
	\end{equation*} Then for each $m \ge 1$, the $m$-fold symmetrization in $y$ gives the decay $|y|^{-m}$: \begin{equation*}
	\begin{split}
	\frac{1}{m}\sum_{i=1}^m K(x, O^i_{2\pi/m} y ) = \frac{P_m(x,y)}{\prod_{i=1}^m |x -O^i_{2\pi/m} y|^2  }
	\end{split}
	\end{equation*} where $P_m(x,y)$ is a vector of homogeneous polynomials in the components of $x$ and $y$ of degree $2m-1$, which contains powers of $y_1$ and $y_2$ only up to the degree $m$. 
\end{lemma}

\begin{proof}
	Fix some $y_0 \ne 0$. Then \begin{equation*}
	\begin{split}
	P_m(x,y_0) = \nabla^\perp_x \left( \prod_{i=1}^m |x - O^i_{2\pi/m} y_0|^2 \right) ~,
	\end{split}
	\end{equation*} where the function \begin{equation*}
	\begin{split}
	F_{y_0}(x) := \prod_{i=1}^m |x - O^i_{2\pi/m} y_0|^2
	\end{split}
	\end{equation*} is a $m$-fold symmetric polynomial in $x_1$ and $x_2$. Hence it only consists of terms which have degree 0, $m$, and $2m$ in $x_1,x_2$. The statement follows.
\end{proof} As a corollary, we have \begin{corollary}
For $m\ge 3$, define the kernel \begin{equation*}
\begin{split}
K^{(m)}(x,y) = \frac{1}{m}\sum_{i=1}^m K(x, O^i_{2\pi/m} y ).
\end{split}
\end{equation*} Then, for $|y| \ge 2|x|$, we have \begin{equation*}
\begin{split}
|K^{(m)}(x,y)| \le C_m \frac{|x|^{m-1}}{|y|^m}
\end{split}
\end{equation*} for some constant $C_m > 0$. Moreover, for points $y, y'$ satisfying $|y|, |y'| \ge 2|x|$ and $|y - y'| \ge |x|$, \begin{equation*}
\begin{split}
\left| K^{(m)}(x,y) - K^{(m)}(x,y') \right| \le C_m |y - y'| \frac{|x|^{m-1}}{|y|^{m+1}}. 
\end{split}
\end{equation*}
\end{corollary}

We now turn to the task of obtaining uniqueness of the Euler solution. We will reduce it to the following Osgood uniqueness condition for a suitable quantity.

\begin{lemma}\label{lem:osgood}
	If a continuous function $f:[0,\epsilon) \rightarrow \mathbb{R}_{\ge 0}$ satisfy $f(0) = 0$ and \begin{equation}
	f(t) \le C \int_0^t f(s) \ln(1+ \frac{1}{f(s)} ) ds
	\end{equation}	then $f \equiv 0$.
\end{lemma}
\begin{proof}
	See for example the book of Marchioro and Pulvirenti \cite[p. 68]{MP}.
\end{proof}

For simplicity, let us set $\rho(a) := a \ln(1+1/a)$ for $a \ge 0$. 

\begin{lemma}[Uniqueness]\label{lem:unique}
	Given $\omega_0 \in L^\infty(\mathbb{R}^2)$ with 4-fold rotational symmetry, there is at most one solution of the $2D$ Euler equation (in vorticity formulation) with $\omega \in L^\infty_t L^\infty_x$ with 4-fold rotational symmetry. 
\end{lemma}

The assumption that $\omega$ is 4-fold symmetric is just for concreteness and simplicity of notation; a similar argument goes through for any $m \ge 3$. 

\begin{proof}
	Assume there exist two solution triples, $(\omega,u,\Phi)$ and $(\tilde{\omega},\tilde{u},\tilde{\Phi})$. A simple observation is that we have a linear bound \begin{equation}
	\begin{split}
	|\Phi(t,x)| \le |x| + \int_0^t |u(s,\Phi(s,x))|ds \le |x| + C\V\omega_0\V_{L^\infty} \int_0^t |\Phi(s,x)|ds~. 
	\end{split}
	\end{equation} This implies that we have bounds \begin{equation*}
	\begin{split}
	|x|\exp(-C\V\omega_0\V_{L^\infty}t) \le |\Phi(t,x)| \le |x|\exp(C\V\omega_0\V_{L^\infty}t),
	\end{split}
	\end{equation*} and the same bounds hold for $|\tilde{\Phi}(t,x)|$. Similarly, \begin{equation*}
	\begin{split}
	|\Phi(t,x) - x| \le \int_0^t |u(s,\Phi(s,x))| ds \le C\V  \omega_0\V_{L^\infty}\int_0^t |\Phi(s,x)| ds ,
	\end{split}
	\end{equation*} so that for any $\epsilon > 0$ ($\epsilon = 1/10$ will suffice for our purposes), there exists $T = T(\epsilon,\V \omega_0\V_{L^\infty})$ such that for $0 \le t \le T$, we have bounds \begin{equation}\label{eq:comparable_ratio}
	\begin{split}
	1 - \epsilon \le \frac{|\Phi(t,x)|}{|x|}\, , \, \frac{|\tilde{\Phi}(t,x)|}{|x|} \le 1 + \epsilon
	\end{split}
	\end{equation} as well as \begin{equation}
	\begin{split}
	\frac{|\Phi(t,x) - x|}{|x|}\,,\, \frac{|\tilde\Phi(t,x) - x|}{|x|} \le \epsilon 
	\end{split}
	\end{equation} and hence \begin{equation}
	\begin{split}
	\frac{|\Phi_t(x)-\tilde{\Phi}_t(x)|}{|x|} \le 2\epsilon
	\end{split}
	\end{equation} for all $0\le t \le T = T(\epsilon, \V \omega_0\V_{L^\infty} )$, uniformly in $x \in \mathbb{R}^2 $. From now on, we restrict ourselves to the time interval $[0,T]$, with $\epsilon = 1/10$.

	We want to close an estimate of the form in Lemma \ref{lem:osgood} in terms of the quantity \begin{equation}
	\begin{split}
	f(t):= \int_{\mathbb{R}^2} \frac{|\Phi_t(x)-\tilde{\Phi}_t(x)|}{|x|} \cdot \frac{|\omega_0(x)|}{1 + |x|^3}dx~.
	\end{split}
	\end{equation}
	Note that ${|\omega_0(x)|}/(1 + |x|^3) dx$ is a finite measure, and its measure is bounded by a constant multiple of $\V \omega_0\V_{L^\infty}$. In particular, $f$ is finite for all time. We then write \begin{equation}
	\begin{split}
	\Phi_t(x) - \tilde{\Phi}_t(x) =& \int_0^t \int_{\mathbb{R}^2} \left[ K(\Phi_s(x)-\Phi_s(y)) - K(\Phi_s(x) - \tilde{\Phi}_s(y)) \right] \omega_0(y)dy ds \\
	&\quad +\int_0^t \tilde{u}_s(\Phi_s(x)) - \tilde{u}_s(\tilde{\Phi}_s(x)) ds~,
	\end{split}
	\end{equation} and integrate against  \begin{equation*}
	\begin{split}
	\frac{1}{|x|}\frac{|\omega_0(x)|}{1+|x|^3}dx~.
	\end{split}
	\end{equation*} From the log-Lipschitz bound on $\tilde{u}$, and using \eqref{eq:comparable_ratio}, we bound the second term in terms of \begin{equation*}
	\begin{split}
	C\V \omega_0 \V_{L^\infty}\int_0^t \int_{\mathbb{R}^2} \frac{|\Phi_s(x) -\tilde{\Phi}_s(x)|}{|x|} \ln\left( \frac{c|x|}{|\Phi_s(x) -\tilde{\Phi}_s(x)|} \right) \frac{|\omega_0(x)|}{1+|x|^3} dx ds \le C\int_0^t    \rho (f(s)) ds~,
	\end{split}
	\end{equation*} where we have used the Jensen's inequality with respect to the finite measure $\frac{|\omega_0(x)|}{1+|x|^3}dx$.\footnote{We note that, while the multiplicative constant $C > 0$ arising from the Jensen's inequality depends on the total mass of the measure, the constant is uniformly bounded in terms of $\V \omega_0\V_{L^\infty}$.} This is precisely the bound we want. 
	
	Turning to the first term, we have \begin{equation}
	\begin{split}
	\int_0^t ds \int_{y} \left[ \int_x  \left[ K(\Phi_s(x)-\Phi_s(y)) - K(\Phi_s(x) - \tilde{\Phi}_s(y)) \right] \frac{1}{|x|} \frac{|\omega_0(x)|}{1 +|x|^3} dx \right] \omega_0(y)  dy~.
	\end{split}
	\end{equation} For each fixed $y$, once we obtain an estimate of the form \begin{equation}\label{eq:want}
	\begin{split}
	&\left| \int_x  \left[ K(\Phi_s(x)-\Phi_s(y)) - K(\Phi_s(x) - \tilde{\Phi}_s(y)) \right]  \frac{1}{|x|} \frac{|\omega_0(x)|}{1 +|x|^3} dx \,\, \omega_0(y) \right| \\ &\qquad \le C \frac{|\Phi(t,y)-\tilde{\Phi}(t,y)|}{|y|} \ln\left( \frac{c|y|}{|\Phi(t,y)-\tilde{\Phi}(t,y)|} \right) \frac{|\omega_0(y)|}{1+ |y|^3}~,
	\end{split}
	\end{equation} we obain a bound $C \int_0^t \rho(f(s))ds$, and this completes the proof. 
	
	We consider 4 regions. First, $A = \{  x: |x| \le |y|/10  \}$, $D = \{ x: |x| > 10|y| \}$. Then set $B = \{  x: |x-y| \le 3 |\Phi(t,y)-\tilde{\Phi}(t,y)|  \}$ and $C = \mathbb{R}^2 \backslash (A \cup B \cup D)$. Regions $A, B, D$ do not overlap and $C$ is an annulus-shape domain. 
	
	\medskip
	
	(i) Region $A$
	
	\smallskip
	
	We symmetrize each of the kernels in $y$. Then, combining two fractions together, we may pull out a factor of $|\Phi(t,y) - \tilde{\Phi}(t,y)|$. Then note that each factor in the denominator is bounded below by a constant multiple of $|x-y|$, which is in turn bounded below by a multiple of $|y|$ in this region. Therefore, we obtain a bound of the form (see Lemma \ref{lem:decay}) \begin{equation}
	\begin{split}
	\int_A \left[	C|\Phi(t,y)-\tilde{\Phi}(t,y)| \frac{|x|^{14}+|x|^3 |y|^{11}}{|y|^{16}} \right] \frac{1}{|x|} \frac{|\omega_0(x)|}{1 +|x|^3} dx \,\, |\omega_0(y)|~.
	\end{split}
	\end{equation} This is integrable in $x$, and integrating gives a bound \begin{equation}
	\begin{split}
	C \frac{|\Phi(t,y)-\tilde{\Phi}(t,y)|}{|y|}  \frac{|\omega_0(y)|}{1 + |y|^3}~.
	\end{split}
	\end{equation}
	
	\medskip
	
	(ii) Region $D$
	
	\smallskip

	This time, we symmetrize the kernel in $x$. Proceeding similarly as in the region $A$, we obtain a bound \begin{equation}
	\begin{split}
	\int_D |\Phi(t,y)-\tilde{\Phi}(t,y)| \frac{|y|^{14} + |y|^2 |x|^{12} }{|x|^{16}} \frac{1}{|x|} \frac{|\omega_0(x)|}{1 +|x|^3} dx \,\, |\omega_0(y)|~.
	\end{split}
	\end{equation}
	This is integrable in $x$, and clearly we have \begin{equation*}
	\begin{split}
	\int_D \frac{|y|^{14} + |y|^2 |x|^{12} }{|x|^{16}} \frac{|y|}{|x|} \frac{1 + |y|^3}{1 +|x|^3} dx \le C~.
	\end{split}
	\end{equation*} This results in the same bound as in the region $A$.
	
	\medskip
	
	(iii) Regions $B$ and $C$
	
	\smallskip
	
	In this case, we have $|x| \approx |y|$, so we can freely interchange the powers of $|x|$ with $|y|$ in the denominator. We then estimate the kernel exactly as in the proof of the log-Lipschitz bound of $u(x)$. 
	
	\medskip
	
	Collecting the bounds, we obtain \eqref{eq:want}. Hence, we have obtained \begin{equation*}
	\begin{split}
	f(t) \le C \int_0^t \rho(f(s))ds~,
	\end{split}
	\end{equation*} on some time interval $t \in [0,T]$ with $T = T(\V \omega_0 \V_{L^\infty})$ and $C = C(m,\V\omega_0\V_{L^\infty})$. Lemma \ref{lem:osgood} guarantees that $f \equiv 0 $ on $[0,T]$, and since $\V \omega_0 \V_{L^\infty} = \V \omega_T \V_{L^\infty}$, this argument can be extended to any finite time.
	
	Hence, we have shown that the particle trajectories $\Phi(t,x)$ and $\tilde{\Phi}(t,x)$ coincide on the support of $\omega_0$, for all time. This trivially implies that $\omega \equiv \tilde{\omega}$ (and therefore, $\Phi(t,x) \equiv \tilde{\Phi}(t,x)$ everywhere). The proof is now complete. 
\end{proof}

\begin{proof}[Proof of Theorem \ref{thm:main}]
	It only remains to establish the existence. A particularly nice feature of the proof in \cite{MP}, which we have adopted here, is that the existence is shown in a completely parallel manner as the uniqueness. We only recall the main steps. 
	
	\medskip
	
	(i) Construction of the approximate sequence
	
	\smallskip
	
	Starting with the initial value $\omega^{(0)}(t,x) := \omega_0(x)$, we inductively define \begin{equation}\label{eq:approximate_seq}
	\begin{split}
	u^{(n)}_t(x) &:= p.v.\, (K * \omega_t^{(n-1)})(x)~,\\
	\frac{d}{dt} \Phi^{(n)}_t (x) &:= \left( u^{(n)}_t \circ \Phi^{(n)}_t \right) (x)~,\\
	\omega^{(n)}_t(x) &:= \left( \omega_0 \circ (\Phi^{(n)}_t)^{-1} \right) (x)~. 
	\end{split}
	\end{equation} Our previous lemmas guarantee that each of these definitions makes sense. It is important that the symmetry property remains valid at each step of the iteration. 
	
	\medskip
	
	(ii) Convergence of the sequence
	
	\smallskip
	
	Define for $n \ge 1$, \begin{equation*}
	\begin{split}
	\delta^n(t) := \int_{\mathbb{R}^2} \frac{|\Phi^{(n)}_t(x)-\Phi^{(n-1)}_t(x)|}{|x|} \cdot \frac{|\omega_0(x)|}{1 + |x|^3}dx~.
	\end{split}
	\end{equation*} The argument from the proof of Lemma \ref{lem:unique} applied instead to the pair $(\Phi^{(n)}, \Phi^{(n-1)})$ shows that this quantity is finite and satisfies the inequality \begin{equation*}
	\begin{split}
	\delta^n (t) \le C \int_0^t \rho (\delta^n(s)) ds + C \int_0^t \rho (\delta^{n-1}(s))ds~. 
	\end{split}
	\end{equation*} Upon introducing \begin{equation*}
	\begin{split}
	\bar{\delta}^N(t) := \sup_{n>N} \delta^n(t)~, 
	\end{split}
	\end{equation*} we obtain \begin{equation*}
	\begin{split}
	\bar{\delta}^N(t) \le C \int_0^t	\rho(\bar{\delta}^{N-1}(s)) ds~. 
	\end{split}
	\end{equation*} This is sufficient to show that \begin{equation*}
	\begin{split}
	\lim_{N \rightarrow \infty} \bar{\delta}^N(t) \rightarrow 0~,
	\end{split}
	\end{equation*} uniformly in some short time interval $[0,T]$. This length of the time interval depends only on $\V \o_0 \V_{L^\infty}$ and $m$, so that this argument extends to any finite time. 
	
	\medskip
	
	(iii) Properties of the limit
	
	\smallskip
	
	We have thus shown that, for each fixed time $t \in \mathbb{R}_{+}$, there exists a map $\Phi_t$ defined on the support of $\omega_0$, satisfying \begin{equation*}
	\begin{split}
	\int_{\mathbb{R}^2} \frac{|\Phi_t^{(n)}(x)- \Phi_t(x)|}{|x|} \cdot \frac{|\omega_0(x)|}{1 + |x|^3} dx \longrightarrow 0~, \qquad n \rightarrow \infty~.
	\end{split}
	\end{equation*} Then set \begin{equation*}
	\begin{split}
	\omega_t(x) = \left\{ \begin{array}{rl}
	 \omega_0(\Phi_t^{-1}(x)) &\mbox{ if  $ x= \Phi_t(y)$ for some $ y \in \mbox{supp}(\omega_0) $}~, \\
	0 &\mbox{ otherwise .} \end{array} \right.
	\end{split}
	\end{equation*} This vorticity is $m$-fold rotationally symmetric and bounded. Therefore, we may define $u_t$ and then $\Phi_t$ on the entire plane. It is direct to show that this map is a measure preserving homeomorphism for each time moment. The triple $(\omega_t, u_t, \Phi_t)$ solves the $2D$ Euler equation and satisfies our assumptions. 

	\smallskip
	
	This finishes the proof. \end{proof}

\begin{example}
	In this example, we collect several situations where Theorem \ref{thm:main} applies. 
	\begin{enumerate}
		\item (Torus) Consider the 2D torus $\mathbb{T}^2 = [-\pi,\pi)^2$ and assume that the initial vorticity $\omega_0 \in L^\infty(\mathbb{T}^2)$ satisfies the symmetry $\omega_0(x_1,x_2) = \omega_0(-x_2,x_1) = \omega_0(-x_1,-x_2) = \omega_0(x_2,-x_1)$. Then, we may identify such an initial data with a vorticity defined on $\mathbb{R}^2$ which has 4-fold rotational symmetry around the point $(0,0)$. 
		\item (Square and equilateral triangle) Consider the $2D$ Euler equation on the unit square $\square = [0,1]^2$ with slip boundary conditions, i.e., $u \cdot n = 0$ on the boundary where $n$ is the unit normal vector. Assume that we are given $\omega_0 \in L^\infty(\square)$ which is odd across the diagonal, that is, $\omega_0(x_1,x_2) = -\omega_0(x_2,x_1)$ for $x_1,x_2 \in [0,1]$. One may then extend it as an odd function with respect to all the sides of the square to obtain a 4-fold symmetric initial vorticity on $\mathbb{R}^2$.
		
		A similar procedure can be done for the case of an equilateral triangle. Here, we obtain a 6-fold symmetric vorticity on the plane. 
		
		It is interesting to note that in these cases, assuming that the mean of vorticity is zero on the periodic domain, the velocity is actually bounded on $\mathbb{R}^2$ and therefore our solutions coincide with Serfati's.
		\item (Rational sectors) This time, consider the $2D$ Euler equation on the sector $S_{\pi/m} =  \{ (r,\theta) : 0 \le r < \infty, 0 \le \theta \le \pi/m \}$ for $m \ge 3$ with slip boundary conditions. Given $\omega_0 \in L^\infty(S_{\pi/m})$, we extend it as an odd function onto the whole plane across the boundaries; i.e. $\omega_0(r,\theta) = -\omega_0(r, \frac{2\pi}{m} - \theta )$ for all $r, \theta$. Then, we obtain an $m$-fold symmetric vorticity.
	\end{enumerate}
\end{example}

\begin{remark}
	In all of the above examples, it is not hard to show that if we have additional regularity of the initial data, e.g. $\omega_0 \in C^{k,\alpha}(D)$ for $D \in \{ \mathbb{T}^2, \square, S_{\pi/m} \}$, then this regularity propagates by the Euler equation, even though the odd extension onto $\mathbb{R}^2$ could be discontinuous across the symmetry axis. 
\end{remark}

It is well known that if initially $\omega_0 \in C^{0,1} \cap L^1 (\mathbb{R}^2)$, then the maximum of the gradient can grow at most double exponentially in time. This double exponential bound can be excluded, at least at the origin, in all of the above situations; indeed, this is a direct consequence of the estimate \begin{equation*}
\begin{split}
|u(x)| \le C\V\omega_0\V_{L^\infty} |x|
\end{split}
\end{equation*} which says that no fluid particle can approach the origin faster than the exponential rate. This recovers some of the very recent results of Itoh, Miura, and Yoneda \cite{IMY1,IMY2}. 

\begin{corollary}[see \cite{IMY1,IMY2}]
	Assume that we are in one of the above domains and the initial vorticity $\omega_0$ satisfies the required symmetry assumptions. If in addition $\omega_0$ is Lipschitz, then we have the following exponential bound on the gradient at the origin for all time: \begin{equation*}
	\begin{split}
	\sup_{x\ne 0}\frac{|\omega(t,x)-\omega(t,0)|}{|x|} \le \V\nabla \omega_0\V_{L^\infty} \exp(c\V \omega_0\V_{L^\infty}t)~.
	\end{split}
	\end{equation*}
\end{corollary}

\subsection{Propagation of the angular regularity}

In this subsection, we show well-posedness of the $2D$ Euler equation in certain scaling invariant spaces, which encodes regularity in the angle direction. Let us use the notation \begin{equation}\label{eq:newalpha}
\begin{split}
\V \omega \V_{\mathring{C}^{\alpha}(\mathbb{R}^2)} &:= \V\omega\V_{L^\infty(\mathbb{R}^2)} +  \V |x|^\alpha \omega \V_{C_*^{\alpha}(\mathbb{R}^2)} \\
& := \sup_{x} |\omega(x)| + \sup_{x \ne x'} \frac{\left||x|^\alpha \omega(x) - |x'|^\alpha \omega(x')\right| }{|x-x'|^\alpha}
\end{split}
\end{equation} for $0 < \alpha <1$, and the endpoint case is simply defined by \begin{equation}\label{eq:endpoint}
\begin{split}
\V \omega \V_{\mathring{C}^{0,1}(\mathbb{R}^2)} := \V\omega\V_{L^\infty(\mathbb{R}^2)} + \sup_{x \in \mathbb{R}^2 \backslash \{0\}} \left(|x| |\nabla \omega(x)| \right)~. 
\end{split}
\end{equation} Higher order norms $\mathring{C}^{k,\alpha}$ can be defined for $k \ge 1$: first when $0<\alpha<1$, \begin{equation}\label{eq:newalpha2}
\begin{split}
\V \omega \V_{\mathring{C}^{k,\alpha}(\mathbb{R}^2)} := \V\omega\V_{\mathring{C}^{k-1,1}(\mathbb{R}^2)} +  \V |x|^{k+\alpha} \nabla^k\omega \V_{{C}^{\alpha}(\mathbb{R}^2)} 
\end{split}
\end{equation} and \begin{equation}\label{eq:endpoint2}
\begin{split}
\V \omega \V_{\mathring{C}^{k,1}(\mathbb{R}^2)} := \V\omega\V_{L^\infty(\mathbb{R}^2)} + \sup_{x \in \mathbb{R}^2 \backslash \{0\}} \left(|x|^{k+1} |\nabla^{k+1} \omega(x)| \right)~. 
\end{split}
\end{equation} Here $\nabla^d \omega(x)$ is just a vector which consists of all expressions of the form $\partial_{{i_1}} \cdots \partial_{i_d} \omega(x)$ for $(i_1,\cdots,i_d) \in \{ 1,2 \}^d$. 
It is clear that these spaces deal with angular regularity. Indeed, in the ideal case when $\omega$ depends only on the angle, i.e. $\omega(r,\theta) = h(\theta)$ for some profile $h$ defined on the unit circle, \begin{equation*}
\begin{split}
\V \omega\V_{\mathring{C}^{k,\alpha}(\mathbb{R}^2)} \approx \V h\V_{C^{k,\alpha}(S^1)}~.
\end{split}
\end{equation*}

Note that if we have initial data $\omega_0 \in \mathring{C}^{k,\alpha}(\mathbb{R}^2) $ then $\omega_0$ is actually $C^{k,\alpha}$ away from the origin. This information propagates in time; the solution constructed in the previous subsection remains $C^{k,\alpha}$ away from the origin for all time. Of course it needs to be proved that $\omega(t,\cdot)$ actually belongs to $\mathring{C}^{k,\alpha}(\mathbb{R}^2)$. Not surprisingly, the bound turns out to be double exponential in time. 

\begin{theorem}\label{thm:smooth_away_from_0}
	Assume that $\omega_0\in\mathring{C}^{k,\alpha} $ is $(m+k)$-fold rotationally symmetric with some $m \ge 3$, with $k \ge 0$ and $0<\alpha \le 1$. Then, the unique solution $\omega(t,x) \in L^\infty([0,\infty);L^\infty(\mathbb{R}^2)  ) $ belongs to $L^\infty_{loc}\mathring{C}^{k,\alpha}$ with a bound \begin{equation}\label{eq:double_exp_gradient_bound}
	\begin{split}
	\V \omega(t) \V_{\mathring{C}^{k,\alpha}} \le C\exp(c_1 \exp(c_2 t) )
	\end{split}
	\end{equation} with constants depending only on $k, \alpha$, and $ \V \omega_0 \V_{\mathring{C}^{k,\alpha}}$.
\end{theorem}

In the proof, we shall need the following simple calculation: 
\begin{lemma}\label{lem:velocity_gradient_2}
	Under the assumptions of Theorem \ref{thm:smooth_away_from_0}, for any $0< \alpha \le 1$, we have a  bound \begin{equation}\label{eq:velocity_gradient_2}
	\begin{split}
	\Vert \nabla u \Vert_{L^\infty} \le C_\alpha \Vert \omega \Vert_{L^\infty} \left( 1 + \ln\left( 1 + c_\alpha\frac{ \V  \omega \V_{\mathring{C}^\alpha} }{\Vert \omega \Vert_{L^\infty}} \right) \right)~,
	\end{split}
	\end{equation} with some constants $c_\alpha,C_\alpha$ depending only on $0<\alpha \le 1$.
\end{lemma}

\begin{proof}
	We recall that each entry of the matrix $\nabla u (x)$ has an explicit representation (see \cite{BeCo} for instance) involving a linear combination of the expressions \begin{equation*}
	\begin{split}
	p.v.\,\int_{\mathbb{R}^2} \frac{(x_1-y_1)(x_2-y_2)}{|x-y|^4} \omega(y)dy~,\qquad p.v.\, \int_{\mathbb{R}^2} \frac{(x_1-y_1)^2-(x_2-y_2)^2}{|x-y|^4} \omega(y)dy~
	\end{split}
	\end{equation*} and a constant multiple of $\omega(x)$. 
	
	Let us first deal with the first expression, restricting ourselves to the Lipschitz case $\alpha = 1$. We split $\mathbb{R}^2$ into the regions (i)$|x-y| \le l |x| $, (ii) $l|x| < |x-y| \le 2|x|$, and (iii) $2|x| < |x-y|$, where $l\le 1/2$ is a number to be chosen below. In the third region, we use symmetry of $\omega$ to gain integrable decay of the kernel, which results in a bound of the form $C\Vert \omega \Vert_{L^\infty}$. In the first region, we may rewrite  \begin{equation*}
	\begin{split}
	p.v.\,\int_{\mathbb{R}^2} \frac{(x_1-y_1)(x_2-y_2)}{|x-y|^4} (\omega(y) - \omega(x)) dy~,
	\end{split}
	\end{equation*} and the given Lipschitz bound implies \begin{equation*}
	\begin{split}
	C l|x| \sup_{y: |x-y| \le l|x|} |\nabla \omega(y)| \le Cl  \sup |z| |\nabla\omega(z)|  ~,
	\end{split}
	\end{equation*} which together gives a bound \begin{equation*}
	\begin{split}
	p.v.~ \int_{ |x-y| \le l|x|} \frac{C}{|x||x-y|} \V |\cdot| \nabla\omega(\cdot) \V_{L^\infty}  dy \le Cl \V |\cdot| \nabla\omega(\cdot) \V_{L^\infty}~.
	\end{split}
	\end{equation*}
	Lastly, in the second region we directly integrate to obtain a bound \begin{equation*}
	\begin{split}
	C \Vert \omega \Vert_{L^\infty} \ln\left( \frac{c}{l} \right)~.
	\end{split}
	\end{equation*} Optimizing \begin{equation*}
	\begin{split}
	l := \min\left( \frac{1}{2}, \frac{\Vert \omega \Vert_{L^\infty} }{\sup |z| |\nabla\omega(z)|}  \right) 
	\end{split}
	\end{equation*} establishes the claimed bound. The other expression can be treated in a similar fashion. Then the $C^\alpha$ bound may be obtained in a parallel manner: one just use the H\"{o}lder assumption on the region (ii) to remove the singularity, and then optimize in $l$ accordingly.
\end{proof}

\begin{proof}[Proof of Theorem \ref{thm:smooth_away_from_0}]
	
	We first deal with the case $k = 0$. To establish the double exponential growth rate in time, it is most efficient to pass to the Lagrangian formulation directly (see the introduction of \cite{KS} for instance). 
	
	Starting with \begin{equation*}
	\begin{split}
	\frac{d}{dt} \Phi(t,x) = u(t,\Phi(t,x))~,
	\end{split}
	\end{equation*} we obtain \begin{equation*}
	\begin{split}
	\left| \frac{d}{dt} \left( \Phi(t,x)-\Phi(t,x')   \right)  \right| &\le \V \nabla u(t) \V_{L^\infty} \left|\Phi(t,x)-\Phi(t,x') \right| \\
	&\le \left(1+ \ln(1+ \V \omega(t)\V_{\mathring{C}^{0,\alpha}} )\right) \left|\Phi(t,x)-\Phi(t,x') \right| ~,
	\end{split}
	\end{equation*} assuming $\V \omega \V_{L^\infty}= 1$ for simplicity. Integrating, \begin{equation*}
	\begin{split}
	 e^{-\int_0^t 1+ \ln(1+ \V \omega(s)\V_{\mathring{C}^{0,\alpha}} ) dt} \le  \left|\frac{\Phi(t,x)-\Phi(t,x')}{x-x'} \right| \le e^{\int_0^t 1+ \ln(1+ \V \omega(s)\V_{\mathring{C}^{0,\alpha}} ) dt}~,
	\end{split}
	\end{equation*} and it is clear that the same upper and lower bounds are available for the inverse of the flow map $\Phi_t^{-1}$.
	
	Given the above bound, we estimate \begin{equation*}
	\begin{split}
	\frac{|\omega(t,x) - \omega(t,x')|}{|x-x'|^\alpha}~,
	\end{split}
	\end{equation*} assuming that $|x|$ and $|x'|$ are comparable; $|x'|/2 \le |x| \le 2|x'|$. This is possible thanks to Lemma \ref{lem:equiv} that we proved earlier. Then, from the transport formula $\omega(t,x) = \omega_0(\Phi_t^{-1}(x))$,
	\begin{equation*}
	\begin{split}
	|\omega(t,x) - \omega(t,x')| \le  \V \omega_0\V_{\mathring{C}^{0,\alpha}} \frac{|\Phi_t^{-1}(x) - \Phi_t^{-1}(x')|^\alpha}{\min(|\Phi_t^{-1}(x)|^\alpha,|\Phi_t^{-1}(x')|^\alpha)}~.
	\end{split}
	\end{equation*} Recall that $c\V\omega_0\V_{L^\infty} e^{-Ct} \le |\Phi_t^{-1}(x)| \le c\V\omega_0\V_{L^\infty} e^{Ct}$ and similar for $x'$. Therefore $|\Phi_t^{-1}(x)|$ and $|\Phi_t^{-1}(x')|$ are comparable up to a factor of $e^{Ct}$, and from the above bound for inverse particle trajectories, we obtain \begin{equation*}
	\begin{split}
	|x|^\alpha|\omega(t,x) - \omega(t,x')| \le C_\alpha |x - x'|^\alpha \V \omega_0\V_{\mathring{C}^{0,\alpha}} e^{c_\alpha \int_0^t  1+ \ln(1+ \V \omega(s)\V_{\mathring{C}^{0,\alpha}} ) dt }~.
	\end{split}
	\end{equation*} At this point it is not hard to show that the desired double exponential bound holds. 

	Now we set $k = 1$. It suffices to establish the corresponding estimate from Lemma \ref{lem:velocity_gradient_2} with $\nabla u, \omega$ replaced by $\nabla^2 u, \nabla \omega$ respectively. We argue as in the proof of Lemma \ref{lem:velocity_gradient_2} but $\nabla\omega$ is not rotationally symmetric anymore. Instead, we recall that to gain integrable decay in the kernel, it is sufficient for $\nabla\omega$ to be orthogonal with respect to $y_1 y_2$ and $y_1^2 - y_2^2$ on large circles. To this end, we simply compute \begin{equation*}
	\begin{split}
	\int_{B_0(R)} y_1y_2 \nabla\omega(y)dy = - \int_{B_0(R)} \begin{pmatrix}
	y_2 \\ y_1
	\end{pmatrix} \omega(y) dy + \int_{\partial B_0(R)} y_1y_2 \frac{1}{|y|} \begin{pmatrix}
	y_1 \\ y_2
	\end{pmatrix} \omega(y) dy = 0 ~
	\end{split}
	\end{equation*} for all $R > 0$, once we assume that $\omega(y)$ is rotationally symmetric for some $m \ge 4$. The case $k>1$ can be treated in a strictly analogous manner.
\end{proof}

For later use, let us show that indeed the sharp version of the $\mathring{C}^{0,\alpha}$-bound holds, so that $\nabla u$ in the above setting actually belongs to the space $\mathring{C}^{0,\alpha}$.

In the following, assume that $T$ is a singular integral operator defined on $\mathbb{R}^2$ by a degree $-2$ homogeneous kernel $P$ which has mean zero on circles and smooth away from the origin. Assume further that for $m \ge 3$, the $m$-fold symmetrization of $P$ gives integrable decay; more precisely, we require that there exist constants $C,c > 0$, so that \begin{equation}\label{eq:integrable_decay_kernel}
\begin{split}
\left|\sum_{i=1}^m P(x- O^i_{2\pi/m} y) \right | \le C \frac{|x|}{|y|^3}~,\qquad |y| \ge c|x|
\end{split}
\end{equation} holds. One may note that the Riesz kernels \begin{equation*}
\begin{split}
R_i(x-y) := \frac{x_i - y_i}{|x-y|^3}
\end{split}
\end{equation*} for $i \in \{ 1, 2 \}$ satisfy the above requirements. This will be used in our analysis of the SQG equation. 

We begin with  the $L^\infty$ bound: 

\begin{lemma}\label{lem:critical_L_infty}
	Let $g \in \mathring{C}^{0,\alpha}(\mathbb{R}^2)$ with some $0<\alpha<1$ and $m$-fold symmetric for some $m\ge 3$. Then,\begin{equation*}
	\begin{split}
	\V Tg \V_{L^\infty} \le C_{\alpha} \V g\V_{\mathring{C}^{0,\alpha}}~.
	\end{split}
	\end{equation*}
\end{lemma}

\begin{proof}
	For all $x \in \mathbb{R}$, we need a uniform bound on \begin{equation*}
	\begin{split}
	p.v.\,\int_{\mathbb{R}^2} P(x-y)(g(y) - g(x)) dy~.
	\end{split}
	\end{equation*} We simply split $\mathbb{R}^2$ into (i) $|x-y| \le |x|/2$, (iii) $|y| \ge c|x| $, and (ii) the remainder set, which we denote as $A$. In the first region, we use the H\"{o}lder assumption to get the bound \begin{equation*}
	\begin{split}
	\int_{|x-y| \le |x|/2} \frac{C}{|x-y|^{2-\alpha}} \frac{\V g\V_{\mathring{C}^{0,\alpha}}}{|x|^\alpha} dy \le C'_\alpha \V g\V_{\mathring{C}^{0,\alpha}}~.
	\end{split}
	\end{equation*} In (iii), we symmetrize in $y$ and by the assumption \eqref{eq:integrable_decay_kernel} on the kernel $P$, \begin{equation*}
	\begin{split}
	C \int_{|y| \ge c|x|}  \frac{|x|}{|y|^3} \V g \V_{L^\infty} dy \le C' \V g \V_{L^\infty}~.
	\end{split}
	\end{equation*} In the set $A$, \begin{equation*}
	\begin{split}
	\left|\int_{A} P(x-y)g(y)dy \right| \le \int_{A} |P(x-y)| |g(y)|dy \le \int_{|x|/2<|x-y| \le c'|x| } \frac{\V g \V_{L^\infty}}{|x-y|^2} dy \le C' \V g \V_{L^\infty}~.
	\end{split}
	\end{equation*} This finishes the proof. 
\end{proof}

\begin{remark}[The standard counterexample]
	Note that if $g$ is only $2$-fold symmetric, the lemma is not true even when $g$ is compactly supported. Take explicitly the case $g(x_1,x_2) = x_1 x_2/|x|^2 \cdot \chi(|x|)$ where $\chi(\cdot)$ is a smooth bump function supported on $B_0(2)$ and identically equal to 1 on $B_0(1)$, we have that $Tg(x_1,x_2) \approx C\log|x|x_1 x_2/|x|^2$ near the origin, where $T$ is the singular integral operator defined by the kernel $\nabla K$, where $K$ is the Biot-Savart kernel. It is therefore worth noting that even for this logarithmic function the following Lemma \ref{lem:critical_C_alpha} holds. 
\end{remark}

\begin{lemma}\label{lem:critical_C_alpha}
	Let $g \in \mathring{C}^{0,\alpha}(\mathbb{R}^2)$ with some $0<\alpha<1$ (and not necessarily rotationally symmetric). Then, for $|x| \le |x'|$, \begin{equation*}
	\begin{split}
	\left| Tg(x)-Tg(x')  \right| \le C_\alpha \V g\V_{\mathring{C}^{0,\alpha}} \frac{|x-x'|^\alpha}{|x|^\alpha}~,
	\end{split}
	\end{equation*}
	for some constant $C_\alpha$ depending only on $\alpha$ and diverging as $\alpha\rightarrow 0$ like ${\alpha}^{-1}$. 
\end{lemma}

\begin{proof} 
	We assume $x',x\in\mathbb{R}^2\backslash\{0\}$ satisfy $|x|\leq |x'|.$ We need to prove the following bound: 
	$$\Big|\int_{\mathbb{R}^2} P(x-y) (g(y)-g(x))dx-\int_{\mathbb{R}^2}P(x'-y) (g(y)-g(x'))dx\Big|\leq C_\alpha\V g\V_{\mathring{C}^{0,\alpha}}\frac{|x-x'|^\alpha}{|x|^\alpha}~.$$
	First note that 
	$$\int_{|x-y|\leq 3|x-x'|} |P(x-y) (g(y)-g(x))| + |P(x'-y) (g(y)-g(x'))|dy\leq C_{\alpha}\V g\V_{\mathring{C}^{0,\alpha}}\frac{|x-x'|^{\alpha}}{|x|^\alpha}~.$$
	This is due to the fact that in this region, $|g(x)-g(y)|\leq c\frac{\V g\V_{\mathring{C}^{0,\alpha}}}{|x|^\alpha}|x-y|^\alpha$ and similar for $x'$.
	
	Hence, we only\footnote{In fact, there is also the symmetric difference set $A:= \{ |y| \leq 10|x|\}\cap \{|x-y|\geq 3|x-x'|\}\Delta\{3|x-x'|\leq |x-y|\leq 10|x|\}.$ However, it is clear that $A\subset\{5|x|\leq |x-y|\leq 15|x|\}$, so $\int_{A}P(x'-y)(g(x')-g(y))-\int_{A} P(x-y)(g(x)-g(y))dy$ can easily be shown to be less than $C\V g\V_{\mathring{C}^{0,\alpha}}\frac{|x-x'|^\alpha}{|x|^\alpha}.$} need to prove that \begin{equation*}
	\begin{split}
	\Big|\int_{3|x-x'|\leq|x-y|} P(x-y) (g(y)-g(x))  - P(x'-y) (g(y)-g(x')) dy \Big|\leq C_\alpha\V g\V_{\mathring{C}^{0,\alpha}}\frac{|x-x'|^\alpha}{|x|^\alpha}~.
	\end{split}
	\end{equation*}
	We split \begin{equation*}
	\begin{split}
	&\Big|\int_{3|x-x'|\leq|x-y| } P(x-y) (g(y)-g(x))  - P(x'-y) (g(y)-g(x')) dy \Big|  \\
	&\qquad \leq\Big|\int_{3|x-x'|\leq|x-y| } (P(x-y)-P(x'-y))(g(y)-g(x'))dy\Big| + \\
	&\qquad \qquad + \Big|\int_{3|x-x'|\leq|x-y| } P(x-y)(g(x')-g(x))dy\Big|~.
	\end{split}
	\end{equation*} 
	But the last integral is 0 due to the fact that the domain of integration is spherical. For the remaining term, using the mean value theorem, we see that $$|P(x-y)-P(x'-y)|\leq C\frac{|x-x'|}{|x-y|^3},$$ since $$|x-y|\approx |x'-y|$$ in the region of integration. Hence, \begin{equation*}
	\begin{split}
	&\Big|\int_{3|x-x'|\leq|x-y| } (P(x-y)-P(x'-y))(g(y)-g(x'))dy\Big| \\
	&\qquad \le\int_{3|x-x'|\leq |x-y|}\frac{|x-x'|}{|x-y|^3}|g(y)-g(x')|dy~.
	\end{split}
	\end{equation*} We then use \eqref{eq:equiv2} from Lemma \ref{lem:equiv}, $|x-y| \approx |x' - y|$, and $|x| \le |x'|$ to bound \begin{equation*}
	\begin{split}
	|g(y) - g(x')| \le \frac{|x'-y|^\alpha}{|x'|^\alpha} \V g\V_{\mathring{C}^\alpha} \le C \frac{|x -y|^\alpha}{|x |^\alpha} \V g\V_{\mathring{C}^\alpha}
	\end{split}
	\end{equation*} Then, we can bound \begin{equation*}
	\begin{split}
	\int_{3|x-x'|\leq |x-y|}\frac{|x-x'|}{|x-y|^3}|g(y)-g(x')|dy \le C \V g\V_{\mathring{C}^\alpha} \int_{3|x-x'|\leq |x-y|}\frac{|x-x'|}{|x-y|^3} \frac{|x -y|^\alpha}{|x |^\alpha} dy \le C \V g\V_{\mathring{C}^\alpha}\frac{|x -y|^\alpha}{|x |^\alpha} .
	\end{split}
	\end{equation*} Now we are done. 
\end{proof}

\begin{corollary}\label{cor:sharp_C_alpha}
	Let $g \in \mathring{C}^{0,\alpha}(\mathbb{R}^2)$ with some $0<\alpha<1$ and $m$-fold symmetric for some $m\ge 3$. Then, \begin{equation*}
	\begin{split}
	\V Tg \V_{\mathring{C}^{0,\alpha}}\leq C_{\alpha} \V g\V_{\mathring{C}^{0,\alpha}}~
	\end{split}
	\end{equation*} holds.
\end{corollary}

\begin{proof}
	This is a direct consequence of the above two lemmas. Indeed, for $0<|x| \le |x'|$, take the difference \begin{equation*}
	\begin{split}
	|x|^\alpha Tg(x) - |x'|^\alpha Tg(x') = \left(|x|^\alpha - |x'|^\alpha \right) Tg(x') - |x|^\alpha \left( Tg(x) - Tg(x') \right)~.
	\end{split}
	\end{equation*} The first term is under control using the previous $L^\infty$ bound. Then, from the latter lemma we obtain that \begin{equation*}
	\begin{split}
	|x|^\alpha \frac{| Tg(x) - Tg(x')|}{|x-x'|^\alpha}  \le C_\alpha \left(1 + \frac{|x-x'|^{1-\alpha}}{|x|^{1-\alpha}} \right)~,
	\end{split} 
	\end{equation*} and if $|x - x'| \le |x|$, we are done, and otherwise we again use the $L^\infty$ bound to finish the proof. 
\end{proof}

\begin{remark}
	In the usual formulation of the $2D$ Euler equation, we have \begin{equation*}
	\begin{split}
	\pr_t u + u\cdot\nabla u + \nabla p = 0, \qquad \nabla\cdot u = 0. 
	\end{split}
	\end{equation*} Taking the divergence of both sides, one obtains $\Delta p = 2 \pr_1 u_1 \pr_2u_2 - 2\pr_1 u_2\pr_2u_1$. Under the assumptions of Theorem \ref{thm:smooth_away_from_0}, it is not difficult to uniquely determine the pressure $p = p(u)$ at each time moment by requiring \begin{itemize}
		\item $p$ is $m$-fold symmetric around the origin,
		\item $|p(x)| \le C|x|^2$, and
		\item $\nabla^2 p \in \mathring{C}^\alpha(\mathbb{R}^2)$. 
	\end{itemize} 
\end{remark}

\section{Radially homogeneous solutions to the $2D$ Euler}\label{sec:homog}

In this section, we investigate the evolution of the class of radially homogeneous solutions to the $2D$ Euler equation. More concretely, we are interested in solutions which satisfy  \begin{equation}\label{eq:homogeneity}
\begin{split}
\omega(r,\theta) = \omega(1,\theta)~,\qquad u(r,\theta) = r u(1,\theta)~, \qquad \Psi(r,\theta) = r^2 \Psi(1,\theta)
\end{split}
\end{equation} for all time, again with some $m$-fold rotational symmetry where $m$ is greater than or equal to 3. A simple scaling analysis shows that the above factors of $r$ are the only possible set of degrees of homogeneity which may be propagated by Euler. Indeed, observe that if the velocity is radially homogeneous with degree 1, it maps any line through the origin to another such line, and this keeps the degree 0 homogeneity assumption for the vorticity. In turn, this ensures that the velocity remains homogeneous with degree 1, explicitly by the Biot-Savart law. 

Therefore, one sees that the $2D$ Euler equation reduces to a $1D$ dynamical system defined on the unit circle, which is not volume (length) preserving in general, and hence non-trivial. This $1D$ system is easily shown to be well-posed (in $L^\infty$ and $C^\alpha$ for instance), and the resulting solution provides the unique solution to Euler via \eqref{eq:homogeneity}.

A crucial property of this model is that the corresponding \textit{angular} velocity is smoother than the advected scalar by degree 2, which is striking in view of the relation $\nabla \times u = \omega$. Indeed, when we impose the radial homogeneity assumption on $2D$ Euler, the velocity vector field decomposes into the angular part and the radial part, and the latter is indeed only one degree smoother than the vorticity, but of course the $1D$ model is not affected by the radial velocity at all; see the expression in \eqref{eq:vv}. 

In the following, we first introduce the system and collect a few simple a priori bounds. Then we proceed to check that its solution actually gives a radially homogeneous solution to the $2D$ Euler equation.

\subsection{The $1D$ system}

Consider the following transport system defined on the circle $S^1 =\{ \theta : -\pi \le \theta < \pi \}$: \begin{equation}\label{eq:1D}
	\begin{split} 
		\left\{  
		\begin{array}{rl}
			\pr_t h(t,\theta) +  2 H(t,\theta) \partial_\theta h(t,\theta) &= 0~, \\
			h(0,\theta) &= h_0(\theta)~,
		\end{array}
		\right.
	\end{split}
\end{equation} with $h_0 \in L^\infty(S^1)$ and $m$-fold rotationally symmetric for some $m \ge 3$. Here, $H(t,\cdot)$ is the unique solution of \begin{equation}\label{eq:inverse_Laplacian}
\begin{split}
h(t,\theta) = 4H(t,\theta) + H''(t,\theta)~, \qquad \frac{1}{2\pi
}\int_{-\pi}^{\pi} H(\theta) \exp(\pm 2i\theta)d\theta =  0~.
\end{split}
\end{equation}  Hence $H$ has the same rotational symmetry as $h$, and therefore a solution to \eqref{eq:1D} stays $m$-fold rotationally symmetric for all $t > 0$. 

Here and in the following, given an integrable function on $S^1 = [-\pi,\pi)$, we use the convention that its Fourier coefficients are given by the formula \begin{equation*}
\begin{split}
\hat{f}_k = \frac{1}{2\pi} \int_{-\pi}^{\pi} f(\theta) \exp(ik\theta)d\theta~.
\end{split}
\end{equation*}

Before we state and prove a few a priori inequalities, let us demonstrate that the transformation in \eqref{eq:inverse_Laplacian} which sends $h$ to $H$ is given by a simple and explicit kernel. 

\begin{lemma}[Biot-Savart law]
	We have \begin{equation}\label{eq:Biot_Savart_1D}
	\begin{split}
	H(\theta) = \frac{1}{2\pi} \int_{-\pi}^{\pi} K_{S^1}(\theta-\bar{\theta}) h(\bar{\theta}) d\theta~,
	\end{split}
	\end{equation} where \begin{equation}\label{eq:Biot_Savart_kernel_1D}
	\begin{split}
	K_{S^1}(\theta) := \frac{\pi}{2} \sin(2\theta)  \frac{\theta}{|\theta|} - \frac{1}{2} \sin(2\theta)\theta - \frac{1}{8}\cos(2\theta)~.
	\end{split}
	\end{equation}
\end{lemma}

\begin{proof}
	It suffices to check that the function $K_{S^1}$ has Fourier coefficients $1/(4-k^2)$ for $|k| \ne 2$, and 0 when $|k| = 2$. This is a simple computation. Indeed, one can easily arrive at this expression by observing that \begin{equation*}
	\begin{split}
	\frac{1}{4-k^2} = \frac{1}{4} \left( \frac{1}{2-k} + \frac{1}{2+k} \right)
	\end{split}
	\end{equation*} and that $\pi \mbox{sign}(\theta) - \theta$ has Fourier coefficients $1/(ik)$. 
\end{proof}

\begin{figure}
	\includegraphics[height=60mm]{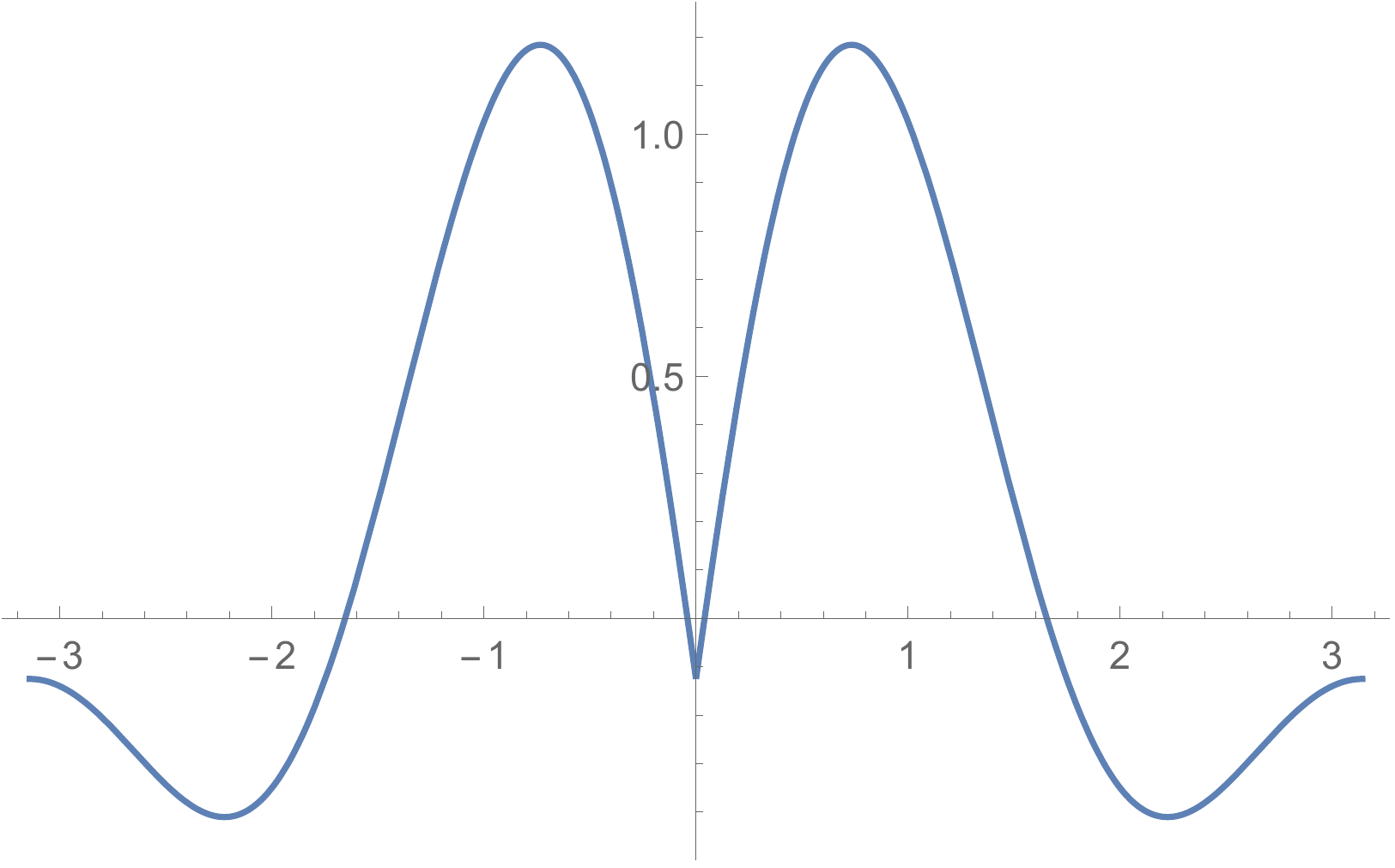}
	\centering
	\caption{The graph of $K_{S^1}$}
	\label{fig:kernel}
\end{figure}

\begin{remark}
	The expression \eqref{eq:Biot_Savart_1D} is truly the Biot-Savart law for the $2D$ Euler equation restricted to radially homogeneous data. Actually, it is possible to derive the expression \eqref{eq:Biot_Savart_kernel_1D} by starting with the usual Biot-Savart law (where $x \in \mathbb{R}^2$ corresponds to $(r,\theta)$): \begin{equation*}
	\begin{split}
	H({\theta}) = \frac{1}{2} u(x) \cdot \frac{1}{r} \begin{pmatrix}
	-\sin{\theta}  \\ \cos{\theta}
	\end{pmatrix} = \frac{1}{4\pi} p.v. \, \frac{1}{r} \int_{ \mathbb{R}^2} \frac{(x-y)^\perp}{|x-y|^2} \omega(y)dy \cdot \begin{pmatrix}
	-\sin{\theta}  \\ \cos{\theta}
	\end{pmatrix}~, 
	\end{split}
	\end{equation*} re-writing in polar coordinates, and then explicitly integrating out the radial variable.
\end{remark}

\begin{lemma}[a priori inequalities]
	Assume that $h(t,\theta)$ is a smooth solution to \eqref{eq:1D}. For each fixed time, we have \begin{equation*}
	\begin{split}
	\Vert H \Vert_{W^{n+2,\infty}} \le C_n \V h \V_{W^{n,\infty}}~,\qquad \Vert H \Vert_{W^{n+1,\infty}} \le C_n \V h \V_{W^{n,1}}~.
	\end{split}
	\end{equation*} Next, \begin{equation*}
	\begin{split}
	\frac{d}{dt}\V h(t)\V_{L^\infty} =  0~, \qquad \left|\frac{d}{dt}\V h(t)\V_{L^1} \right| \le  C\V h_0\V_{L^\infty} \V h(t)\V_{L^1} ~,\qquad \left|\frac{d}{dt} \V h(t)\V_{L^1}\right| \le c \V h(t)\V_{L^1}^2~.
	\end{split}
	\end{equation*} Moreover, if initially the data is nonnegative, then $\V h(t) \V_{L^1} = \V h_0\V_{L^1}$. Lastly, the gradient may grow at most exponentially in time; \begin{equation*}
	\begin{split}
	\left| \frac{d}{dt} \V\partial_\theta h(t)\V_{L^\infty} \right| \le C \V h_0\V_{L^\infty}  \V\partial_\theta h(t)\V_{L^\infty}~.
	\end{split}
	\end{equation*}
\end{lemma}

\begin{proof}
	We check the first statement for $n = 0$. The case $n > 0$ can be treated similarly. From \eqref{eq:Biot_Savart_1D}, it is clear that \begin{equation*}
	\begin{split}
	\V H'\V_{L^\infty} \le \pi \V K_{S^1}\V_{Lip}  \V h\V_{L^1}~. 
	\end{split}
	\end{equation*} Since we have trivially $\V H \V_{L^\infty} \le C\V h\V_{L^1} \le C'\V h\V_{L^\infty}$, \begin{equation*}
	\begin{split}
	\V H''\V_{L^\infty} \le \V h \V_{L^\infty} + 4 \V H\V_{L^\infty} \le C \V h \V_{L^\infty}~.
	\end{split}
	\end{equation*}
	
	We define the flow map on $S^1$, by solving \begin{equation*}
	\begin{split}
	\frac{d}{dt} \phi(t,\theta) =  2H(t, \phi(t,\theta) ) ~, \qquad \phi(0,\theta) = \theta~ . 
	\end{split}
	\end{equation*} We therefore have \begin{equation*}
	\begin{split}
	\frac{d}{dt} h(t,\phi(t,\theta)) = 0 ~,
	\end{split}
	\end{equation*} and the $L^\infty$ norm of $h$ is conserved in time. To derive statements regarding its $L^1$ norm, we multiply both sides of \eqref{eq:1D} by $\mbox{sign}(h)$ and integrate in space to obtain \begin{equation*}
	\begin{split}
	\frac{d}{dt} \int_{S^1} |h(t,\theta)| d\theta = - \int_{S^1} 2H \partial_\theta |h| d\theta =  \int_{S^1} 2H' |h| d\theta~,
	\end{split}
	\end{equation*} and the inequalities follow from our previous bounds. If $h_0 \ge 0$, then non-negativity is preserved by the flow, and integrating \eqref{eq:1D} in space gives that \begin{equation*}
	\begin{split}
	\frac{d}{dt}\int_{S^1}h(t,\theta)d\theta = \int_{S^1} 2H' h d\theta = \int_{S^1}  8H'H + 2 H'H'' = \int_{S^1} 4 (H^2)' - \left( (H')^2 \right)' d\theta= 0~,
	\end{split}
	\end{equation*} which shows that the $L^1$ norm is conserved. This also states that we may assume that the mean of $h_0$ is zero, without loss of generality. The dynamics in the general case may be recovered by means of the transformation \eqref{eq:rotating_solution}. 
	
	Regarding the last statement, by differentiating \eqref{eq:1D}, we obtain \begin{equation}\label{eq:gradient}
	\begin{split}
	\left(\partial_t + 2H(\theta)\partial_\theta\right) h'(\theta) = -2H'(\theta)h'(\theta)~.
	\end{split}
	\end{equation} Composing with the flow map and taking the $L^\infty$-norm of both sides gives \begin{equation*}
	\begin{split}
	\frac{d}{dt} \Vert h'(t,\cdot)\Vert_{L^\infty} \le 2 \Vert H'(t,\cdot)\Vert_{L^\infty}  \Vert h'(t,\cdot)\Vert_{L^\infty}~.
	\end{split}
	\end{equation*} 
	
	This finishes the proof.
\end{proof}

\begin{remark}
	The flow homeomorphisms $\Phi(t,\cdot): \mathbb{R}^2 \rightarrow \mathbb{R}^2$ for all $t\ge 0$ are well-defined as bi-Lipschitz maps in this setting. Note also that the $1D$ flow $\phi(t,\cdot)$ is ``volume preserving'' if and only if $0 = H'(\theta)$, and hence the $L^p$-norms for $h$ will not be conserved for $p < \infty$ in general.
\end{remark}

\begin{proposition}
	The system \eqref{eq:1D} is well-posed in $h \in L^\infty([0,\infty),L^\infty)$ and in $L^\infty([0,\infty),C^{k,\alpha})$ for any $k \ge 0$ and $0 \le \alpha \le 1$. The solution gives the unique solution to $2D$ Euler by setting \begin{equation}\label{eq:vv}
	\begin{split}
	\omega(t,r,\theta) &= h(t,\theta)~, \\
	u(t,r,\theta) &= 2H(t,\theta) \begin{pmatrix}
	-r\sin\theta \\ r\cos\theta
	\end{pmatrix} - H'(t,\theta) \begin{pmatrix}
	r\cos\theta \\ r\sin\theta
	\end{pmatrix}~, \\
	\Psi(t,r,\theta) &= r^2 H(\theta)~.
	\end{split}\end{equation}
\end{proposition}
\begin{proof}
	It is not difficult to show that the system is globally well-posed, first in the case $h_0 \in L^\infty$. Indeed, $h \in L^\infty$ guarantees that $H$ is Lipschitz (indeed, even $H'$ is Lipschitz), which allows us to solve for the flow map. A standard iteration scheme will give existence. The resulting solution is global thanks to the conservation of $h$ in $L^\infty$. Uniqueness can be shown along similar lines. Given the $L^\infty$ well-posedness, the corresponding statement in $C^{k,\alpha}$ is direct to verify.
	
	We now check that it provides a solution to Euler. Indeed, from \eqref{eq:vv}, we see that the equation \begin{equation*}
	\begin{split}
	\partial_t \omega + (u\cdot\nabla)\omega = 0~
	\end{split}
	\end{equation*} simply reduces to \eqref{eq:1D}. Moreover, a direct computation gives \begin{equation}\label{eq:curl_v}
	\begin{split}
	\nabla \times u = 4H + H'' = \omega~, \qquad \nabla \cdot u = 0~, \qquad u = \nabla^\perp \Psi~.
	\end{split}
	\end{equation} This establishes the statement, as Theorem \ref{thm:main} states that the $2D$ Euler solution with such a data is unique. 
\end{proof}

\begin{example}
	We begin by noting that the examples in Section \ref{sec:EU} defines either a stationary or a purely rotating solution to the system \eqref{eq:1D}. 
	
	The conservation of $L^1$ for non-negative data gives more examples of rotating solutions. Take some interval of length $L$ less than $2\pi/m$, and place them $m$-fold rotationally symmetric around the circle. Then, these patch solutions simply rotates. In general, if we add more intervals, then they may ``exchange'' lengths.
\end{example}

\subsection{Trend to equilibrium under odd symmetry and positivity}\label{subsec:sp}

Let us now study the evolution of the $1D$ system \eqref{eq:1D} in more detail. For concreteness, we will assume from now on that the vorticity is 4-fold rotationally symmetric: \begin{equation*}
\begin{split}
h_0(\theta) = h_0(\theta + \pi/2) = h_0 (\theta + \pi) = h_0 (\theta + 3\pi/2)~.
\end{split}
\end{equation*} Therefore, we may view the solution as defined on $[-\pi/4,\pi/4]$ as a periodic function. It turns out that once we impose \textit{odd symmetry} and  \textit{positivity}, we can get a fairly satisfactory picture of the long-time dynamics. That is, we assume \begin{equation*}
\begin{split}
h_0(\theta) = -h_0(-\theta)~,\qquad h_0 \ge 0 \mbox{ on } [0,\pi/4]~.
\end{split}
\end{equation*} First, it is easy to check from \eqref{eq:1D} and \eqref{eq:inverse_Laplacian} that the odd symmetry is going to be preserved, and since the endpoints of the interval $[0,\pi/4]$ are fixed, the solution remains non-negative. 

One motivation for imposing odd symmetry as well as positivity is that this scenario is expected to exhibit the fastest possible rate of gradient growth. Indeed, a very important consequence of the above extra assumptions is that there is a \textit{sign} for the velocity as well. This drives all the fluid particles from $(-\pi/4,\pi/4)$ towards the fixed point $\theta^* = 0$, which stretches the gradient at that point. In more detail, for fast gradient growth, it is necessary to have a lower bound on the velocity gradient, and $H'$ measured in the maximum norm scales as the $L^1$-norm of $h$, and the effect of oscillations in the sign of $h$ would be to simply reduce the magnitude of $H'$.  

We show that for such initial data $h_0$ which is supported near the fixed point $\theta^* = 0$, the sup-norm of the gradient $h'(t,\cdot)$ grows without bound as time goes to infinity, while the solution itself converges to the rest state in all $L^p(S^1)$ with $p < \infty$. This excludes the exponential growth rate, which is the optimal possible one.  To be clear, it does not follow that the exponential growth rate is impossible for any initial data. 

\begin{figure}
	\includegraphics[height=40mm]{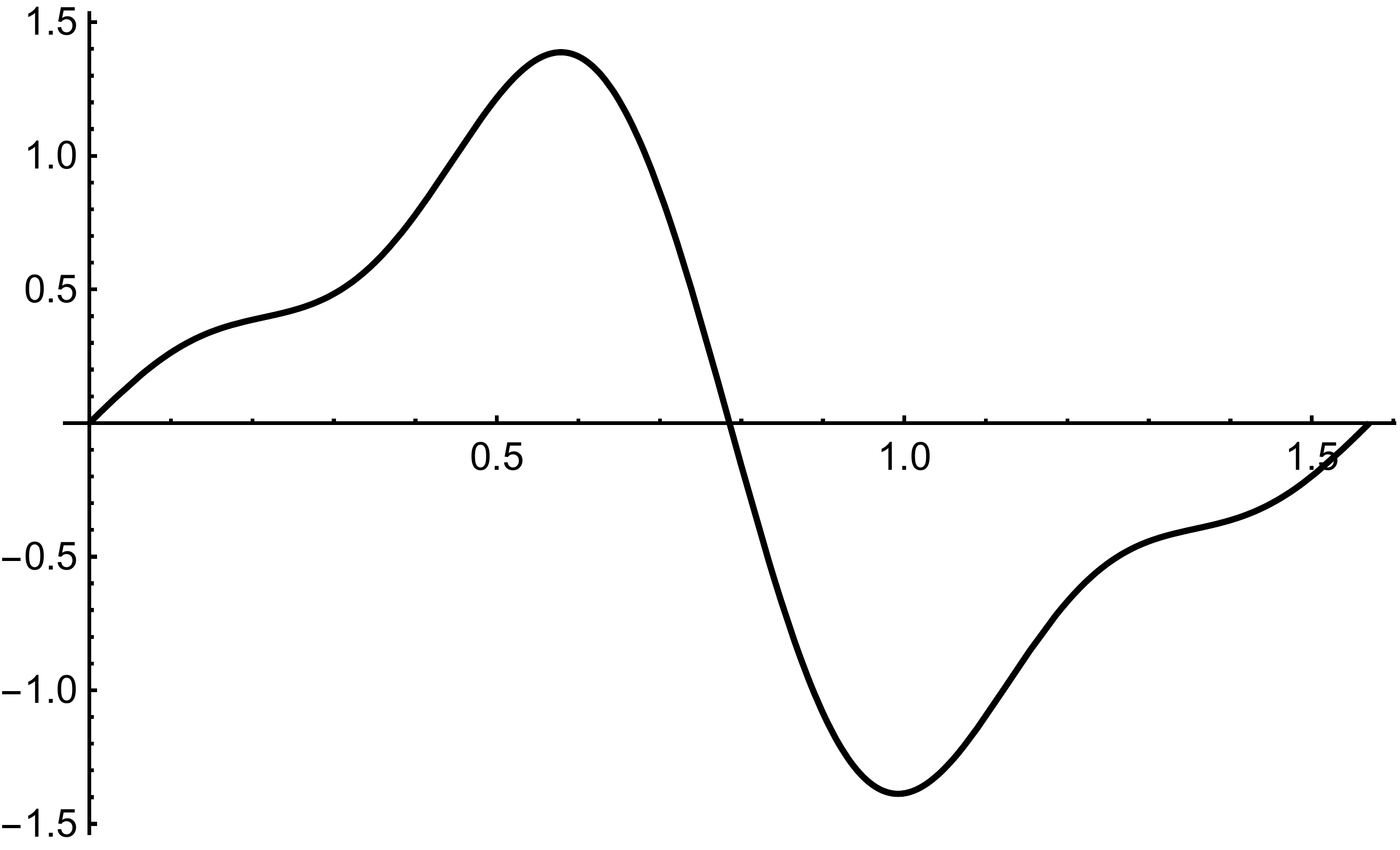}
	\includegraphics[height=40mm]{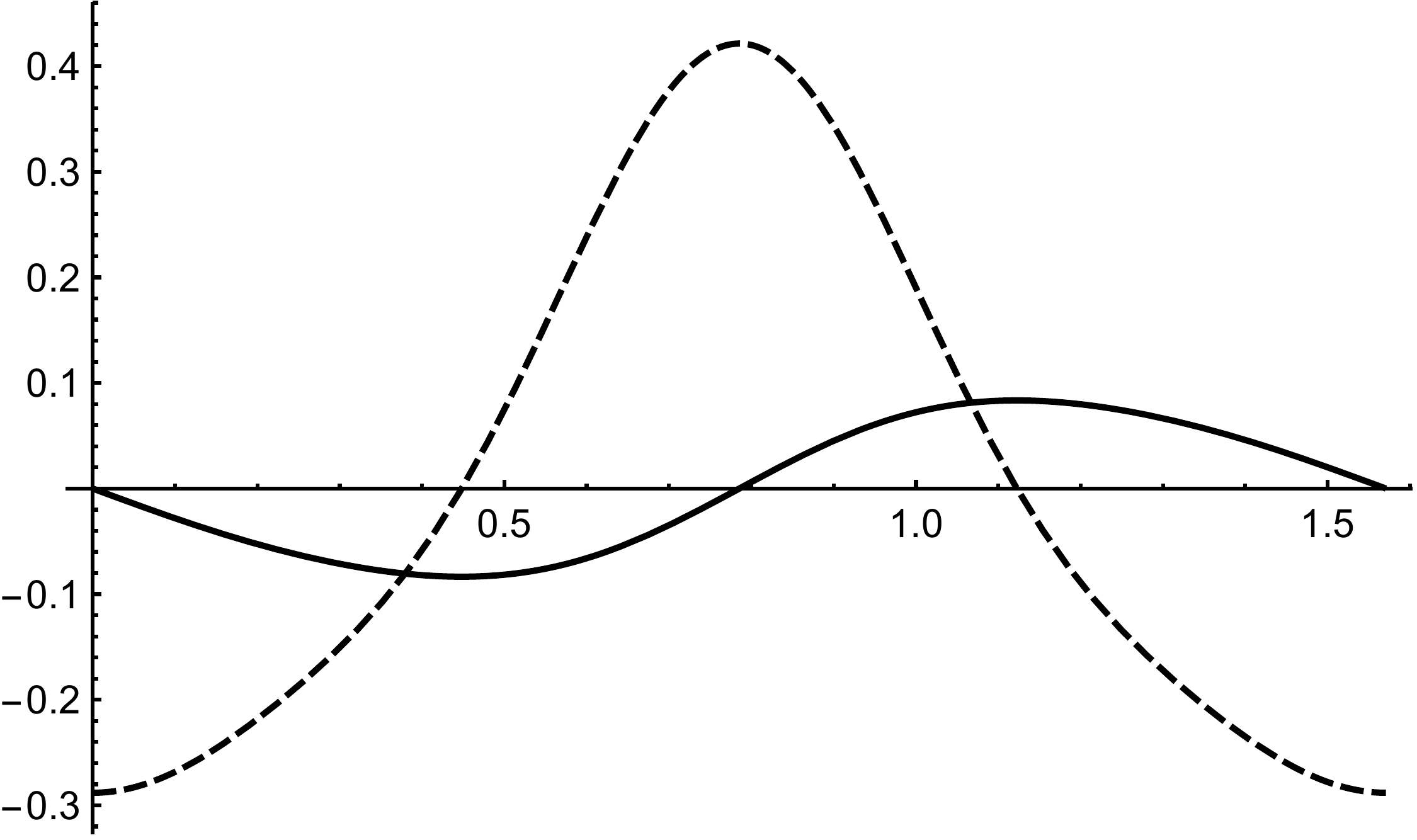}
	\centering
	\caption{Left: an example of ``odd and positive'' $h_0$, drawn on $[0,\pi/2]$. Right: associated graphs of $H_0$ (solid line) and $H_0'$ (dotted line), drawn on the same interval.}
	\label{fig:odd_positive}
\end{figure}

\begin{lemma}\label{lem:positivity}
	If $h$ is odd and non-negative on $[0,\pi/4]$, we have $H(t,\cdot) < 0$ and $H''(t,\cdot) > 0$ on $(0,\pi/4)$. In particular, $H'$ is strictly increasing on $(0,\pi/4)$ with $H'(0) < 0 < H'(\pi/4)$. 
\end{lemma}
\begin{proof}
	This immediately follows from a simple manipulation on the kernel \eqref{eq:Biot_Savart_kernel_1D}. First, using the 4-fold rotational symmetry, we may consider the symmetrized kernel \begin{equation*}
	\begin{split}
	K^1_{S^1} (\theta) =\frac{1}{4} \sum_{j=0}^3 K_{S^1}(\theta + j\pi/2) = \frac{\pi}{8} \left| \sin(2\theta)\right|~.
	\end{split}
	\end{equation*} Then, integrating  against an odd function $h$, \begin{equation*}
	\begin{split}
	H(\theta) = \frac{4}{2\pi}\int_{-\pi/4}^{\pi/4} K^1_{S^1}(\theta-\theta') h(\theta') d\theta = \frac{8}{2\pi}\int_{0}^{\pi/4} \frac{\pi}{8} \left( \left| \sin(2\theta- 2\theta')\right| - \left| \sin(2\theta + 2\theta')\right|  \right) h(\theta')d\theta'
	\end{split}
	\end{equation*} and then it is sufficient to note that the function $\left| \sin(2\theta- 2\theta')\right| - \left| \sin(2\theta + 2\theta')\right|$ is strictly negative on $(\theta,\theta') \in (0,\pi/4)^2$. Indeed, on $[0,\pi/4]$, $K(\theta,\theta') := \left| \sin(2\theta- 2\theta')\right| - \left| \sin(2\theta + 2\theta')\right|$ equals $-2\cos(2\theta)\sin(2\theta')$ for $\theta \ge \theta' $ and $K(\theta,\theta') = K(\theta',\theta)$. Hence, $H'' = h - 4H > 0$ on $(0,\pi/4)$. 
\end{proof}

As an immediate corollary, we have the following classification of stationary solutions: 

\begin{corollary}
	An odd, non-negative $h_0 \in L^\infty(S^1)$ defines a stationary solution of \eqref{eq:1D} if and only if it equals a constant on $(0,\pi/4)$. That is, there are no nontrivial such stationary solutions.
\end{corollary}
\begin{proof}
	The function $h_0 \in L^\infty(S^1) $ defines a stationary solution if and only if $H_0 h_0' = 0$ in the sense of distributions. Assuming that $h_0$ is not identically zero and non-negative on $(0,\pi/4)$, we have $H_0 < 0$ on $(0,\pi/4)$ by Lemma \ref{lem:positivity}, which means that $h_0$ equals a constant almost everywhere on this interval. 
\end{proof}

\begin{theorem}\label{thm:growth}
	Let $0 \ne h_0 \in W^{1,\infty}(S^1)$ be odd at 0 and non-negative on $[0,\pi/4]$. Assume further that the support of $h_0$ restricted to $[0,\pi/4]$ is contained in $[0,\epsilon]$ for some $0 < \epsilon \le 1/4$. Then we have the following bounds on the $L^1$-norm: \begin{equation}\label{eq:decay_L1_ep}
	\begin{split}
	\frac{1}{t  + \V h_0 \V_{L^1}^{-1} } \le \V h(t)\V_{L^1} \le \frac{1}{(1 - 8\epsilon^2)t + \V h_0 \V_{L^1}^{-1}},\quad \forall t \ge 0. 
	\end{split}
	\end{equation} In particular, the solution converges to the rest state in $L^1$ as $t \rightarrow +\infty$. On the other hand, the gradient grows almost quadratically for all times: \begin{equation}\label{eq:growth_gradient_ep}
	\begin{split}
	h'_0(0) \left(  \V h_0\V_{L^1} t + 1 \right)^{2(1-2\epsilon^2)} \le \V h'(t)\V_{L^\infty} \le \V h'_0\V_{L^\infty} \left( (1-8\epsilon^2)\V h_0\V_{L^1} t + 1 \right)^{2(1-8\epsilon^2)^{-1}}, \quad \forall t \ge 0,
	\end{split}
	\end{equation} assuming $h'_0(0) > 0$. 
\end{theorem}

\begin{proof}
	We shall use the formula \begin{equation}\label{eq:Hgradformula}
	\begin{split}
	H'(\theta) =  \sin(2\theta)\int_0^\theta h(\theta')\sin(2\theta')d\theta' - \cos(2\theta) \int_{\theta}^{\pi/4} h(\theta')\cos(2\theta')d\theta' 
	\end{split}
	\end{equation} which follows directly from differentiating the kernel expression for $H(\theta)$. From this it is clear from non-negativity of $h$ that \begin{equation*}
	\begin{split}
	\Vert H' \Vert_{L^\infty} = \max\left( -H'(0), H'(\pi/4)  \right)~. 
	\end{split}
	\end{equation*} Furthermore, under the assumption that the support of $h$ is contained in $[0,\epsilon]$ with $\epsilon \le 1/4$ (which is valid for all $t \ge 0$), \begin{equation*}
	\begin{split}
	- H'(t,0) = \int_0^{\pi/4} \cos(2\theta)h(t,\theta)  d\theta \ge \frac{4}{3}  \int_0^{\pi/4}  \sin(2\theta)h(t,\theta)  d\theta > H'(t,\pi/4)~,
	\end{split}
	\end{equation*} so $\V H'(t)\V_{L^\infty} = -H'(t,0)$ for all $t \ge 0$. Then, it is straightforward to see using \eqref{eq:Hgradformula} that \begin{equation*}
	\begin{split}
	\V h(t)\V_{L^1} \ge \V H'(t)\V_{L^\infty} \ge (1 - 2\epsilon^2) \V h(t)\V_{L^1},\quad \forall t \ge 0 . 
	\end{split}
	\end{equation*} Again using \eqref{eq:Hgradformula}, we also obtain \begin{equation*}
	\begin{split}
	0 \le H'(t,\pi/4) \le 2\epsilon \V h(t)\V_{L^1}. 
	\end{split}
	\end{equation*}
	
	Now integrating the system \eqref{eq:1D} on $(0,\pi/4)$  gives \begin{equation*}
	\begin{split}
	\frac{d}{dt} \int_0^{\pi/4} h(t,\theta)d\theta &= \int_0^{\pi/4} -2H(t,\theta) h'(t,\theta)d\theta =  \int_0^{\pi/4} 2H'(t)(4H(t) + H''(t)) d\theta \\
	&=   -(H'(t,0))^2  + (H'(t,\pi/4))^2   ~.
	\end{split}
	\end{equation*} Since \begin{equation*}
	\begin{split}
	(1 - 8\epsilon^2) \V h(t)\V_{L^1}^2 \le H'(t,0)^2 - H'(t,\pi/4)^2 \le \V h(t)\V_{L^1}^2,
	\end{split}
	\end{equation*} we obtain that \begin{equation*}
	\begin{split}
	-(1 - 8\epsilon^2)  \V h(t)\V_{L^1}^2   \ge \frac{d}{dt} \V h(t)\V_{L^1} \ge - \V h(t)\V_{L^1}^2 ~,
	\end{split}
	\end{equation*} and integrating in time gives \eqref{eq:decay_L1_ep}. 
	
	Then, to obtain the desired upper bound for $\V h'\V_{L^\infty}$, we note that for all $\theta \in (0,\pi/4)$ with $h'_0(\theta) \ne 0$ we have from \eqref{eq:gradient} and \eqref{eq:decay_L1_ep} that \begin{equation*}
	\begin{split}
	\left| \frac{d}{dt} \ln\left(  h'\circ \phi(t,\theta ) \right) \right| \le  2\V H'\V_{L^\infty} \le  2 \V h(t)\V_{L^1} \le \frac{2}{(1-8\epsilon^2)t + \V h_0\V_{L^1}^{-1}},\qquad \forall t \ge 0 . 
	\end{split}
	\end{equation*}  Integrating in time gives the second inequality of \eqref{eq:growth_gradient_ep}. The lower bound is simply obtained by evaluating the equation for $h'$ at 0 and using \eqref{eq:decay_L1_ep}: \begin{equation*}
	\begin{split}
	\frac{d}{dt} h'(t,0) = -2 H'(t,0) h'(t,0) \ge \frac{2 (1 - 2\epsilon^2) }{t + \V h_0\V_{L^1}^{-1}} h'(t,0).
	\end{split}
	\end{equation*} Integrating the above in time finishes the proof of \eqref{eq:growth_gradient_ep}. 
\end{proof}

\begin{remark}
	It is likely that a more careful analysis can establish similar statements for much more general initial data which is odd and non-negative on $(0,\pi/4)$, with less assumptions on the initial support. We do not dwell on this issue here.
\end{remark}

\subsubsection{Sharp gradient growth in the presence of a boundary}

Inspired by the work of Kiselev and Sverak \cite{KS}, let us demonstrate that when our domain has a boundary, we can achieve the sharp growth rate. The growth in our case actually occur away from the boundary but its role is to keep enough $L^1$-mass for all time, which in turn guarantees the uniform rate of growth.

We consider the problem \eqref{eq:1D} on the compact interval $Q := \{ -\pi/4 \le \theta \le \pi/4  \}$, where the endpoints are not identified with each other. If we again consider the class of initial data on $Q$ with odd symmetry around $\theta^* = 0$, then the system on $Q$ is well-posed in $h(t,\cdot) \in L^\infty(Q)$ simply because it exactly corresponds to solving \eqref{eq:1D} on the whole circle with the rotationally extended initial data \begin{equation*}
\begin{split}
\tilde{h}_0(\theta) = h_0(\theta + k\pi/2)~, \quad k \in \mathbb{Z}~.
\end{split}
\end{equation*} 
Similarly, it can be  shown that if $h_0 \in W^{n,\infty}(Q)$ then we have $(h, H) \in W^{n,\infty}(Q)\times W^{{n+2},\infty}(Q)$ for all time. Together with the odd symmetry assumption, this provides a unique solution to the $2D$ Euler equation on the sector \begin{equation*}
\begin{split}
\tilde{Q} := \{ (r,\theta) \in \mathbb{R}^2 : - \pi/4 \le \theta \le \pi/4 \}~,
\end{split}
\end{equation*} which satisfy the slip boundary conditions. Our uniqueness proof can be easily adapted to this setting, and note that the rotational symmetry assumption is even hidden in this case. Let us mention that the Yudovich theorem has been successfully extended to domains with polygonal corners, very recently; see \cite{co1,co2,co3}.

\begin{theorem}\label{thm:growth_boundary}
	Let $h_0 \in W^{1,\infty}(Q)$ be odd with respect to $\theta^* = 0$ and non-negative on $[0,\pi/4]$. If $h_0(\pi/4) > 0$, then $h(t,\cdot)$ converges to the odd stationary solution which equals $h_0(\pi/4)$ on $(0,\pi/4)$: \begin{equation*}
	\begin{split}
	\Vert h(t,\cdot) - h_0(0) \Vert_{L^1(0,\pi/4)} \rightarrow 0~, \quad t \rightarrow \infty~. 
	\end{split}
	\end{equation*}
	In addition, the solution exhibits the sharp growth rate of the gradient; 
	\begin{equation*}
	\begin{split}
	\V h'(t)\V_{L^\infty}  \ge C\exp( ct  )~,
	\end{split}
	\end{equation*} with some constant $c,C > 0$ depending only on the norm $\Vert h_0 \Vert_{W^{1,\infty}}$. 
\end{theorem}

\begin{proof}
	The fact that we have initially $h_0(\pi/4) > 0$ gives a global-in-time lower bound on the $L^1$-norm; indeed, one can find some triangle near the fixed point $\pi/4$ which lies below the graph of $h(t)$ for all $t \ge 0$. Indeed using this argument one has for all small $\delta > 0$ that \begin{equation*}
	\begin{split}
	\Vert h(t,\cdot) \Vert_{L^1(\delta,\pi/4-\delta)} \ge \epsilon(\delta) > 0~. 
	\end{split}
	\end{equation*} Since the kernel $K^1_{S^1}$ for $H$ is strictly negative on $(\delta,\pi/4 - \delta) \times (\delta,\pi/4-\delta)$, we obtain a global in time, positive lower bound for $-H(t,\theta)$ for $\theta \in [\delta,\pi/4-\delta]$. The velocity is simply $2H$ and therefore the particle starting from $ \pi/4-\delta$ reaches the point $ \delta$ after some finite time. Since $\delta > 0$ is arbitrary and $h(t)$ is continuous in space and uniformly bounded in $L^\infty$, convergence in $L^1$ to the stationary solution which equals $h_0(\pi/4)$ is established. 
	
	To obtain the exponential growth statement, take any $x \in (0,\pi/4)$ for which $h'_0(x) \ne 0$ and recall that \begin{equation*}
	\begin{split}
	\frac{d}{dt} h'(t,\phi(t,x)) = -2 H'(t,\phi(t,x)) h'(t,\phi(t,x))~.
	\end{split}
	\end{equation*} We know that for any $\delta > 0$ there is $T(\delta) > 0$ such that for all $t > T(\delta)$, $0 <\phi(t,x) <\delta$. Since the convergence in $L^1$ to the stationary solution implies \begin{equation*}
	\begin{split}
	-H'(t,0) = \int_0^{\pi/4} \cos(2\theta)h(t,\theta)d\theta \ge c h_0(\pi/4)~
	\end{split}
	\end{equation*} for all sufficiently large $t$ (and taking into account that $|H''(t,\theta)|$ is uniformly bounded in $t,\theta$), we deduce that $h'(t,\phi(t,x))$ grows exponentially in time. 
\end{proof}

%\begin{remark}
%	Consider the special case when the initial data $h_0 \in C^0(S^1) $ is monotonically increasing on the interval $[0,\pi/4]$. Then the solution is monotonic for all time, and the quantity \begin{equation*}
%	\begin{split}
%	L(t) = \int_0^{\pi/4} |h_0(\pi/4) - h(t,\theta)| d\theta 
%	\end{split}
%	\end{equation*} strictly decreases with time, i.e. $L$ defines a Lyapunov functional. This vanishes only for the stationary solution which equals a constant on $(0,\pi/4)$. From the continuity in time of the dynamical system (say in $L^1(S^1)$), the statements of Theorem \ref{thm:growth_boundary} immediately follows.
%\end{remark}

\begin{remark}
	The above result shows that the system is not globally stable in $L^1(S^1)$ (and similarly in all $L^p$ with $p<\infty$). That is, even if $h^{(1)}_0, h^{(2)}_0 \in C^0$ are arbitrarily close in $L^1$, it is clear that \begin{equation*}
	\begin{split}
	\V h^{(1)}(t) - h^{(2)}(t) \V_{L^1} \ge C \left| h^{(1)}_0(\pi/4) - h^{(2)}_0(\pi/4) \right|~
	\end{split}
	\end{equation*} for $t$ sufficiently large, where $C > 0$ is an absolute constant. 
\end{remark}

\begin{remark}
	We close this section by noting that there is really nothing special about the 4-fold symmetry assumption, and analogous results can be obtained for $m$-fold symmetric data with any $m \ge 3$. Let us only point out that when $m \ge 3$, we have \begin{equation*}
	\begin{split}
	\sum_{j=0}^{m-1} K_{S^1}(\theta + j 2\pi/m )  = c^1_m \left| \sin(m\theta/2) \right| + c^2_m~.
	\end{split}
	\end{equation*} with some constants $c^1_m, c^2_m$ depending only on $m$.
\end{remark}

\subsubsection{Growth of compactly supported solutions to $2D$ Euler}

As we have discussed in the introduction, the fact that we have solutions to the $1D$ system whose gradient grows almost immediately implies that there are compactly supported solutions to the $2D$ Euler equation whose gradient grows at least as fast as the $1D$ solutions do. 

\begin{corollary}
	For any $\epsilon>0$, there exists a $c>0$ and a 4-fold symmetric initial data $\omega_0 \in \mathring{C}^{0,1}(\mathbb{R}^2)$ whose unique solution to the $2D$ Euler equation grows in time with rate $t^{2-\epsilon}$: \begin{equation*}
	\begin{split}
	\V \omega(t)\V_{\mathring{C}^{0,1}(\mathbb{R}^2)} \ge ct^{2-\epsilon}~.
	\end{split}
	\end{equation*} If we consider the $2D$ Euler equation on the sector $\tilde{Q} = \{(r,\theta): -\pi/4 \le \theta \le \pi/4 \}$ with slip boundary conditions and vorticity odd in $\theta$, then there is initial data $\omega_0 \in \mathring{C}^{0,1}(\tilde{Q})$ which grows exponentially: \begin{equation*}
	\begin{split}
	\V \omega(t)\V_{\mathring{C}^{0,1}(\tilde{Q})} \ge c\exp(ct)~.
	\end{split}
	\end{equation*} In both cases, the solution can be compactly supported and the velocity is Lipschitz everywhere in space. 
\end{corollary}

We will only give a sketch of the proof since it essentially follows the same lines as the proof of Theorem \ref{thm:lwp_cbu}. For the second case of the corollary, we also need the global well-posedness of $2D$ Euler with vorticity in $C^{0,1}$ and $\mathring{C}^{0,1}$ on the sector $\tilde Q=\{(r,\theta): -\pi/4\le\theta\leq \pi/4\}$ (with the odd assumption on the vorticity). To obtain a priori estimates in either $C^{0,1}$ or $\mathring{C}^{0,1}$ of the vorticity, one just needs an $L^\infty$-bound on the velocity which follows immediately from Lemma 3.2 of \cite{EJB} by taking any $0 < \alpha < 1 $. Then it suffices to construct an approximating sequence satisfying the a priori estimates, and this can be done in a parallel manner with Theorem 1 of \cite{EJB} where we have treated the Boussinesq case (the Euler equations being the special case with trivial density function). The global existence follows from the logarithmic estimate (3.2) from Lemma 3.5 of \cite{EJB} which holds for $\alpha = 1$.  

\begin{proof}
	In both cases, take a $1D$ initial data $h_0 \in W^{1,\infty}$ from the $1D$ system \eqref{eq:1D} (on $S^1$ in the first case and $[-\pi/4,\pi/4]$ in the second case) which exhibits the desired growth rate of the gradient, and consider an initial data to the $2D$ Euler equation of the form \begin{equation*}
	\begin{split}
	\omega_0 = \omega^{2D}_0 +  h_0(\theta) \in \mathring{C}^{0,1}
	\end{split}
	\end{equation*} with any $\omega^{2D}_0 \in C^{0,1}(\mathbb{R}^2)$, which enjoys the same set of symmetries with $ h_0$. Here $\omega^{2D}_0$ may be chosen that $\omega_0$ is compactly supported. Then, there is a unique global-in-time solution to 2D Euler in the space $\mathring{C}^{0,1}$, and it is indeed straightforward to show that \begin{equation*}
	\begin{split}
	\omega^{2D}(t) := \omega(t) -   h(t)
	\end{split}
	\end{equation*} remains in $C^{0,1}(\mathbb{R}^2)$ for all time. (See for instance the proof of Theorem \ref{thm:lwp_cbu} where we establish this type of statement for the SQG equation.) Then, for any large $t > 0$, we have a pointwise inequality \begin{equation*}
	\begin{split}
	\left| |x|\nabla \omega(t,x) \right| \ge \left| |x| \nabla   h(t,\theta)  \right| - |x||\nabla \omega^{2D}(t,x)|
	\end{split}
	\end{equation*} and since $|\nabla\omega^{2D}(t,x)| \le C(t)$, taking the limit $|x| \rightarrow 0$ guarantees that \begin{equation*}
	\begin{split}
	\sup_{x \in \mathbb{R}^2 \backslash \{ 0\} } \left| |x|\nabla\omega(t,x) \right| \ge \frac{1}{2}\V\partial_\theta h(t) \V_{L^\infty(S^1)} 
	\end{split}
	\end{equation*} holds.
\end{proof}

\subsection{Measure-valued data and quasi-periodic solutions}

As it was noted in the previous sections, in the $1D$ system \eqref{eq:1D}, the active scalar just in $L^1$ was sufficient to guarantee that the velocity is Lipschitz. This implies that we can actually solve the equation with $L^1$ initial data, or even with finite signed measures $\mathcal{M}$. 

If we restrict to the class of atomic measures, i.e. measures supported on a finite set, then we obtain a well-posed dynamical system of point vortices. The resulting measure-valued solutions give vortex-sheet solutions to the $2D$ Euler equation. The associated velocity on $2D$ is not Lipschitz but only locally bounded in the radial direction, and it is unclear whether this is the unique solution in the class of measure-valued vorticity.

We describe the $1D$ system of point vortices. For simplicity, we keep the assumption that the vorticity is 4-fold symmetric, and describe the data only in an interval of length $\pi/2$ in $S^1$. 

\begin{proposition}
	Consider initial data \begin{equation*}
	\begin{split}
	h_0 (\theta) = \sum_{j=1}^N a_j \delta_{\theta^j_0}~,
	\end{split}
	\end{equation*} where $a_j \in \mathbb{R}$ are weights and $\theta^j_0 \in [0,\pi/2)$. The unique global in time solution of the following ODE system \begin{equation}\label{eq:point_vortex_ODE}
	\begin{split}
	\frac{d}{dt} \theta^j(t) = \frac{2}{\pi} \sum_{l=1}^N a_l  \sin (|2\theta^l(t) - 2\theta^j(t)|)~, \qquad \theta^j(0) = \theta^j_0~,\qquad j = 1, \cdots, N
	\end{split}
	\end{equation} gives the unique solution of \eqref{eq:1D} by setting \begin{equation*}
	\begin{split}
	h(t,\theta) = \sum_{j=1}^N a_j \delta_{\theta^j(t)}~.
	\end{split}
	\end{equation*}
\end{proposition}

We note that the point vortices cannot collide with each other since the conservation of the norm $\V h(t,\cdot) \V_{\mathcal{M}(S^1)}$ implies that the velocity is Lipschitz for all time. This also shows that point vortices can approach each other at an exponential rate. 

It will be convenient to assume that the points are distinct and ordered; \begin{equation*}
\begin{split}
\theta_0^1 < \theta_0^2 < \cdots < \theta_0^{N} ~, \qquad |\theta_0^{N} - \theta_0^1 | < \frac{\pi}{2}
\end{split}
\end{equation*} and then the ordering will be preserved by the dynamics. Since there is a mean rotation of the whole system, it is more efficient to study the distance between adjacent vortices; set \begin{equation*}
\begin{split}
z^j(t) := \theta^{j+1}(t) - \theta^j(t)~,\qquad j = 1, \cdots, N-1~. 
\end{split}
\end{equation*} Then we may re-write the ODE system \eqref{eq:point_vortex_ODE} in terms of $z^1,\cdots,z^{N-1}$. 

In the simplest case when there is only one point vortex (modulo rotational symmetry), then it simply rotates with a speed proportional to its weight. When we have two points $\theta^1$ and $\theta^2$ with equal weights (say 1), the difference $z^1 = \theta^2 - \theta^1$ does not change in time since the kernel is even. Therefore, two vortices rotate together, keeping the distance between them. 

Interestingly, once we add one more vortex, then the evolution of the distances $z^1$ and $z^2$ is given by a completely integrable Hamiltonian system. All orbits are periodic except for one stationary point, and any positive real number is achieved as a period of some orbit.

Since there is a constant speed mean rotation of the vortices, the resulting system for $\theta^1, \theta^2$, and $\theta^3$ is quasi-periodic in general. 

\begin{figure}
	\includegraphics[height=60mm]{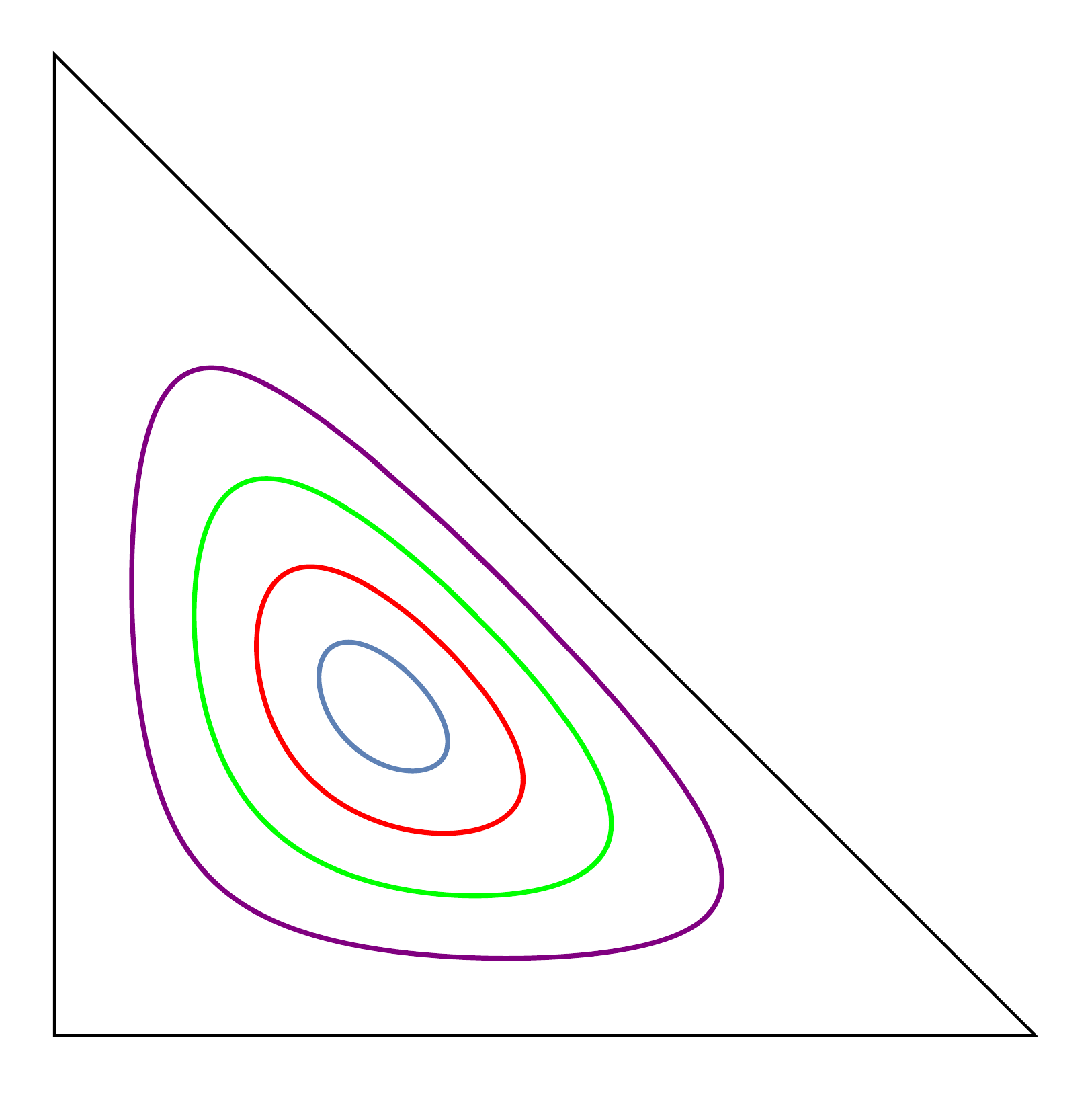}
	\centering
	\caption{A few trajectories for the system \eqref{eq:integrable}}
	\label{fig:periodic}
\end{figure}

\begin{theorem}[Quasi-periodic solutions]
	Given three point vortices of equal weight satisfying $0 \le \theta^1_0 < \theta^2_0 < \theta^3_0 < \pi/2$, the dynamics of gaps $z^1$ and $z^2$ is always periodic, and stationary only when $z^1 = z^2 = \pi/6$. This implies that generically the evolution of the triple $\{ \theta^1, \theta^2, \theta^3  \}$ undergoes a quasi-periodic motion.
\end{theorem}

\begin{proof}
	We have the system (by rescaling the time variable with some absolute constant if necessary) \begin{equation}\label{eq:integrable}
	\begin{split}
	\frac{d}{dt} z^1(t) &= \sin(2z^1 + 2z^2) - \sin(2z^2) \\
	\frac{d}{dt} z^2(t) &= \sin(2z^1)- \sin(2z^1 + 2z^2) ~.
	\end{split}
	\end{equation} The phase space is given by the pairs $(z^1,z^2)$ satisfying $ 0 < z^1, z^2 < \pi/2$ and $z^1 + z^2 < \pi/2$, and hence the sines are indeed non-negative. We then observe that the system is Hamiltonian with \begin{equation*}
	\begin{split}
	E(z^1,z^2) = \cos(2z^1) + \cos(2z^2) - \cos(2z^1 + 2z^2)~.
	\end{split}
	\end{equation*} This Hamiltonian takes its minimum value of 1 precisely on the boundary of the phase space $\{ z^1 = 0 \}, \{ z^2 = 0 \}$, and $\{ z^1 + z^2 = 0 \}$, and takes the unique maximum of 3/2 at the stationary point $(\pi/6,\pi/6)$. For any number $1 < \alpha < 3/2$, the set where $E = \alpha$ is a smooth curve inside the phase space, which must be the orbit of any initial data lying on the curve. This forces every orbit to be periodic.
\end{proof}

\begin{example}
	To clarify the type of solution we have obtained for the $2D$ Euler equation, consider on the circle the case of four point masses with the same weight, at points $\theta = 0, \pi/2, \pi, 3\pi/2$. Then, at $t = 0$, the resulting stream function and the velocity vector field on $\mathbb{R}^2$ is (a constant multiple of) \begin{equation*}
	\begin{split}
	\Psi_0(x_1,x_2) = |x_1||x_2|~, \qquad u_0(x_1,x_2) = \begin{pmatrix}
	-|x_1|\mbox{sign}(x_2) \\
	|x_2|\mbox{sign}(x_1)
	\end{pmatrix}
	\end{split}
	\end{equation*} so that the vorticity equals \begin{equation*}
	\begin{split}
	\omega_0(x_1,x_2) = |x_2|\delta_{\{ x_1 = 0 \} } + |x_1| \delta_{\{ x_2 = 0 \} }~,
	\end{split}
	\end{equation*} which is a locally finite measure with density growing as $|x|$. The solution simply rotates with constant angular speed. We emphasize that the dynamics is very different from the case of \begin{equation*}
	\begin{split}
	\omega(x_1,x_2) =\delta_{\{ x_1 = 0 \} } + \delta_{\{ x_2 = 0 \} }~,
	\end{split}
	\end{equation*}  which would form a spiral at the origin instantaneously.
\end{example}

\section{The SQG equation}\label{sec:SQG}

In this section, we consider the surface quasi-geostrophic (SQG) equation on $\mathbb{R}^2$: \begin{equation}\label{eq:SQG}
\begin{split}
\partial_t \Theta + u \cdot\nabla \Theta=0 ~, \qquad
u = \nabla^\perp (-\Delta)^{-1/2} \Theta~.
\end{split}
\end{equation} This system was suggested in works of Constantin, Majda, and Tabak \cite{CMT1,CMT2} as a two-dimensional mathematical model for geophysical atmospheric flows. It is very popular as a simpler model for the 3D Euler equation, because there are strong structural similarities between the two -- an observation from the original papers \cite{CMT1,CMT2}. Indeed, taking $\nabla^\perp$ of both sides of \eqref{eq:SQG}, \begin{equation*}
\begin{split}
\partial_t \nabla^\perp \Theta + (u\cdot\nabla) \nabla^\perp \Theta = \left[\nabla u \right] \nabla^\perp \Theta~,
\end{split}
\end{equation*} which resembles the 3D Euler equation once we put $\omega$ in place of $\nabla^\perp \Theta$. Although there are numerous interesting works on the issue of well-posedness of the SQG equation, it is an outstanding open problem to decide whether there exists a smooth solution which blows up in finite time.

\subsection{Local well-posedness in critical spaces}\label{subsec:lwp_sqg}

We show in this section that the system \eqref{eq:SQG} is locally well-posed in the scaling invariant spaces introduced in Section \ref{sec:EU}. 

It seems that in addition to the $m$-fold rotational symmetry for some $m\ge 3$, we need the active scalar $\Theta$ to be odd with respect to an axis. Under these symmetry assumptions, it is direct to verify the following uniqueness statement, along the lines of the proof of Lemma \ref{lem:unique_Poisson} which guaranteed that one can uniquely solve $-\Delta \Psi = \omega$ in the case of $2D$ Euler.

\begin{lemma}
	There is no nontrivial solution to the problem \begin{equation*}
	\begin{split}
	|\nabla| \Psi = 0~,
	\end{split}
	\end{equation*} assuming that $\Psi$ is $m$-fold symmetric with some $m\ge 3$, odd with respect to an axis, and satisfies the growth condition \begin{equation*}
	\begin{split}
	|\Psi(x)| \le C|x|^2~,\qquad x \in \mathbb{R}^2~.
	\end{split}
	\end{equation*}
\end{lemma}

In the local well posedness result below, the velocity is defined by $u = \nabla^\perp\Psi$, where $\Psi$ is the unique solution of $|\nabla|\Psi  = \Theta$. We show that (see below in Lemma \ref{lem:calpha_SQG}) $\nabla\Theta\in \mathring{C}^{0,\alpha}$ indeed implies $\nabla u \in \mathring{C}^{0,\alpha}$. 

\begin{theorem}
	Assume that the initial data $\Theta_0$ is $m$-fold rotationally symmetric for some $m \ge 3$, odd with respect to an axis, and $\nabla\Theta_0 \in \mathring{C}^{0,\alpha}(\mathbb{R}^2)$. Then, for some $T = T(\V \nabla\Theta\V_{\mathring{C}^{0,\alpha}(\mathbb{R}^2)}) > 0$, there is a unique solution $\nabla\Theta \in C([0,T);\mathring{C}^{0,\alpha}(\mathbb{R}^2))$ to the SQG equation with $\Theta(t,\cdot)$ being $m$-fold rotationally symmetric and odd.
\end{theorem}

Note that by definition of the spaces $\mathring{C}^{0,\alpha}$, we have $|\nabla\Theta(x)| \le C$ and hence we can treat initial data which grow as $|\Theta(x)| \approx C|x|$ and smooth away from the origin.

\begin{proof}
	We collect necessary a priori bounds. Taking the gradient of \eqref{eq:SQG} and composing with the flow map gives \begin{equation}\label{eq:SQG_gradient}
	\begin{split}
	\partial_t \left( \nabla\Theta\circ \Phi_t  \right) = -[\nabla u]^T \circ \Phi_t  \cdot  \left(\nabla\Theta\circ \Phi_t \right) ~.
	\end{split}
	\end{equation} The $L^\infty$ bound immediately follows: \begin{equation}\label{eq:Linfty_theta}
	\begin{split}
	\left| \frac{d}{dt} \V\nabla\Theta(t)\V_{L^\infty} \right| \le \V \nabla u(t)    \V_{L^\infty} \cdot \V\nabla\Theta(t) \V_{L^\infty} ~.
	\end{split}
	\end{equation} Similarly, we have the classical $L^\infty$ estimate for the flow and its inverse \begin{equation}\label{eq:Linfty_flow}
	\begin{split}
	\left| \frac{d}{dt} \V\nabla\Phi_t\V_{L^\infty} \right| \le \V \nabla u(t)    \V_{L^\infty} \cdot \V\nabla\Phi_t \V_{L^\infty}~, \qquad \left| \frac{d}{dt} \V\nabla\Phi_t^{-1}\V_{L^\infty} \right| \le \V \nabla u(t)    \V_{L^\infty} \cdot \V\nabla\Phi_t^{-1} \V_{L^\infty}~.
	\end{split}
	\end{equation}  % which ensures that $\V\nabla\Theta(t)\V_{L^\infty} \le C \exp(\int_0^t \V\nabla u(s)\V_{L^\infty}ds )$. 
	Note that under the presence of the $L^\infty$ estimate, for the purpose of estimating the $\mathring{C}^{0,\alpha}$-norms we may only consider points at a comparable distance; we will assume $0 < |x'| \le 2|x| \le 2|x'|$. Next, observe that the points $\Phi_t(x)$ and $\Phi_t(x')$ have comparable distance from the origin: the symmetry of $\Theta$ guarantees that the origin is a fixed point for all time $u(t,0) \equiv 0$ and hence \begin{equation*}
	\begin{split}
	\left|\frac{d}{dt} \Phi(t,x)\right| \le \V \nabla u(t)\V_{L^\infty} |\Phi(t,x)|~,
	\end{split}
	\end{equation*} so that \begin{equation*}
	\begin{split}
	|x| \exp( -\int_0^t \V\nabla u(s)\V_{L^\infty}ds ) \le |\Phi(t,x)| \le |x| \exp( \int_0^t \V\nabla u(s)\V_{L^\infty}ds )
	\end{split}
	\end{equation*} and similarly for $\Phi(t,x')$. On the other hand, given two comparable points $z, z'$ we can find for each $t > 0$ points $\Phi(t,x) = z$ and $\Phi(t,x') = z'$, so that \begin{equation*}
	\begin{split}
	\frac{\left|  |z|^\alpha \nabla\Theta_t(z) - |z'|^\alpha \nabla\Theta_t(z')    \right|}{|z-z'|^\alpha} &\le C(\V \nabla\Phi_t\V_{L^\infty}) \frac{\left|  |\Phi_t(x)|^\alpha \nabla\Theta_t(\Phi_t(x)) - |\Phi_t(x')|^\alpha \nabla\Theta_t(\Phi_t(x'))    \right|}{|x-x'|^\alpha}   \\
	&\le C(\V \nabla\Phi_t\V_{L^\infty})  \left( \frac{\left|  |x|^\alpha \nabla\Theta_t(\Phi_t(x)) - |x'|^\alpha \nabla\Theta_t(\Phi_t(x'))    \right|}{|x-x'|^\alpha} + \V \nabla\Theta_t \V_{L^\infty}    \right)
	\end{split}
	\end{equation*} which simply shows that the ``semi-norm'' \begin{equation*}
	\begin{split}
	Q(t) := \sup_{x\ne x',|x|\approx |x'|} \frac{\left| |x|^\alpha \nabla\Theta_t\circ \Phi_t(x) -  |x'|^\alpha \nabla\Theta_t\circ \Phi_t(x')  \right|}{|x-x'|^\alpha} 
	\end{split}
	\end{equation*} together with $\V \nabla\Theta_t \V_{L^\infty}$ controls the full norm $\V \nabla\Theta\V_{\mathring{C}^{0,\alpha}}$, up to multiplicative constants which may depend only on $\alpha$ and $\V \nabla\Phi_t\V_{L^\infty}$ (It is clear that an estimate $Q(t) \lesssim \V\nabla\Theta\V_{\mathring{C}^{0,\alpha}}$ holds).
	
	Multiplying \eqref{eq:SQG_gradient} by $|\cdot|^\alpha$, taking the difference evaluated at $x, x'$, and dividing by $|x-x'|^\alpha$ we obtain on the left hand side simply \begin{equation*}
	\begin{split}
	\partial_t \left( |x|^\alpha \nabla\Theta\circ \Phi_t(x) -  |x'|^\alpha \nabla\Theta\circ \Phi_t(x')      \right)~.
	\end{split}
	\end{equation*}

	On the right hand side, we need to estimate \begin{equation*}
	\begin{split}
	\frac{|x|^\alpha [\nabla u]^T \circ \Phi_t(x)  \cdot  \left(\nabla\Theta\circ \Phi_t(x) \right) - |x'|^\alpha [\nabla u]^T \circ \Phi_t(x')  \cdot  \left(\nabla\Theta\circ \Phi_t(x') \right)  }{|x-x'|^\alpha}~,
	\end{split}
	\end{equation*} and a simple computation gives a bound of the form \begin{equation*}
	\begin{split}
	C(\V \nabla \Phi_t\V_{L^\infty}) \left(  \V \nabla u \V_{\mathring{C}^{0,\alpha}} \V \nabla \Theta\V_{L^\infty} + \V \nabla u \V_{\mathring{C}^{0,\alpha}}\frac{\left| |x|^\alpha \nabla\Theta_t\circ \Phi_t(x) -  |x'|^\alpha \nabla\Theta_t\circ \Phi_t(x')  \right|}{|x-x'|^\alpha} \right)~.
	\end{split}
	\end{equation*} Taking the supremum over comparable $x,x'$ gives \begin{equation*}
	\begin{split}
	\left|\frac{d}{dt} Q(t)\right| \le C(\V \nabla \Phi_t\V_{L^\infty}) \V \nabla u \V_{\mathring{C}^{0,\alpha}} \left( \V \nabla \Theta\V_{L^\infty} + Q(t) \right)~.
	\end{split}
	\end{equation*} Assuming for a moment that the following bound \begin{equation}\label{eq:critical_bound}
	\begin{split}
	\V \nabla u \V_{\mathring{C}^{0,\alpha}} \le C_\alpha \V \nabla \Theta \V_{\mathring{C}^{0,\alpha}} 
	\end{split}
	\end{equation} holds, we obtain  \begin{equation*}
	\begin{split}
	\left|\frac{d}{dt} Q(t)\right| \le C(\V \nabla \Phi_t\V_{L^\infty})  \left( \V \nabla \Theta\V_{L^\infty} + Q(t) \right)^2~.
	\end{split}
	\end{equation*} and together with the $L^\infty$ estimates \eqref{eq:Linfty_theta} and \eqref{eq:Linfty_flow}, we can close the estimate in terms of $Q$, $ \V \nabla \Theta \V_{L^\infty}$, and $\V \nabla\Phi_t \V_{L^\infty}$, which establishes some a priori estimate on $ \V \nabla \Theta \V_{\mathring{C}^{0,\alpha}}$. 
	
	Given the a priori estimate, the velocity is Lipschitz, and one can solve for the flow maps. This allows one to build an iteration scheme analogously to the $2D$ Euler case in the previous section, and to prove uniqueness and existence for some time $t > 0$. It is clear that the time of existence can be extended as long as $ \V \nabla \Theta \V_{\mathring{C}^{0,\alpha}}$ remains finite. 
\end{proof}

It only remains to establish the estimate \eqref{eq:critical_bound}. A delicate point here is that $\nabla\Theta$ is not really rotationally symmetric for some integer $m'$. But recall that the real condition we need is that it is orthogonal to certain Fourier modes; it turns out that in this specific case, we only require orthogonality with respect to the \textit{first} modes:
\begin{lemma}
	For each $R > 0$, we have orthogonality \begin{equation*}
	\begin{split}
	\int_{\partial B(0,R)} y_1 \nabla\Theta(y)dy = 0 = \int_{\partial B(0,R)} y_2 \nabla\Theta(y)dy~.
	\end{split}
	\end{equation*}
\end{lemma}

\begin{proof}
	Let us focus on the left equality. It is equivalent to showing that \begin{equation*}
	\begin{split}
	\int_{ B(0,R)} y_1 \nabla\Theta(y)dy = 0 
	\end{split}
	\end{equation*} for all $ R > 0$. We simply integrate by parts to get \begin{equation*}
	\begin{split}
	- \int_{B(0,R)} \nabla y_1 \Theta(y) dy + \int_{\partial B(0,R)} y_1 \frac{1}{|y|} \begin{pmatrix}
	y_1 \\ y_2
	\end{pmatrix} \Theta(y) dy 
	\end{split}
	\end{equation*} and note that all the terms vanish by the symmetries of $\Theta(y)$.
\end{proof}

\begin{lemma}\label{lem:calpha_SQG} The following estimate 
	\begin{equation}
	\begin{split}
	\V \nabla u \V_{\mathring{C}^{0,\alpha}} \le C_\alpha \V \nabla \Theta \V_{\mathring{C}^{0,\alpha}} 
	\end{split}
	\end{equation} holds.
\end{lemma}
\begin{proof}
	Given the orthogonality condition, we immediately obtain an integrable decay in the Riesz kernel: \begin{equation*}
	\begin{split}
	p.v.~\int_{\mathbb{R}^2} \frac{(x_i - y_i)}{|x-y|^3} \nabla\Theta(y)dy = p.v.~\int_{\mathbb{R}^2} \left[ \frac{(x_i - y_i)}{|x-y|^3} + \frac{y_i}{|y|^3} \right] \nabla\Theta(y)dy
	\end{split}
	\end{equation*} for $i = 1, 2$, and the new kernel \begin{equation*}
	\begin{split}
	\frac{(x_i - y_i)}{|x-y|^3} + \frac{y_i}{|y|^3}
	\end{split}
	\end{equation*} decays as $|y|^{-3}$ for $|x| \le C|y|$  and $|y| \rightarrow \infty$. Now we proceed exactly as in the proof of Lemma \ref{lem:critical_L_infty} and Lemma \ref{lem:critical_C_alpha}, just using the decaying kernel in a region of the form $|x| \le C|y|$. 
\end{proof}

\subsection{The 1$D$ system for radially homogeneous SQG}

We now investigate the SQG equation \eqref{eq:SQG} under the radial homogeneity assumption, keeping the assumption that the active scalar $\Theta$ is rotationally symmetric. 

It will be instructive to consider the whole family of so-called modified SQG equation, which are active scalar equation defined on $\mathbb{R}^2$ via \begin{equation}\label{eq:alpha_SQG}
\begin{split}
\partial_t \Theta + u \cdot\nabla \Theta=0 ~, \qquad
u = \nabla^\perp (-\Delta)^{-\alpha} \Theta~,
\end{split}
\end{equation} with a parameter $0 \le \alpha \le 1$. The cases $\alpha =1 $ and $\alpha = 1/2$ correspond to the $2D$ Euler equation and the (usual) SQG equation, respectively. The corresponding velocity vector fields become more singular as $\alpha$ approaches 0. Note however that when $\alpha = 0$, the system is trivial. 

We seek solutions to the system \eqref{eq:alpha_SQG} which satisfies the homogeneity assumption \begin{equation}\label{eq:theta_ansatz}
\begin{split}
\Theta(t,r,\theta) = r^{2- 2\alpha} g(t,\theta)~
\end{split}
\end{equation} for some profile $g$ defined on the unit circle. The exponent $2-2\alpha$ on $r$ is the only possible one for the propagation in time of the ansatz \eqref{eq:theta_ansatz}. Furthermore, the divergence free assumption on velocity forces it to take the form \begin{equation}\label{eq:velocity_ansatz}
\begin{split}
u(t,r,\theta) = -2 G(t,\theta) r\begin{pmatrix}
-\sin\theta \\
\cos\theta
\end{pmatrix} +  \pr_\theta G(t,\theta) r  \begin{pmatrix}
\cos\theta\\
\sin\theta
\end{pmatrix}~, 
\end{split}
\end{equation} where $G(t,\cdot)$ is the angular part of the associated stream function: \begin{equation}\label{eq:SQG_stream}
\begin{split}
\Psi(r,\theta) = r^2 G(\theta)~.
\end{split}
\end{equation} Inserting \eqref{eq:theta_ansatz} and \eqref{eq:velocity_ansatz} into the system \eqref{eq:alpha_SQG}, we obtain the family of $1D$ active scalar systems \begin{equation}\label{eq:alpha_SQG_homog}
\begin{split}
\partial_t g - 2 G \pr_\theta g + (2-2\alpha) \pr_\theta G g = 0 ~. 
\end{split}
\end{equation} To complete the formulation, it suffices to express $G$ in terms of $g$. For simplicity, we now restrict ourselves to the SQG case of $\alpha = 1/2$. It turns out that we can even allow for the rotational symmetry with $m = 2$. 

\begin{lemma}
	Assume that $G$ and $g$ are orthogonal to the Fourier modes $\{ 1, \exp(i\theta), \exp(-i\theta) \}$ on $S^1 = [-\pi,\pi)$. We then have the following relation between $G$ and $g$; \begin{equation}\label{eq:Gandg}
	\begin{split}
	\widehat{G}_k = \frac{1}{-|k| - 3|k|/(k^2-1)} \hat{g}_k
	\end{split}
	\end{equation} for all $|k| \ge 2$.
\end{lemma}

\begin{proof}
	As in Castro-C\'{o}rdoba \cite{CaCo}, we use the representation formula for the operator $(-\Delta)^{1/2}$: \begin{equation*}
	\begin{split}
	(-\Delta)^{1/2}\Psi (x) = \frac{1}{2\pi} p.v.\, \int_{\mathbb{R}^2} \frac{\Psi(x) - \Psi(y)}{|x-y|^3} dy ~.
	\end{split}
	\end{equation*} Re-writing in polar coordinates and inserting the ansatz \eqref{eq:SQG_stream}, we obtain \begin{equation*}
	\begin{split}
	(-\Delta)^{1/2}\Psi (r,\theta) &= \frac{1}{2\pi} p.v.\, \int_{-\pi}^{\pi} \int_{0}^{\infty} \frac{r^2 G(\theta) - \bar{r}^2 G(\bar{\theta})}{(r^2 + \bar{r}^2 - 2r\bar{r} \cos(\theta-\bar{\theta}))^{3/2}} \bar{r} d\bar{r}d\bar{\theta}~ \\
	&= \frac{1}{2\pi} p.v.\,  \int_{-\pi}^{\pi} \int_{0}^{\infty} \frac{(r^2 G(\theta) - {r}^2 G(\bar{\theta})) + (r^2 G(\bar\theta) - \bar{r}^2 G(\bar{\theta}))}{(r^2 + \bar{r}^2 - 2r\bar{r} \cos(\theta-\bar{\theta}))^{3/2}} \bar{r} d\bar{r}d\bar{\theta}~.
	\end{split}
	\end{equation*} Regarding the first term, by setting $\bar{r} = \eta r $, we obtain \begin{equation*}
	\begin{split}
	\frac{1}{2\pi} \int_{-\pi}^{\pi} \int_{0}^{\infty} \frac{r^2 G(\theta) - {r}^2 G(\bar{\theta})}{(r^2 + \bar{r}^2 - 2r\bar{r} \cos(\theta-\bar{\theta}))^{3/2}} \bar{r} d\bar{r}d\bar{\theta} &= \frac{r}{2\pi}  \int_{-\pi}^{\pi} \int_{0}^{\infty} \frac{\eta (G(\theta)-G(\bar{\theta}))}{(1+\eta^2 - 2\eta \cos(\theta-\bar{\theta}) )^{3/2}} d\eta d\bar\theta \\
	&= \frac{r}{2\pi}  \int_{-\pi}^{\pi} \frac{G(\theta)-G(\bar{\theta})}{1-\cos(\theta-\bar{\theta})} d\bar{\theta} \\
	&= \frac{r}{2\pi} \int_{-\pi}^{\pi} \frac{G(\theta)-G(\bar{\theta})}{2\sin^2\left(\frac{\theta-\bar{\theta}}{2}\right)} d\bar{\theta} = - r (|\nabla| G)(\theta)~,
	\end{split}
	\end{equation*} where $|\nabla|$ is the square root of $-\Delta$ defined on the circle, noting that $|\nabla| = \frac{d}{d\theta} \mathcal{H}$ with $\mathcal{H}$ being the Hilbert transform on the circle which has the kernel $\cot(\theta/2)/2$.
	
	Turning to the next term, we first rewrite \begin{equation*}
	\begin{split}
	p.v.\,\int_{-\pi}^{\pi}  \int_{0}^{\infty} \frac{G(\bar{\theta}) (r^2-\bar{r}^2)}{(r^2 + \bar{r}^2 - 2r\bar{r} \cos(\theta-\bar{\theta}))^{3/2}} \bar{r} d\bar{r}d\bar{\theta} &= \lim_{R \rightarrow\infty} r \int_{-\pi}^{\pi} G(\bar{\theta}) \int_{R^{-1}}^{R}  \frac{\eta(1-\eta^2)}{(1+\eta^2 - 2\eta \cos(\theta-\bar{\theta}) )^{3/2}}d\eta d\bar{\theta}
	\end{split}
	\end{equation*} and by evaluating the $\eta$ integral explicitly, we obtain \begin{equation*}
	\begin{split}
	\lim_{R \rightarrow\infty} r \int_{-\pi}^{\pi} G(\bar{\theta})	\left[-\frac{3 - 6 a \eta + \eta^2}{\sqrt{1 - 2 a \eta + \eta^2}} - 3 a \log(-a + \eta + \sqrt{1 - 2 a \eta + \eta^2}) \right] \bigg|^R_{R^{-1}} d\bar{\theta}~,
	\end{split}
	\end{equation*} where we have set $a := \cos(\theta-\bar{\theta})$ for simplicity. Regarding the first term, note that as long as $G(\bar{\theta})$ is orthogonal to $1$ and $\exp(\pm i\bar{\theta})$, \begin{equation*}
	\begin{split}
	0 = p.v.~\int_{-\pi}^{\pi} \int_0^\infty \eta G(\bar{\theta}) d\eta d\bar{\theta} = p.v.~\int_0^{2\pi} \int_0^\infty \cos(\theta-\bar{\theta}) G(\bar{\theta}) d\eta d\bar{\theta} = p.v.~\int_{-\pi}^{\pi} \int_0^\infty 1 \cdot G(\bar{\theta}) d\eta d\bar{\theta}~,
	\end{split}
	\end{equation*} and therefore \begin{equation*}
	\begin{split}
	\int_{-\pi}^{\pi} G(\bar{\theta})	\left[-\frac{3 - 6 a \eta + \eta^2}{\sqrt{1 - 2 a \eta + \eta^2}}   \right] \bigg|^R_{R^{-1}} d\bar{\theta} \longrightarrow 0~, \qquad R \rightarrow \infty~.
	\end{split}
	\end{equation*} Then, the last term evaluated at $\eta = R \rightarrow \infty$ can be treated similarly, and for $\eta = R^{-1} \rightarrow 0$, we obtain the term \begin{equation*}
	\begin{split}
	r \int_{-\pi}^{\pi} G(\theta - \bar{\theta} ) 3 \cos( \bar{\theta}) \log\left( 1 - \cos( \bar{\theta})  \right)d\bar{\theta} 
	\end{split}
	\end{equation*} and relying upon the formula \begin{equation*}
	\begin{split}
	\int_{-\pi}^{\pi} 3\cos(m\bar\theta) \cos(\bar\theta) \log(1 - \cos(\bar\theta))d\bar\theta = -\frac{6\pi |m|}{m^2-1} ~,\qquad |m| \ge 2~,
	\end{split}
	\end{equation*} we finally deduce that the operator\begin{equation*}
	\begin{split}
	G(\theta) \mapsto \frac{1}{2\pi} p.v.~\int_{-\pi}^{\pi} \int_0^\infty \frac{G(\bar{\theta})\eta(1-\eta^2)}{(1+\eta^2 - 2\eta \cos(\theta-\bar{\theta}) )^{3/2}} d\eta d\bar\theta
	\end{split}
	\end{equation*} transforms $\exp(im\theta)$ into $-3|m|/(m^2-1) \exp(im\theta)$, given that $|m| \ge 2$. 
\end{proof}

The formula \eqref{eq:Gandg} which expresses the transformation $g \mapsto G$ as a Fourier multiplier shows that \begin{equation*}
\begin{split}
G = -|\nabla|^{-1}g + \mbox{ smoother term }.
\end{split}
\end{equation*}
Neglecting the smooth term for now, we obtain a simple evolution equation in terms of $g$: \begin{equation}\label{eq:SQG_homog_approx}
\begin{split}
\partial_t g - 2( |\nabla|^{-1}g  )  \pr_\theta g + \mathcal{H}(g) g = 0 ~. 
\end{split}
\end{equation} Interestingly, this particular $1D$ equation and closely related systems have been investigated recently by many authors, the main motivation being that such $1D$ systems model various higher dimensional systems arising in Fluid dynamics, including 3D Euler, SQG, Burgers, Vortex sheets, and so forth.

The system \eqref{eq:SQG_homog_approx} is actually a particular case ($a = 2$) of the following model introduced in Okamoto-Sakajo-Wunsch \cite{OSW}: \begin{equation}\label{eq:deGre}
\begin{split}
\pr_t f + a |\n|^{-1}f \pr_x f = \mathcal{H}(f) f~,\quad a \in \mathbb{R}~.  
\end{split}
\end{equation} These systems are sometimes called the De Gregorio models after \cite{DG1,DG2}, and it is not hard to show that they are locally well-posed.

In a paper of Castro and C{\'o}rdoba \cite{CaCo}, they observed that with the ansatz \begin{equation*}
\begin{split}
\Theta(t,x_1,x_2) = x_2 f(t,x_1)~, 
\end{split}
\end{equation*} if $f$ solves the De Gregorio model \eqref{eq:deGre} with $a = 1$, then one obtains a solution to the SQG equation. Indeed, under this stagnation-point similitude ansatz, it is straightforward to show that the stream function is given by \begin{equation*}
\begin{split}
\Psi(t,x_1,x_2) = -x_2 (\mathcal{H}F) (t,x_1)~.
\end{split}
\end{equation*} where $F$ is an antiderivative of $f$. In the same paper, they were able to show that in the De Gregorio model, finite time blow-up of smooth solutions occurs for $a < 0$. The problem of deciding whether finite time blow-up is possible for $ a > 0$ seems to be open. 

We note that the case $a = -1$ had been suggested as a toy model of SQG in \cite{CCF} and also as a model of the vortex sheet problem in \cite{BLM}. This is often called C{\'o}rdoba, C{\'o}rdoba, Fontelos system: in terms of the antiderivative $F$, one obtains a more familiar form \begin{equation*}
\begin{split}
\pr_t F + ( \mathcal{H}F ) \pr_x F = 0~.
\end{split}
\end{equation*}

We close this section by recording the local well-posedness statement for the $1D$ model. From the estimate \begin{equation*}
\begin{split}
\V G \V_{C^{k+1,\alpha}(S^1)} \le C_{k,\alpha} \V g \V_{C^{k,\alpha}(S^1)} 
\end{split}
\end{equation*} for $k \ge 0$ and $ 0 < \alpha <1 $ which follows from the explicit relation \eqref{eq:Gandg}, it is straightforward to establish local well-posedness of the 1D SQG in H\"{o}lder spaces:

\begin{proposition}
	The system \eqref{eq:alpha_SQG_homog} with $\alpha = 1/2$ is locally well posed in $g \in L^\infty([0,T);C^{k,\alpha}(S^1))$ with any $k \ge 0$ and $0<\alpha<1$, and $g(t,\theta)$ rotationally symmetric with some $m \ge 2$ and odd with respect to a point on the circle. 
\end{proposition}

\subsection{Blow-up condition for the SQG via $1D$ blow-up}

We define the $C^{1,\alpha}$-norm on $\mathbb{R}^2$ by \begin{equation*}
\begin{split}
\V f \V_{C^{1,\alpha}(\mathbb{R}^2)} := | f(0)| +  \V \nabla f\V_{L^\infty(\mathbb{R}^2)} + \V \nabla f \V_{C^\alpha_*(\mathbb{R}^2)}~.
\end{split}
\end{equation*} The $C^{1,\alpha}$-functions that we will deal with in this section will be indeed growing linearly in space in general. 

\begin{theorem}[Local well-posedness and conditional blow-up]\label{thm:lwp_cbu}
	Consider the class of initial data $\Theta_0$ on $\mathbb{R}^2$ which is $m$-fold rotationally symmetric, odd with respect to an axis, and has a decomposition \begin{equation*}
	\begin{split}
	\Theta_0(x_1,x_2) = \Theta_0^{2D}(x_1,x_2) + \Theta_0^{1D}(x_1,x_2)~,
	\end{split}
	\end{equation*} where $\Theta_0^{2D}(x_1,x_2) \in C^{1,\alpha}(\mathbb{R}^2)$ and $\Theta_0^{1D}(r,\theta) = r g_0(\theta)$ in polar coordinates where $g_0(\theta) \in C^{2,\alpha}(S^1)$. Then the following statements hold. 
	\begin{enumerate}
		\item (Local well-posedness) There exists some time $T = T(\V \Theta_0^{2D} \V_{C^{1,\alpha}(\mathbb{R}^2)}, \V g_0\V_{C^{2,\alpha}(S^1)}  ) > 0 $ such that there is a unique solution to the SQG equation satisfying a decomposition of the form \begin{equation*}
		\begin{split}
		\Theta(t)  = \Theta^{2D}(t) + \Theta^{1D}(t)
		\end{split}
		\end{equation*} where $\Theta^{2D} \in C([0,T);C^{1,\alpha}(\mathbb{R}^2))$ and $\Theta^{1D}(t) = r {g}(t,\theta)$ for $g\in C([0,T);C^{2,\alpha}(S^1))$ being the unique solution of the $1D$ system \eqref{eq:alpha_SQG_homog}, and $\Theta^{2D},\Theta^{1D}$  satisfy $m$-fold rotational symmetry together with an odd symmetry with respect to an axis. The unique solution could be continued past some time $T^* > 0$ if and only if \begin{equation}\label{eq:BKM}
		\begin{split}
		\int_0^{T^*} \V \nabla \Theta^{2D}(t) \V_{L^\infty} + \V  \partial_\theta  {g}(t,\theta) \V_{L^\infty} dt < + \infty~.
		\end{split}
		\end{equation}
		\item (Conditional blow-up result) Assume that there exists an initial data $g_0(\theta) \in C^{2,\alpha}(S^1)$ whose unique local-in-time solution to the $1D$ SQG system \eqref{eq:alpha_SQG_homog} blows up in $C^{2,\alpha}$ at some finite time $T^*$. Then, for any initial data $\Theta_0^{2D} \in C^{1,\alpha}(\mathbb{R}^2)$, the unique solution given in the above to the initial data $\Theta_0(x_1,x_2) = \Theta_0^{2D}(x_1,x_2) + r g_0(\theta)$ blows up at some finite time $0 < T \le T^*$. In particular the initial data $\Theta_0$ is Lipschitz in space and can be compactly supported. 
	\end{enumerate}
\end{theorem}

\begin{remark}
	Given that the 2D part $\Theta^{2D}$ belongs to $C^{1,\alpha}(\mathbb{R}^2)$, the assumption that the $1D$ part of the solution has angular regularity of $g \in C^{2,\alpha}(S^1)$ is simply the minimal assumption, and one can impose any higher regularity, e.g. we may assume $g_0 \in C^\infty(S^1)$ and it propagates in time.
\end{remark}

% Note that our solutions constructed in the above theorem can be compactly supported and uniformly Lipschitz continuous in space. 

We make a simple but powerful observation. Assume that $\Theta \in C^{1,\alpha}(\mathbb{R}^2)$ and $m$-fold symmetric, where $m \ge 3$. Then the stream function $\Psi = |\nabla|^{-1}\Theta$ is again $m$-fold symmetric and belongs to $C^{2,\alpha}(\mathbb{R}^2)$. The Taylor expansion at the origin therefore gives \begin{equation*}
\begin{split}
\Psi(x_1,x_2) = A_1 x_1 + A_2 x_2 + B_1 x_1x_2 + B_2 (x_1^2 + x_2^2) + O(|x|^{2+\alpha})~.
\end{split}
\end{equation*} However, the rotationally symmetry condition kills the coefficients $A_1, A_2$, and $B_1$. Moreover, the $B_2$ term should vanish by the odd symmetry assumption. In particular, by taking $\nabla^\perp$, $u(x_1,x_2) \le C|x|^{1+\alpha}$. Moreover, $\Theta \in C^{1,\alpha}(\mathbb{R}^2)$ with symmetry forces $\nabla\Theta(0) = 0$ and indeed $|\Theta(x)| \le C|x|^{1+\alpha}$ holds as well.

Along these lines, we have the following simple calculation, which will be used in the $C^\alpha$-estimate: \begin{lemma}
	Assume that $f \in \mathring{C}^{\alpha}$ and $g \in C^\alpha$ with $g(0) = 0$. Then we have the product rule \begin{equation*}
	\begin{split}
	\V fg \V_{C^\alpha} \le C\V f \V_{\mathring{C}^{\alpha}} \V g \V_{C^\alpha}~. 
	\end{split}
	\end{equation*}
\end{lemma}
\begin{proof}
	It is an immediate consequence of the definition of the space $\mathring{C}^{\alpha}$. Take $0 < |x| \le |x'|$ and take \begin{equation*}
	\begin{split}
	\frac{|f(x)g(x) -f(x')g(x')|}{|x-x'|^\alpha} \le |f(x)| \frac{|g(x)-g(x')|}{|x-x'|^\alpha} + \frac{|g(x')|}{|x'|^\alpha} \left( \frac{| |x|^\alpha f(x) - |x'|^\alpha f(x')|}{|x-x'|^\alpha} + |f(x)|   \right)~,
	\end{split}
	\end{equation*} which establishes the desired bound. 
\end{proof}

\begin{proof}[Proof of local well-posedness]

	We simply write $\Theta(t)  = \Theta^{2D}(t) + \Theta^{1D}(t)$, where by definition $\Theta^{1D}(t) = r g(t)$ where $g$ is the unique local solution in $C([0,T);C^{2,\alpha}(S^1))$ of the system \eqref{eq:alpha_SQG_homog} with initial data $g_0$. Then we denote the associated velocity of $\Theta^{1D}(t)$ by $u^{1D}(t)$, and note that $\nabla u^{1D}, \nabla\Theta^{1D}$ belongs to $\mathring{C}^{1,\alpha}$. Then, the system for the remaining term $\Theta^{2D}:=\Theta - \Theta^{1D}$ is forced to be \begin{equation*}
	\begin{split}
	\partial_t \Theta^{2D}(t) + ( u^{1D}(t)  +  u^{2D}(t) )\cdot\nabla \Theta^{2D}(t)+ u^{2D}(t) \cdot\nabla\Theta^{1D}(t) = 0 ~,
	\end{split}
	\end{equation*} where $u^{2D}$ is the associated velocity lying in $C^{1,\alpha}(\mathbb{R}^2)$. Taking the gradient, \begin{equation}\label{eq:gradient_system_SQGLWP}
	\begin{split}
	&\partial_t \nabla\Theta^{2D}(t) + \left(( u^{1D}(t)  +  u^{2D}(t) )\cdot\nabla\right) \nabla \Theta^{2D}(t) \\
	&\qquad + (\nabla u^{1D}(t)  + \nabla u^{2D}(t) )\cdot\nabla \Theta^{2D}(t) + u^{2D}(t)\cdot\nabla^2\Theta^{1D}(t)+ \nabla u^{2D}(t) \cdot\nabla\Theta^{1D}(t) = 0 ~.
	\end{split}
	\end{equation} We just need an $L^\infty $ and an $C^\alpha$ estimate. We compose with the flow map generated by the velocity $u^{1D}+u^{2D}$ (let us suppress from writing out the composition with the flow everywhere), and then take absolute values to get \begin{equation*}
	\begin{split}
	&\left| \frac{d}{dt}  \nabla \Theta^{2D}(x)  \right| \le C\left( \V \nabla u^{1D}\V_{L^\infty}+\V\nabla  u^{2D}\V_{L^\infty} \right) \V \nabla\Theta^{2D}\V_{L^\infty} \\
	&\qquad\qquad\qquad\quad + \V  \frac{u^{2D}(x)}{|x|} \V_{L^\infty} \cdot  \V|x| \nabla^2\Theta^{1D}(x)\V_{L^\infty} + \V \nabla u^{2D}\V_{L^\infty} \V \nabla \Theta^{1D} \V_{L^\infty}~,
	\end{split}
	\end{equation*} and then we may use the bound $|x| |\nabla^2\Theta^{1D}(x)| \le C(|g''|+|g'|)$ which comes from the radial homogeneity of $\Theta^{1D}$. This establishes the $L^\infty$ bound.
	
	Next, taking the $C^\alpha$ norm, we have \begin{equation*}
	\begin{split}
	\left| \frac{d}{dt} \V \nabla\Theta^{2D}  \V_{C^\alpha} \right| &\le \V \nabla u^{2D} \V_{L^\infty} \V \nabla \Theta^{2D} \V_{C^\alpha} + \V \nabla u^{2D} \V_{C^\alpha} \V \nabla \Theta^{2D} \V_{L^\infty} + \V \nabla u^{1D} \cdot \nabla\Theta^{2D} \V_{C^\alpha} \\
	&\qquad + \V u^{2D} \cdot\nabla^2\Theta^{1D}\V_{C^\alpha} + \V  \nabla u^{2D}(t) \cdot\nabla\Theta^{1D} \V_{C^\alpha}~.
	\end{split}
	\end{equation*} To begin with, from the classical singular integral bound, we have \begin{equation*}
	\begin{split}
	\V \nabla u^{2D} \V_{C^\alpha}  \le C \V \nabla \Theta^{2D} \V_{C^\alpha}~,\qquad \V \nabla u^{2D} \V_{L^\infty}  \le C\V \nabla \Theta^{2D} \V_{L^\infty} \left(1 + \ln \left(  1+ \frac{\V \nabla \Theta^{2D} \V_{C^\alpha}}{\V \nabla \Theta^{2D} \V_{L^\infty}}  \right)  \right)~.
	\end{split}
	\end{equation*}  We estimate the remaining product terms using the above lemma. First, we have \begin{equation*}
	\begin{split}
	 \V \nabla u^{1D} \cdot \nabla\Theta^{2D} \V_{C^\alpha} \le \V \nabla u^{1D}\V_{\mathring{C}^\alpha} \V \nabla\Theta^{2D} \V_{C^\alpha} \le C \V \partial_\theta {g}\V_{C^\alpha(S^1)}\V \nabla\Theta^{2D} \V_{C^\alpha}
	\end{split}
	\end{equation*} thanks to the fact that $|\nabla\Theta^{2D}(x)|\le C|x|^\alpha$. Then, similarly we estimate \begin{equation*}
	\begin{split}
	\V \nabla u^{2D}(t) \cdot\nabla\Theta^{1D} \V_{C^\alpha} \le \V \nabla u^{2D}\V_{C^\alpha} \V \Theta^{1D}\V_{\mathring{C}^\alpha}  \le \V \nabla u^{2D}\V_{C^\alpha} \V\partial_\theta {g}\V_{C^\alpha(S^1)}
	\end{split}
	\end{equation*} and regarding the last term $\V u^{2D} \cdot\nabla^2\Theta^{1D}\V_{C^\alpha}$, one can explicitly compute \begin{equation*}
	\begin{split}
	u^{2D}\cdot \nabla^2 \left(r g(\theta) \right) &=  \frac{u^{2D}}{|x|} \cdot \Big(  (\nabla x - |x|^{-2} x\otimes x)g(\theta) + \\
	&\qquad  (-|x|^{-2} x\otimes x^\perp - \nabla x^\perp - |x|^{-2} x^\perp \otimes x  )g'(\theta) + |x|^{-2} x^\perp \otimes x^\perp g''(\theta)  \Big)
	\end{split}
	\end{equation*} and the expression in the large parentheses belongs to $\mathring{C}^{0,\alpha}$ simply because $g'' \in C^\alpha(S^1)$, and $u^{2D}/|x|$ belongs to $C^\alpha(\mathbb{R}^2)$ with $|u^{2D}/|x|| \le C|x|^\alpha$. Indeed, to see that $u^{2D}/|x|$ belongs to $C^{\alpha}(\mathbb{R}^2)$ we take two points $x,y$ with $|y|<|x|$
	\[\Big|\frac{u^{2D}(x)}{|x|}-\frac{u^{2D}(y)}{|y|}\Big|\leq \frac{|u^{2D}(x)-u^{2D}(y)|}{|x|}+|u^{2D}(y)|\frac{|x-y|}{|x||y|}\leq |\nabla u^{2D}(\zeta_{x,y})|\frac{|x-y|}{|x|}+|y|^{\alpha}\frac{|x-y|}{|x|}\lesssim |x-y|^\alpha\] where $|\zeta_{x,y}|\leq |x|$ lies on the line connecting $x$ and $y$ and we have used $|u(y)|\lesssim |y|^{1+\alpha}$ and $|\nabla u^{2D}(\zeta_{x,y})|\lesssim |\zeta_{x,y}|^{\alpha}.$  
	
	Therefore we may close the estimates in terms of $\V \nabla\Theta\V_{C^{\alpha}(\mathbb{R}^2)}$ given finiteness of $\V {g} \V_{C^{2,\alpha}(S^1)}$. This gives the desired local well-posedness statement. To establish the blow-up criterion, assume that we are given \begin{equation*}
	\begin{split}
	\int_0^{T^*} \V \nabla \Theta^{2D}(t) \V_{L^\infty} + \V  \partial_\theta {g}(t,\theta) \V_{L^\infty} dt < + \infty~.
	\end{split}
	\end{equation*}  Then, to begin with, this gives finiteness \begin{equation*}
	\begin{split}
	\sup_{t \in [0,T^*]} \V {g}\V_{C^{2,\alpha}(S^1)} \le C < +\infty
	\end{split}
	\end{equation*} and in turn, this guarantees that \begin{equation*}
	\begin{split}
	\sup_{t \in [0,T^*]} \V \nabla\Theta^{2D}\V_{C^{1,\alpha}(\mathbb{R}^2)} \le C < \infty
	\end{split}
	\end{equation*} since the a priori estimate is only log-linear in the norm $\V \nabla\Theta^{2D}\V_{C^{1,\alpha}}$. This concludes the first part of the theorem. 
\end{proof}

\begin{proof}[Proof of conditional blow-up]

	For the sake of contradiction assume that the solution $\Theta = \Theta^{1D} + \Theta^{2D}$ stays smooth up to time $T^*$, which implies in particular that \begin{equation*}
	\begin{split}
	\sup_{t \in [0,T^*]} \V \nabla\Theta(t)\V_{L^\infty} \le C < +\infty~.
	\end{split}
	\end{equation*}
	
	From the blow-up criterion for the $1D$ system, we know that there is an increasing sequence of time moments $t_k \rightarrow T^*$ for which $\V \partial_\theta {g}(t_k)  \V_{L^\infty} \rightarrow +\infty$. Next, note that for any $t < T^*$, \begin{equation*}
	\begin{split}
	\limsup_{x \rightarrow 0} |\nabla\Theta(t,x)| = \limsup_{x \rightarrow 0} | \nabla\Theta^{1D} (t,x)|
	\end{split}
	\end{equation*} simply because $|\nabla\Theta^{2D}(t,0)| \equiv  0$ and $\nabla\Theta^{2D}(t,x)$ is continuous in time and space. We then observe that \begin{equation*}
	\begin{split}
	|\nabla \Theta^{1D}|^2 = |\nabla r \cdot g(\theta) + r\nabla\theta \cdot \partial_\theta g(\theta)|^2 = |g|^2 + |\partial_\theta g|^2 \ge |\partial_\theta g|^2~.
	\end{split}
	\end{equation*}  From this, we have that for $t < T^*$ (since $\nabla\Theta^{1D}$ is smooth in space away from the origin) \begin{equation*}
	\begin{split}
\V \nabla\Theta(t) \V_{L^\infty} \ge \limsup_{x \rightarrow 0} | \nabla\Theta^{1D}(t,x)| \ge  \V \partial_\theta {g}(t)  \V_{L^\infty}~.
	\end{split}
	\end{equation*} However, taking the $\limsup$ in time along the sequence $\{ t_k\}$, we obtain a contradiction. 
\end{proof}

\section{Further directions}

\subsection{Consequences of this work}
\subsubsection*{Global regularity for vortex patches with corners} 
An interesting question is whether global well-posedness can be proven for vortex patches with corners in $2D$. Our analysis in this paper actually covers vortex patches with corners but where the vorticity is constant along straight lines emanating from the origin. In a forthcoming work \cite{EJ2}, which can be seen as a companion to this paper, we prove global well-posedness for more general vortex patches with corners. 
\subsubsection*{3D Euler} 
As can be expected, it is possible to prove local well-posedness for bounded vorticities even in the 3D case under a suitable symmetry condition using the $\mathring{C}^{0,\alpha}$ spaces we defined above. This case is discussed in the Ph.D thesis of the second author \cite{Jthesis}. 
\subsection{Open Questions}
We close this paper by mentioning a few interesting open questions. 
\subsubsection*{V-States with Corners}
In the past few years, much work has been done on the construction of time-periodic solutions to the $2D$ Euler equation and, more generally, to active scalar equations.  Indeed, a V-state is defined to be a vortex patch solution to the $2D$ Euler equation where the vortex patch simply rotates for all time. The general method for proving existence of such V-states, starting with the work of Burbea \cite{Bu}, is to bifurcate from the base V-state, $\chi_{B_1(0)}$ which is the characteristic function of $B_1(0)$, which is a stationary solution of the $2D$ Euler equation. These bifurcation solutions have been shown to be smooth or even analytic in some cases. Hmidi, Mateau and Verdera \cite{HMV}, for example, proved that if a V-state is close in the $C^2$ norm to the disk, then it actually has smooth boundary. In this paper, we proved the existence of V-states with corners whose support is unbounded. Moreover, it can easily be checked that these ``unbounded" V-states with corners actually lead to V-states with corners when the problem is posed on the unit disk \cite{EJ2}. What all this means is that V-states with corners \emph{do} exist when the domain is just the unit disk or when the V-states are allowed to be unbounded. A natural question then arises: can one ``cut-off" the unbounded V-states constructed in this paper to construct bounded V-states which are smooth except at the origin where there is a corner?    
\subsubsection*{Vortex Patches with Corners with Quasi-Periodic motion}
Above, we constructed vortex sheets emanating from the origin which exhibit pendulum-like motion which is quasi-periodic in time. It is not even clear whether there exist vortex patches with quasi-periodic behavior of the type we constructed in this paper (which are constant along lines emanating from the origin). If such solutions do exist in the case of unbounded support, can they be made bounded?
\subsubsection*{Finite-time singularities for the SQG equation with Lipschitz data.}
Prove the existence of finite time singularities for homogeneous solutions of the SQG. As stated above, this would imply the formation of finite time singularities for the SQG equation in the class of Lipschitz compactly supported data.

\section*{Acknowledgments}

The authors would like to thank the anonymous referee for very helpful comments. T. M. Elgindi was partially supported by NSF Grants DMS-1402357 and DMS-1817134 during the completion of this work.

\bibliographystyle{plain}
\bibliography{BRS}

\begin{thebibliography}{10}

\bibitem{AKLN}
David~M. Ambrose, James~P. Kelliher, Milton~C. Lopes~Filho, and Helena~J.
  Nussenzveig~Lopes.
\newblock Serfati solutions to the 2{D} {E}uler equations on exterior domains.
\newblock {\em J. Differential Equations}, 259(9):4509--4560, 2015.

\bibitem{BC}
H.~Bahouri and J.-Y. Chemin.
\newblock \'{E}quations de transport relatives \'a\ des champs de vecteurs
  non-lipschitziens et m\'ecanique des fluides.
\newblock {\em Arch. Rational Mech. Anal.}, 127(2):159--181, 1994.

\bibitem{BLM}
Gregory~R. Baker, Xiao Li, and Anne~C. Morlet.
\newblock Analytic structure of two {$1$}{D}-transport equations with nonlocal
  fluxes.
\newblock {\em Phys. D}, 91(4):349--375, 1996.

\bibitem{BKM}
J.~T. Beale, T.~Kato, and A.~Majda.
\newblock Remarks on the breakdown of smooth solutions for the {$3$}-{D}
  {E}uler equations.
\newblock {\em Comm. Math. Phys.}, 94(1):61--66, 1984.

\bibitem{BMP}
D.~Benedetto, C.~Marchioro, and M.~Pulvirenti.
\newblock On the {E}uler flow in {${\bf R}^2$}.
\newblock {\em Arch. Rational Mech. Anal.}, 123(4):377--386, 1993.

\bibitem{BH}
Fr{\'e}d{\'e}ric Bernicot and Taoufik Hmidi.
\newblock On the global well-posedness for {E}uler equations with unbounded
  vorticity.
\newblock {\em Dyn. Partial Differ. Equ.}, 12(2):127--155, 2015.

\bibitem{BK}
Fr{\'e}d{\'e}ric Bernicot and Sahbi Keraani.
\newblock On the global well-posedness of the 2{D} {E}uler equations for a
  large class of {Y}udovich type data.
\newblock {\em Ann. Sci. \'Ec. Norm. Sup\'er. (4)}, 47(3):559--576, 2014.

\bibitem{BeCo}
A.~L. Bertozzi and P.~Constantin.
\newblock Global regularity for vortex patches.
\newblock {\em Comm. Math. Phys.}, 152(1):19--28, 1993.

\bibitem{BL2}
Jean Bourgain and Dong Li.
\newblock Strong illposedness of the incompressible {E}uler equation in integer
  {$C^m$} spaces.
\newblock {\em Geom. Funct. Anal.}, 25(1):1--86, 2015.

\bibitem{Bu}
Jacob Burbea.
\newblock Motions of vortex patches.
\newblock {\em Lett. Math. Phys.}, 6(1):1--16, 1982.

\bibitem{CaCo}
A.~Castro and D.~C{\'o}rdoba.
\newblock Infinite energy solutions of the surface quasi-geostrophic equation.
\newblock {\em Adv. Math.}, 225(4):1820--1829, 2010.

\bibitem{C}
J.-Y. Chemin.
\newblock Persistance des structures g\'eom\'etriques li\'ees aux poches de
  tourbillon.
\newblock In {\em S\'eminaire sur les \'{E}quations aux {D}\'eriv\'ees
  {P}artielles, 1990--1991}, pages Exp.\ No.\ XIII, 11. \'Ecole Polytech.,
  Palaiseau, 1991.

\bibitem{CISY}
S.~Childress, G.~R. Ierley, E.~A. Spiegel, and W.~R. Young.
\newblock Blow-up of unsteady two-dimensional {E}uler and {N}avier-{S}tokes
  solutions having stagnation-point form.
\newblock {\em J. Fluid Mech.}, 203:1--22, 1989.

\bibitem{Con}
Peter Constantin.
\newblock The {E}uler equations and nonlocal conservative {R}iccati equations.
\newblock {\em Internat. Math. Res. Notices}, (9):455--465, 2000.

\bibitem{CMT2}
Peter Constantin, Andrew~J. Majda, and Esteban Tabak.
\newblock Formation of strong fronts in the {$2$}-{D} quasigeostrophic thermal
  active scalar.
\newblock {\em Nonlinearity}, 7(6):1495--1533, 1994.

\bibitem{CMT1}
Peter Constantin, Andrew~J. Majda, and Esteban~G. Tabak.
\newblock Singular front formation in a model for quasigeostrophic flow.
\newblock {\em Phys. Fluids}, 6(1):9--11, 1994.

\bibitem{CCF}
Antonio C{\'o}rdoba, Diego C{\'o}rdoba, and Marco~A. Fontelos.
\newblock Formation of singularities for a transport equation with nonlocal
  velocity.
\newblock {\em Ann. of Math. (2)}, 162(3):1377--1389, 2005.

\bibitem{CK}
Elaine Cozzi and James~P. Kelliher.
\newblock Well-posedness of the 2d euler equations when velocity grows at
  infinity.
\newblock {\em arXiv:1709.07422}, 2017.

\bibitem{DG1}
Salvatore De~Gregorio.
\newblock On a one-dimensional model for the three-dimensional vorticity
  equation.
\newblock {\em J. Statist. Phys.}, 59(5-6):1251--1263, 1990.

\bibitem{DG2}
Salvatore De~Gregorio.
\newblock A partial differential equation arising in a {$1$}{D} model for the
  {$3$}{D} vorticity equation.
\newblock {\em Math. Methods Appl. Sci.}, 19(15):1233--1255, 1996.

\bibitem{co2}
Francesco Di~Plinio and Roger Temam.
\newblock Grisvard's shift theorem near {$L^\infty$} and {Y}udovich theory on
  polygonal domains.
\newblock {\em SIAM J. Math. Anal.}, 47(1):159--178, 2015.

\bibitem{EJB}
T.~M. {Elgindi} and I.-J. {Jeong}.
\newblock {Finite-time Singularity Formation for Strong Solutions to the
  Boussinesq System}.
\newblock {\em ArXiv e-prints}, August 2017.

\bibitem{E2}
Tarek~M. Elgindi.
\newblock Propagation of {S}ingularities for the 2d incompressible {E}uler
  equations.
\newblock {\em In Preparation}.

\bibitem{E1}
Tarek~M. Elgindi.
\newblock Remarks on functions with bounded {L}aplacian.
\newblock {\em arXiv:1605.05266}, 2016.

\bibitem{EJ2}
Tarek~M. Elgindi and In-Jee Jeong.
\newblock On singular vortex patches.
\newblock {\em In Preparation}.

\bibitem{EM1}
Tarek~M. Elgindi and Nader Masmoudi.
\newblock Ill-posedness results in critical spaces for some equations arising
  in hydrodynamics.
\newblock {\em arXiv:1405.2478}, 2014.

\bibitem{EJill}
Tarek~Mohamed Elgindi and In-Jee Jeong.
\newblock Ill-posedness for the incompressible {E}uler equations in critical
  {S}obolev spaces.
\newblock {\em Ann. PDE}, 3(1):Art. 7, 19, 2017.

\bibitem{GFD}
J.~D. Gibbon, A.~S. Fokas, and C.~R. Doering.
\newblock Dynamically stretched vortices as solutions of the {$3$}{D}
  {N}avier-{S}tokes equations.
\newblock {\em Phys. D}, 132(4):497--510, 1999.

\bibitem{HMV}
Taoufik Hmidi, Joan Mateu, and Joan Verdera.
\newblock Boundary regularity of rotating vortex patches.
\newblock {\em Arch. Ration. Mech. Anal.}, 209(1):171--208, 2013.

\bibitem{IMY2}
Tsubasa Itoh, Hideyuki Miura, and Tsuyoshi Yoneda.
\newblock The growth of the vorticity gradient for the two-dimensional {E}uler
  flows on domains with corners.
\newblock {\em arXiv:1602.00815}, 2016.

\bibitem{IMY1}
Tsubasa Itoh, Hideyuki Miura, and Tsuyoshi Yoneda.
\newblock Remark on {S}ingle {E}xponential {B}ound of the {V}orticity
  {G}radient for the {T}wo-{D}imensional {E}uler {F}low {A}round a {C}orner.
\newblock {\em J. Math. Fluid Mech.}, 18(3):531--537, 2016.

\bibitem{Jthesis}
In-Jee Jeong.
\newblock {\em Dynamics of the incompressible {E}uler equations at critical
  regularity}.
\newblock PhD thesis, Princeton University, 2017.

\bibitem{K1}
James~P. Kelliher.
\newblock A characterization at infinity of bounded vorticity, bounded velocity
  solutions to the 2{D} {E}uler equations.
\newblock {\em Indiana Univ. Math. J.}, 64(6):1643--1666, 2015.

\bibitem{KS}
Alexander Kiselev and Vladimir {\v{S}}ver{\'a}k.
\newblock Small scale creation for solutions of the incompressible
  two-dimensional {E}uler equation.
\newblock {\em Ann. of Math. (2)}, 180(3):1205--1220, 2014.

\bibitem{co3}
Christophe Lacave.
\newblock Uniqueness for two-dimensional incompressible ideal flow on singular
  domains.
\newblock {\em SIAM J. Math. Anal.}, 47(2):1615--1664, 2015.

\bibitem{co1}
Christophe Lacave, Evelyne Miot, and Chao Wang.
\newblock Uniqueness for the two-dimensional {E}uler equations on domains with
  corners.
\newblock {\em Indiana Univ. Math. J.}, 63(6):1725--1756, 2014.

\bibitem{L}
Pierre-Louis Lions.
\newblock {\em Mathematical Topics in Fluid Mechanics: Volume 1: Incompressible
  Models}.
\newblock Oxford University Press, 1996.

\bibitem{LS1}
Xue Luo and Roman Shvydkoy.
\newblock 2{D} homogeneous solutions to the {E}uler equation.
\newblock {\em Comm. Partial Differential Equations}, 40(9):1666--1687, 2015.

\bibitem{LS2}
Xue Luo and Roman Shvydkoy.
\newblock Addendum: 2d homogeneous solutions to the euler equation.
\newblock {\em arXiv:1608.00061}, 2016.

\bibitem{MB}
Andrew~J. Majda and Andrea~L. Bertozzi.
\newblock {\em Vorticity and incompressible flow}, volume~27 of {\em Cambridge
  Texts in Applied Mathematics}.
\newblock Cambridge University Press, Cambridge, 2002.

\bibitem{MP}
Carlo Marchioro and Mario Pulvirenti.
\newblock {\em Mathematical theory of incompressible nonviscous fluids},
  volume~96 of {\em Applied Mathematical Sciences}.
\newblock Springer-Verlag, New York, 1994.

\bibitem{MR}
Piotr Bogus\l~aw Mucha and Walter~M. Rusin.
\newblock Zygmund spaces, inviscid limit and uniqueness of {E}uler flows.
\newblock {\em Comm. Math. Phys.}, 280(3):831--841, 2008.

\bibitem{OSW}
Hisashi Okamoto, Takashi Sakajo, and Marcus Wunsch.
\newblock On a generalization of the {C}onstantin-{L}ax-{M}ajda equation.
\newblock {\em Nonlinearity}, 21(10):2447--2461, 2008.

\bibitem{PJ}
Ian Proudman and Kathleen Johnson.
\newblock Boundary-layer growth near a rear stagnation point.
\newblock {\em J. Fluid Mech.}, 12:161--168, 1962.

\bibitem{SaWu}
Alejandro Sarria and Jiahong Wu.
\newblock Blowup in stagnation-point form solutions of the inviscid 2d
  {B}oussinesq equations.
\newblock {\em J. Differential Equations}, 259(8):3559--3576, 2015.

\bibitem{S1}
Philippe Serfati.
\newblock Solutions {$C^\infty$} en temps, {$n$}-{$\log$} {L}ipschitz born\'ees
  en espace et \'equation d'{E}uler.
\newblock {\em C. R. Acad. Sci. Paris S\'er. I Math.}, 320(5):555--558, 1995.

\bibitem{S2}
Philippe Serfati.
\newblock Structures holomorphes \`a faible r\'egularit\'e spatiale en
  m\'ecanique des fluides.
\newblock {\em J. Math. Pures Appl. (9)}, 74(2):95--104, 1995.

\bibitem{St}
J.~T. Stuart.
\newblock Nonlinear {E}uler partial differential equations: singularities in
  their solution.
\newblock In {\em Applied mathematics, fluid mechanics, astrophysics
  ({C}ambridge, {MA}, 1987)}, pages 81--95. World Sci. Publishing, Singapore,
  1988.

\bibitem{Ta}
Yasushi Taniuchi.
\newblock Uniformly local {$L^p$} estimate for 2-{D} vorticity equation and its
  application to {E}uler equations with initial vorticity in {${\bf bmo}$}.
\newblock {\em Comm. Math. Phys.}, 248(1):169--186, 2004.

\bibitem{TTY}
Yasushi Taniuchi, Tomoya Tashiro, and Tsuyoshi Yoneda.
\newblock On the two-dimensional {E}uler equations with spatially almost
  periodic initial data.
\newblock {\em J. Math. Fluid Mech.}, 12(4):594--612, 2010.

\bibitem{Vi1}
Misha Vishik.
\newblock Incompressible flows of an ideal fluid with vorticity in borderline
  spaces of {B}esov type.
\newblock {\em Ann. Sci. \'Ecole Norm. Sup. (4)}, 32(6):769--812, 1999.

\bibitem{Y1}
V.~I. Yudovich.
\newblock Non-stationary flows of an ideal incompressible fluid.
\newblock {\em Z. Vycisl. Mat. i Mat. Fiz.}, 3:1032--1066, 1963.

\bibitem{Y2}
V.~I. Yudovich.
\newblock Uniqueness theorem for the basic nonstationary problem in the
  dynamics of an ideal incompressible fluid.
\newblock {\em Math. Res. Lett.}, 2(1):27--38, 1995.

\end{thebibliography}
\end{document}